\documentclass{amsart}
\usepackage{graphicx}
\usepackage[margin=1.3in]{geometry}
\usepackage{amssymb}
\usepackage{amsthm}
\usepackage{tikz-cd}
\usepackage{amsmath}
\usepackage{quiver, multicol}
\usepackage{mathtools}
\usepackage{adjustbox}
\usepackage[dvipsnames]{xcolor}
\usepackage{todonotes}
\usepackage{bbold}
\usepackage{mathrsfs}
\usepackage{eucal}
\usepackage[%
    colorlinks,%
    bookmarksdepth=3,%
    citecolor=blue%
]{hyperref}

\usepackage[
backend=biber,
style=alphabetic,
sorting=nyt,
maxbibnames=99,
]{biblatex}
\usepackage{thmtools}
\usepackage{cleveref}
\usepackage{float}

\newtheorem{theorem}{Theorem}[section]
\newtheorem{prop}[theorem]{Proposition}
\newtheorem{lemma}[theorem]{Lemma}
\newtheorem{cor}[theorem]{Corollary}

\newtheorem*{theoremA}{Theorem A}
\newtheorem*{corollaryB}{Corollary B}
\newtheorem*{corollaryC}{Corollary C}
\newtheorem*{theoremD}{Theorem D}
\newtheorem*{corollaryE}{Corollary E}
\theoremstyle{definition}
\newtheorem{definition}[theorem]{Definition}
\theoremstyle{remark}
\newtheorem{remark}[theorem]{Remark}
\newtheorem{example}[theorem]{Example}
\newtheorem{hypothesis}[theorem]{Hypothesis}

\newcommand{\Z}{
\mathbb{Z}
}
\newcommand{\Q}{
\mathbb{Q}
}
\newcommand{\R}{
\mathbb{R}
}

\newcommand{\F}{
\mathbb{F}
}
\newcommand{\K}{
\mathsf{K}
}
\newcommand{\D}{
\mathsf{D}
}
\newcommand{\unit}{
\mathbb{1}
}
\newcommand{\cone}{
\operatorname{cone}
}
\newcommand{\fib}{
\operatorname{fib}
}

\newcommand{\id}{
\operatorname{id}
}
\newcommand{\Hom}{
\operatorname{Hom}
}
\newcommand{\Map}{
\operatorname{Map}
}
\newcommand{\End}{
\operatorname{End}
}

\newcommand{\hocolim}{
\operatorname{hocolim}
}

\newcommand{\Spec}{
\operatorname{Spec}
}
\newcommand{\Spc}{
\operatorname{Spc}
}
\newcommand{\Pic}{
\operatorname{Pic}
}

\newcommand{\supp}{
\operatorname{supp}
}

\newcommand{\ideal}[1]{
\langle #1 \rangle
}
\newcommand{\Loc}{
\operatorname{Loc}
}
\newcommand{\add}{
\mathsf{add}
}
\newcommand{\thick}{
\mathsf{thick}
}

\newcommand{\Top}{
\mathsf{Top}
}
\newcommand{\Set}{
\mathsf{Set}
}

\newcommand{\DM}{
\mathsf{DM}
}
\newcommand{\DTM}{
\mathsf{DTM}
}
\newcommand{\DAM}{
\mathsf{DAM}
}
\newcommand{\DATM}{
\mathsf{DATM}
}
\newcommand{\DQM}{
\mathsf{DQM}
}
\newcommand{\AM}{
\mathsf{AM}
}
\newcommand{\gm}{
\mathsf{gm}
}

\newcommand{\SH}{
\mathsf{SH}
}
\newcommand{\Sp}{
\mathsf{Sp}
}

\newcommand{\Shv}{
\mathsf{Shv}
}
\newcommand{\CAlg}{
\mathsf{CAlg}
}
\newcommand{\p}{
\mathfrak{p}
}
\newcommand{\et}{
\mathsf{\acute{e}t}
}
\newcommand{\per}{
\operatorname{per}
}
\newcommand{\im}{
\operatorname{im}
}

\newcommand{\Proj}{
\operatorname{Proj}
}
\newcommand{\Sper}{
\operatorname{Sper}
}
\DeclareMathOperator{\DMet}{
\DM_{\acute{e}t}
}

\title{Tate and Artin--Tate Motivic Tensor-Triangular Geometry}
\author{Jean Paul Schemeil and Fraser Sparks}
\address{School of Mathematical Sciences, University of Nottingham, Nottingham, NG7 2RD, UK}
\email{jeanpaul.schemeil@nottingham.ac.uk}
\email{fraser.sparks@nottingham.ac.uk}
\date{\today}

\begin{document}

\begin{abstract}
We compute the Balmer spectra of geometric Tate and Artin--Tate motives with $\F_p$-coefficients over any field whose characteristic differs from $p$. For Tate motives, we find the bigraded comparison map to be a homeomorphism to the homogeneous spectrum of the motivic cohomology of the base field. For Artin--Tate motives, we describe the spectrum as a gluing of the spectrum of Artin motives computed by Balmer--Gallauer with the punctured Tate spectrum. We determine all specialisations between these pieces and the resulting topology. Under the hypothesis that the spectrum of rational Tate motives is a singleton, we also compute the integral Tate spectrum after inverting the exponential characteristic of the base field. As a consequence of our results, we identify the punctured Tate spectrum as an open subspace of the spectrum of the category of geometric Voevodsky motives, and draw some conclusions concerning the latter. Finally, we include an appendix where we define generalised periodicity in the context of tensor-triangular geometry, and explore its applications in motivic settings.
\end{abstract}

\maketitle

\tableofcontents

\section{Introduction}
Motives are notoriously enigmatic objects, and motivic tensor-triangular geometry aims to classify them up to the best reasonable notion of equivalence. It does so via geometric means: one associates to $\DM_\gm(k;R)$ a spectral space
\[\DM_\gm(k;R) \longmapsto \Spc(\DM_\gm(k;R))\]
called its \emph{Balmer spectrum}---or simply \emph{spectrum}---whose topology governs the classification of well-behaved subcategories of Voevodsky's category of geometric motives $\DM_\gm(k;R)$ \cite{Voevodsky:TriangulatedMotives} (namely, the thick tensor-triangulated ideals \cite{Balmer:tt-geometry}). Its points are the prime thick tensor ideals of $\DM_\gm(k;R)$.

As such, the classification problem for motives can be recast into the following question:

\begin{center}
    \emph{What is $\Spc(\DM_\gm(k;R))$?}
\end{center}

We should emphasise that this question is not an easy one. For instance, Vishik's isotropic realisations \cite{Vishik:IsotropicMotives} produce a non-exhaustive family of $2^{\mathfrak c}$ points in the spectrum of $\DM_\gm(\R;\Z)$, where $\mathfrak c=2^{\aleph_0}$ is the cardinality of the continuum \cite[Example~5.14]{Vishik:MoravaKTheory}. A complete description of this space must therefore account for these points, identify the remaining ones, and determine both the specialisation relations among them and the topology.

A complementary approach is to study full tt-subcategories generated by appropriately chosen collections of motives. By \cite[Corollary~1.8]{Balmer:Conservative}, the inclusion of a full rigid tt-subcategory
\[\begin{tikzcd}
    \mathcal{K} \arrow[r, hookrightarrow] & \DM_{\gm}(k;R)
\end{tikzcd}\]
induces a continuous surjective restriction map
\[\begin{tikzcd}
    \Spc(\DM_{\gm}(k;R)) \arrow[r, twoheadrightarrow] & \Spc(\mathcal{K}).
\end{tikzcd}\]
Thus, computing $\Spc(\mathcal{K})$ not only classifies the thick tensor ideals of $\mathcal{K}$, but it also provides information about the spectrum of $\DM_\gm(k;R)$.

Tate, Artin and Artin--Tate motives were among the first examples studied from this perspective. Under suitable vanishing assumptions on motivic cohomology, Peter \cite{Peter:ArtinTate} computed the rational Tate and Artin--Tate spectra, covering in particular algebraic extensions both of $\Q$ and of finite fields. Later, Gallauer \cite{Gallauer:TateSpectrum} computed the spectrum of Tate motives with finite coefficients over algebraically closed fields, and with integral coefficients when these fields also satisfy the vanishing assumptions considered by Peter. Balmer--Gallauer extended these computations to Tate and Artin--Tate motives over real closed fields \cite{Balmer:ThreeRealArtinTates}, and later determined the spectrum of Artin motives over arbitrary fields \cite{balmer2025geometry, Balmer:Artin}.

Most of these computations relied on representation-theoretic models for the motivic categories in question. For Tate motives over algebraically closed fields, Gallauer's finite-coefficient computation uses a description of the category in terms of filtered vector spaces. Similarly, Balmer--Gallauer's computation of the spectrum of Artin--Tate motives with $\F_2$-coefficients over real closed fields is obtained by passing to the derived category of filtered $C_2$-representations with a suitable exact structure. Artin motives, in turn, admit a description in terms of permutation modules over the absolute Galois group of the base field. These descriptions facilitate the computation of the motivic tt-spectra in question via passage to the corresponding representation-theoretic category, where one can utilise algebraic methods.

Over more general fields, however, analogous descriptions of Tate and Artin--Tate motives with finite coefficients in terms of filtered Galois representations depend on additional conjectural hypotheses (whose validity is intimately related to the \emph{Koszulity conjecture} on the Milnor $K$-theory ring of the field $k$) \cite{Positselski:MixedArtinTate}. This motivates the development of methods that do not require such algebraic models. Recent work by the authors provided complete computations using tools internal to the motivic category. In \cite{Sparks:IsotropicDATM}, the second named author computed the spectra of isotropic Tate, Artin and Artin--Tate motives over flexible fields with $\F_2$-coefficients, using weight structures to establish tensor-nilpotence properties of morphisms. The first named author in \cite{schemeil2026motivic} computed the spectrum of motives of quadrics over real closed fields with $\F_2$-coefficients, using a jointly conservative family of functors obtained by base change and smashing localisations to reduce the problem to the spectra of cellular motivic categories, which are then studied using weight structure techniques inspired by the work of the second named author.

In this paper, we develop these motivic methods further to extend the computations of the Balmer spectra of Tate and Artin--Tate motives with $\F_p$-coefficients from algebraically closed and real closed fields to arbitrary fields of characteristic different from $p$. The success of our approach suggests that studying motivic tt-categories directly, without passing to a representation-theoretic presentation of them, can yield many advantages: indeed, when tackling these beasts head on, one is offered up a platter of deep geometric and arithmetic results pertinent to the field $k$, which do not necessarily admit analogues in the representation-theoretic world.

We now describe the main results, the principal ingredients of their proofs, and some consequences for the spectrum of all geometric motives.

\subsection{Tate motives}
One of the principal tools for relating tt-geometry to commutative algebra is the comparison map, initially defined by Balmer \cite{Balmer:Spectra}. Let $\mathcal{K}$ be an essentially small tt-category, and $u_1, \dots, u_r \in \Pic(\mathcal{K})$ a collection of $\otimes$-invertible objects. In nice enough cases, we can associate to this data the $\Z^r$-graded ring
\[R^*_{u_1, \dots, u_r}:=\bigoplus_{(n_1, \dots, n_r) \in \Z^r} \Hom_{\mathcal{K}}(\unit, u_1^{\otimes n_1} \otimes \dots \otimes u_r^{\otimes n_r}).\]
There is a continuous comparison map \cite[Appendix A by Ivo Dell'Ambrogio]{Gallauer:Periods}
\[\begin{tikzcd}
    \varrho^h:\Spc(\mathcal{K}) \arrow[r] & \Spec^h(R^*_{u_1, \dots, u_r})
\end{tikzcd}\]
defined by
\[\varrho^h(\mathcal{P})=\langle f \in R^*_{u_1, \dots , u_r} \mid f \text{ homogeneous, } \cone(f) \notin \mathcal{P} \rangle.\]
In general, this map need not be injective nor surjective. However, when it is a homeomorphism, $\Spc(\mathcal{K})$ is entirely described by the homogeneous spectrum of a multigraded ring.

We apply this construction to the category of geometric Tate motives
\[\DTM_\gm(k;\F_p)=\thick(\unit(n) \mid n \in \Z) \subseteq \DM_\gm(k;\F_p).\]
There are two distinguished $\otimes$-invertible objects in this category, namely the Tate twist $\unit(1)$ and the suspension of the unit $\unit[1]$. The corresponding bigraded ring is the motivic cohomology of the base field modulo $p$: 
\[H^{*,*}_{\mathcal{M}}(k;\F_p)=\bigoplus_{a,b \in \Z}\Hom(\unit, \unit(a)[b]).\]
Our first main result of this paper is that this invariant completely determines the spectrum of $\DTM_\gm(k;\F_p)$.

\begin{theoremA}
    Let $k$ be a field and $p \neq \operatorname{char}(k)$ a prime. The comparison map
    \[\begin{tikzcd}
        \varrho^h: \Spc(\DTM_\gm(k;\F_p)) \arrow[r] & \Spec^h(H^{*,*}_{\mathcal{M}}(k;\F_p))
    \end{tikzcd}\]
    is a homeomorphism.
\end{theoremA}

The target of the comparison map admits an explicit description. For $p$ odd, it is the two-point Sierpiński space. At the prime $2$, the Beilinson--Lichtenbaum conjecture (really, theorem \cite{voevodsky2011motivic}) gives
\[H_{\mathcal{M}}^{*,*}(k;\F_2) \simeq (K_*^M(k)/2)[\tau],\]
where $K^M_*(k)/2$ is the Milnor $K$-theory mod $2$, and $\tau$ is the unique non-zero class in $H_{\mathcal{M}}^{0,1}(k;\F_2)$. Thornton's computation \cite{thornton2016homogeneous} describes the homogeneous spectrum of $K^M_*(k)/2$ in terms of the space of orderings (or real spectrum) $\Sper(k)$ of $k$ equipped with the Harrison topology. Combining this description with Theorem~A gives the following explicit computations. Recall that, as is standard, whenever we draw spectral spaces, specialisations are depicted via non-horizontal lines between points: such a line indicates that the lower altitude point specialises to the higher altitude one.

\begin{corollaryB}
    If $p$ is an odd prime or $k$ is not formally real, $\Spc(\DTM_\gm(k;\F_p))$ is the Sierpiński space:
    \[\begin{tikzcd}[ampersand replacement=\&,sep=scriptsize]
	   {\ideal{0}} \\
	   \\
	   {\mathcal{P}_{\acute{e}t}}
	   \arrow[no head, from=1-1, to=3-1]
    \end{tikzcd}\]
    where $\mathcal{P}_{\acute{e}t}$ is the kernel of étale realisation restricted to Tate motives.

    If $p=2$ and $k$ is formally real, then $\Spc(\DTM_\gm(k;\F_2))$ is described as follows:
    
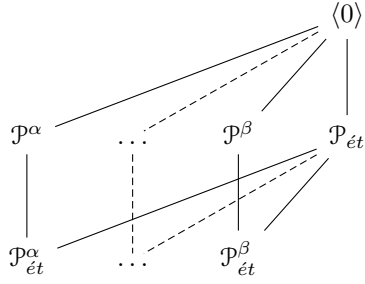
\[\begin{tikzcd}[ampersand replacement=\&,sep=scriptsize]
	   \&\&\& {\ideal{0}} \\
	   \\
	   {\mathcal{P}^{\alpha}} \& \dots \& {\mathcal{P}^{\beta}} \& {\mathcal{P}_{\acute{e}t}} \\
	   \\
	   {\mathcal{P}^{\alpha}_{\acute{e}t}} \& \dots \& {\mathcal{P}^{\beta}_{\acute{e}t}}
	   \arrow[no head, from=1-4, to=3-1]
	   \arrow[dashed, no head, from=1-4, to=3-2]
	   \arrow[no head, from=1-4, to=3-3]
	   \arrow[no head, from=1-4, to=3-4]
	   \arrow[no head, from=3-1, to=5-1]
	   \arrow[dashed, no head, from=3-2, to=5-2]
	   \arrow[no head, from=3-3, to=5-3]
	   \arrow[no head, from=3-4, to=5-1]
	   \arrow[dashed, no head, from=3-4, to=5-2]
	   \arrow[no head, from=3-4, to=5-3]
    \end{tikzcd}\]
    where $\alpha, \dots, \beta \in \Sper(k)$ range over all orderings, and
    \[\mathcal{P}_{\acute{e}t}=\langle \cone(\tau) \rangle.\]
    A subset $Z$ is closed if and only if it is specialisation closed and 
    \[W_1(Z)=\{\alpha \in \Sper(k) \mid \mathcal{P}^{\alpha} \in Z\}, \ W_2(Z)=\{\alpha \in \Sper(k) \mid \mathcal{P}^{\alpha}_{\acute{e}t} \in Z\}\]
    are both closed subsets of $\Sper(k)$.
\end{corollaryB}

As a consequence, we observe that at odd primes the Tate spectrum is independent of the base field, whilst at the prime $2$ the additional points and their topology are controlled by the orderings of this field. These computations recover the previously known spectra over algebraically closed and real closed fields of Gallauer and Balmer--Gallauer.

We briefly outline the strategy of our proof. Suppose first that $k$ contains a primitive $p^{th}$ root of unity. A choice of such determines a class $\tau \in H^{0,1}_{\mathcal{M}}(k;\F_p)$ and an isomorphism
\[H^{*,*}_{\mathcal{M}}(k;\F_p) \simeq (K_*^M(k)/p)[\tau].\]
The homogeneous spectrum of $(K_*^M(k)/p)[\tau]$ decomposes into the closed subset $V(\tau)$ and its open complement $D(\tau)$, each
homeomorphic to $\Spec^h(K_*^M(k)/p)$. This suggests studying the corresponding complementary pieces of the Tate spectrum
\[\supp(\cone(\tau)), \ U(\cone(\tau)) \subseteq \Spc(\DTM_\gm(k;\F_p))\] 
separately. 

To isolate the closed piece $\supp(\cone(\tau))$, we construct a bounded $\otimes$-compatible weight structure on $\DTM_\gm(k;\F_p)$ whose heart is 
\[\mathcal{H}^\mathsf{T} \coloneqq \add(\unit(n)[n] \mid n \in \Z).\]
We call this the \emph{Milnor weight structure}. The associated weight complex functor
\[\begin{tikzcd}
    t:\DTM_\gm(k;\F_p) \arrow[r] & \K^b(\mathcal{H}^\mathsf{T})
\end{tikzcd}\]
kills $\tau$ and, in contrast to the Chow weight complex functor in general, it also detects the tensor-nilpotence of morphisms on $\cone(\tau)$. In particular, this implies that 
\[\im(\Spc(t))= \supp(\cone(\tau)).\]
To compute the spectrum of $\K^b(\mathcal{H}^\mathsf{T})$, we prove that the heart $\mathcal{H}^\mathsf{T}$ is equivalent to the category of finite rank free graded modules over $K^M_*(k)/p$. Dell'Ambrogio--Stevenson's computation \cite{dell2014even} then identifies $\Spc(\K^b(\mathcal{H}^\mathsf{T}))$ via the graded comparison map with $\Spec^h(K^M_*(k)/p)$. We use this description to show that $\Spc(t)$ is a homeomorphism onto its image, and that $\varrho^h$ restricts to a homeomorphism
\[\supp(\cone(\tau)) \cong V(\tau).\]

For the open piece $U(\cone(\tau))$, we prove that, up to idempotent completion, the localisation at $\tau$ identifies geometric
Tate motives with rigid étale Tate motives. In fact, we will establish
the corresponding statement for all geometric motives:
\[(\DM_\gm(k;\F_p)/\ideal{\cone(\tau)})^\natural \simeq \DM_\et^\vee(k;\F_p).\]
Restricting to Tate motives and applying Gallauer's
computation of the spectrum of étale motives \cite{gallauer2021note}, we identify $U(\cone(\tau))$ with $\Spec^h(K_*^M(k)/p)$. Under this identification, $\varrho^h$ restricts to a homeomorphism
\[U(\cone(\tau)) \cong D(\tau).\]
To prove that $\varrho^h$ is a homeomorphism on $\Spc(\DTM_\gm(k;\F_p))$, we determine the specialisations between the closed and open piece using generators of the prime tt-ideals and, when $p=2$ and $k$ is formally real, base change to the real closures. Finally, we remove the root of unity assumption
by showing that adjoining a primitive $p^{th}$ root of unity does not change  the Tate spectrum nor the homogeneous spectrum of the motivic cohomology.

\subsection{Integral Tate motives}
The finite-coefficient computation also determines the spectrum of integral Tate motives for fields $k$ such that 
\[\Spc(\DTM_\gm(k;\Q))=\{\ideal{0}\}.\]
Conjecturally, this should be true for all fields, and it is known to be the case for those that satisfy the vanishing hypothesis in rational motivic cohomology considered by Peter \cite{Peter:ArtinTate} and Gallauer \cite{Gallauer:TateSpectrum}.

\begin{corollaryC}
    Let $k$ be a field such that $\Spc(\DTM_\gm(k;\Q))=\{\ideal{0}\}$, and let $e$ be its exponential characteristic. If $k$ is not formally real, $\Spc(\DTM_\gm(k; \Z[1/e]))$ is described as follows:
    
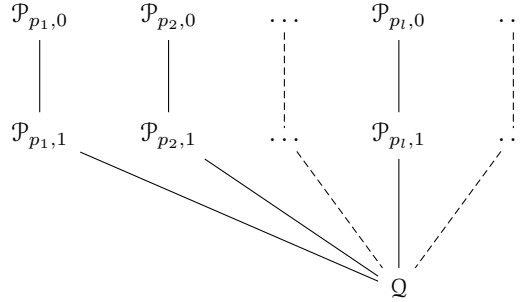
\[\begin{tikzcd}[ampersand replacement=\&,sep=scriptsize]
	   {\mathcal{P}_{p_1,0}} \& {\mathcal{P}_{p_2,0}} \& \dots \& {\mathcal{P}_{p_l,0}} \& \dots \\
	   \\
	   {\mathcal{P}_{p_1,1}} \& {\mathcal{P}_{p_2,1}} \& \dots \& {\mathcal{P}_{p_l,1}} \& \dots \\
	   \\
	   \\
	   \&\&\& {\mathcal{Q}}
	   \arrow[no head, from=1-1, to=3-1]
	   \arrow[no head, from=1-2, to=3-2]
	   \arrow[dashed, no head, from=1-3, to=3-3]
	   \arrow[no head, from=1-4, to=3-4]
	   \arrow[dashed, no head, from=1-5, to=3-5]
	   \arrow[no head, from=3-1, to=6-4]
	   \arrow[no head, from=3-2, to=6-4]
	   \arrow[dashed, no head, from=3-3, to=6-4]
	   \arrow[no head, from=3-4, to=6-4]
	   \arrow[dashed, no head, from=3-5, to=6-4]
    \end{tikzcd}\]
    where $p_i$ range over all prime numbers different from $e$. Moreover, a proper subset $Z$ is closed if and only if it is finite and specialisation closed.
    
    If $k$ is formally real, $\Spc(\DTM_\gm(k; \Z))$ is described as follows:
    
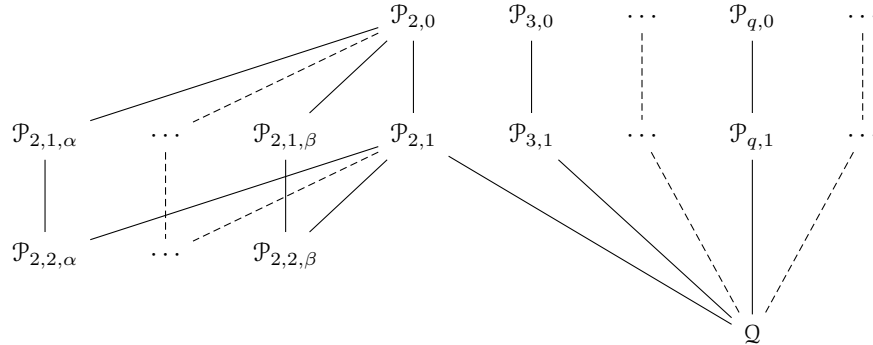
\[\begin{tikzcd}[ampersand replacement=\&,sep=scriptsize]
	   \&\&\& {\mathcal{P}_{2,0}} \& {\mathcal{P}_{3,0}} \& \cdots \& {\mathcal{P}_{q,0}} \& \cdots \\
	   \\
	   {\mathcal{P}_{2,1,\alpha}} \& \cdots \& {\mathcal{P}_{2,1, \beta}} \& {\mathcal{P}_{2,1}} \& {\mathcal{P}_{3,1}} \& \cdots \& {\mathcal{P}_{q,1}} \& \cdots \\
	   \\
	   {\mathcal{P}_{2,2,\alpha}} \& \cdots \& {\mathcal{P}_{2,2,\beta}} \\
	   \&\&\&\&\&\& {\mathcal{Q}}
	   \arrow[no head, from=1-4, to=3-1]
	   \arrow[dashed, no head, from=1-4, to=3-2]
	   \arrow[no head, from=1-4, to=3-3]
	   \arrow[no head, from=1-4, to=3-4]
	   \arrow[no head, from=1-5, to=3-5]
	   \arrow[dashed, no head, from=1-6, to=3-6]
	   \arrow[no head, from=1-7, to=3-7]
	   \arrow[dashed, no head, from=1-8, to=3-8]
	   \arrow[no head, from=3-1, to=5-1]
	   \arrow[dashed, no head, from=3-2, to=5-2]
	   \arrow[no head, from=3-3, to=5-3]
	   \arrow[no head, from=3-4, to=5-1]
	   \arrow[dashed, no head, from=3-4, to=5-2]
	   \arrow[no head, from=3-4, to=5-3]
	   \arrow[no head, from=3-4, to=6-7]
	   \arrow[no head, from=3-5, to=6-7]
	   \arrow[dashed, no head, from=3-6, to=6-7]
	   \arrow[no head, from=3-7, to=6-7]
	   \arrow[dashed, no head, from=3-8, to=6-7]
    \end{tikzcd}\]
    where $q$ ranges over all prime numbers. Moreover, a specialisation closed subset $Z$ is closed if and only if it satisfies one of the following conditions:
    \begin{enumerate}
    \item If $\mathcal{Q} \in Z$, then
    \[Z=\overline{\{\mathcal{Q}\}}\cup W,\]
    where $W$ is a closed subset of 
    \[\supp(\cone(2)) \cong \Spc(\DTM_\gm(k;\F_2)).\] 
    \item If $\mathcal{Q} \not \in Z$, then
    \[Z=Z_{p_1}\cup\cdots\cup Z_{p_n},\]
    where $p_1,\dots,p_n$ are distinct primes and 
    \[Z_{p_j}\subseteq \supp(\cone(p_j))\cong \Spc(\DTM_\gm(k;\F_{p_j})), \quad j=1,\dots,n\]
    are closed.
    \end{enumerate}
\end{corollaryC}

\subsection{Artin--Tate motives}
Now we turn to the category of geometric Artin--Tate motives
\[\DATM_\gm(k;\F_p)=\mathsf{thick}\left(M(E)(n) \mid E/k \text{ finite separable, } n \in \Z \right).\]
It contains as full subcategories the categories of geometric Tate motives and Artin motives; in particular, the inclusions induce surjective restrictions
\[\begin{tikzcd}
    \Spc(\DTM_\gm(k;\F_p)) & \arrow[l, twoheadrightarrow]  \Spc(\DATM_\gm(k;\F_p)) \arrow[r, twoheadrightarrow] & \Spc(\DAM_\gm(k;\F_p)).
\end{tikzcd}\]
By Theorem A and \cite{Balmer:Artin}, the two outer spectra are now known over arbitrary fields of characteristic different from $p$. A natural question is whether or not together these spaces contain enough information to determine the spectrum of Artin--Tate motives and, if so, how this information fits together. Our next result gives an affirmative answer to this question, and describes the interaction: the Artin spectrum can be identified with a closed subspace of $\Spc(\DATM_\gm(k;\F_p))$, and its open complement is homeomorphic to the punctured Tate spectrum $\Spc(\DTM_\gm(k;\F_p)) \setminus \{\ideal{0}\}$. The specialisations between these pieces are determined by base changing to the algebraic and, at the prime $2$, real closures (if they exist) of the base field $k$.

To state the main result, let $\overline{k}$ be an algebraic closure of $k$, and for each ordering $\alpha \in \Sper(k)$, let $k_\alpha$ be the corresponding real closure. We shall write $a_E$ for the restriction to $\Spc(\DATM_\gm(k;\F_p))$ of the isotropic point of \cite{Vishik:MoravaKTheory} associated to a field extension $E/k$, and we write $b_\alpha$ for the image of the generic point of $\Spc(\DAM_\gm(k_\alpha;\F_2))$ (see \cite[Theorem 12.3]{Balmer:ThreeRealArtinTates}) in $\Spc(\DAM_{\gm}(k;\F_2))$ under the map induced by base change.

\begin{theoremD}
    Let $k$ be a field and $p \neq \operatorname{char}(k)$ be a prime. Then,
    \[\Spc(\DATM_\gm(k;\F_p))\]
    contains a closed subspace
    \[V \cong \Spc(\DAM_\gm(k;\F_p))\]
    whose open complement is 
    \[U \cong \Spc(\DTM_\gm(k;\F_p)) \setminus \{\ideal{0}\}.\]
    Under these identifications, the specialisations from points in $U$ to points in $V$ are given as follows:
    \begin{itemize}
        \item $\mathcal{P}_{\acute{e}t} \rightsquigarrow a_{\overline{k}}$;
        \item if $k$ is formally real and $p=2$, then for every $\alpha \in \Sper(k)$ we have
        \begin{enumerate}
            \item $\mathcal{P}^{\alpha} \rightsquigarrow a_{k_\alpha}$;
            \item $\mathcal{P}^{\alpha}_{\acute{e}t} \rightsquigarrow \mathcal{Q}$ if and only if $\mathcal{Q} \in \overline{\{b_\alpha\}}$.
        \end{enumerate}
    \end{itemize}
    A subset $Z \subseteq \Spc(\DATM_\gm(k;\F_p))$ is closed if and only if it is specialisation closed and its intersections with $U$ and $V$ are both closed in the respective subspaces.
\end{theoremD}

As a result, if $p$ is odd or $k$ is not formally real, $\Spc(\DATM_\gm(k;\F_p))$ is obtained from the spectrum of Artin motives by adjoining one open point, whose unique specialisation is the closed point $a_{\overline{k}}$, as shown in the following picture.

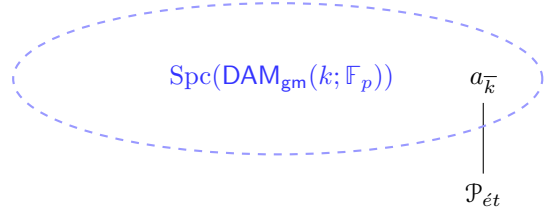
\begin{figure}[H]
\begin{tikzcd}[ampersand replacement=\&,sep=scriptsize,execute at end picture={\draw[dashed,thick,blue!40] (-0.6,0.75) ellipse (3.5cm and 1cm);}]
	\&\& {\textcolor{blue!80}{\Spc(\DAM_\gm(k;\F_p))}} \& {a_{\overline{k}}} \& \\
	\\
	\&\&\& {\mathcal{P}_{\acute{e}t}} \&
	\arrow[no head, from=1-4, to=3-4]
\end{tikzcd}
\end{figure}

If $p=2$ and $k$ is formally real, the specialisations of the additional points from the punctured Tate spectrum to the spectrum of Artin motives are determined by the real closures, and we obtain the following picture.

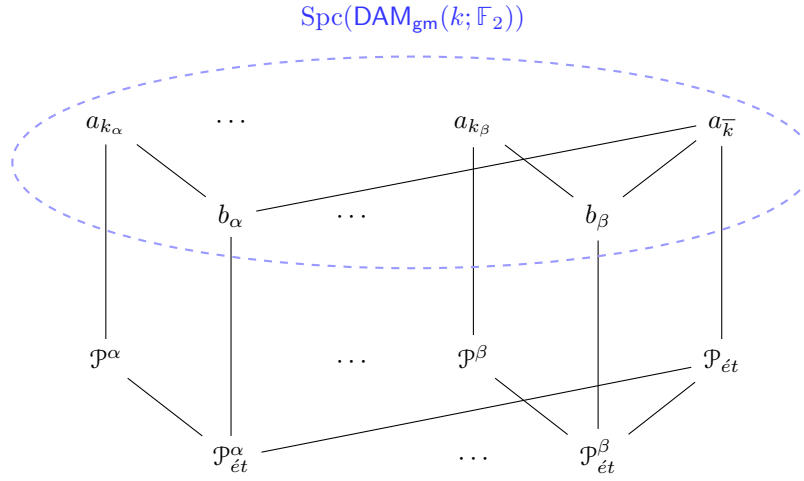
\begin{figure}[H]
\begin{tikzcd}[ampersand replacement=\&, execute at end picture={\draw[dashed,thick,blue!40] (0,1.3) ellipse (5.3cm and 1.4cm); \node[blue!80] at (0,3.2) {$\Spc(\DAM_\gm(k;\F_2))$};}]
	\&\&\& {} \&\& \\
	{a_{k_\alpha}} \& \cdots \&\& {a_{k_{\beta}}} \&\& {a_{\overline{k}}} \\
	\& {b_\alpha} \& \cdots \&\& {b_{\beta}} \\
	\\
	{\mathcal{P}^{\alpha}} \&\& \cdots \& {\mathcal{P}^{\beta}} \&\& {\mathcal{P}_{\acute{e}t}} \\
	\& {\mathcal{P}^{\alpha}_{\acute{e}t}} \&\& \cdots \& {\mathcal{P}^{\beta}_{\acute{e}t}}
	\arrow[no head, no head, from=2-1, to=3-2]
	\arrow[no head, from=2-1, to=5-1]
	\arrow[no head, no head, from=2-4, to=3-5]
	\arrow[no head, from=2-4, to=5-4]
	\arrow[no head, from=2-6, to=5-6]
	\arrow[no head, no head, from=3-2, to=2-6]
	\arrow[no head, from=3-2, to=6-2]
	\arrow[no head, no head, from=3-5, to=2-6]
	\arrow[no head, from=3-5, to=6-5]
	\arrow[no head, from=5-1, to=6-2]
	\arrow[no head, from=5-4, to=6-5]
	\arrow[no head, from=6-2, to=5-6]
	\arrow[no head, from=6-5, to=5-6]
\end{tikzcd}
\end{figure}

We note that the specialisations
\[b_\alpha \rightsquigarrow a_{k_\alpha}, \, b_{\alpha} \rightsquigarrow a_{\overline{k}}; \ \alpha \in \Sper(k),\]
within the spectrum of Artin motives may in principle factor through intermediate points that are not shown in the diagram.

We briefly discuss the main ingredients of the proof. Suppose first that $k$ contains a primitive $p^{th}$ root of unity. The first new input is the construction of a tt-functor
\[\begin{tikzcd}
    \Psi:\DATM_\gm(k;\F_p) \arrow[r] & \DAM_\gm(k;\F_p)
\end{tikzcd}\]
which is a retraction of the inclusion of Artin motives. To construct it, we extend the Milnor weight structure defined for Tate motives to $\DATM_\gm(k;\F_p)$, with heart
\[\mathcal{H}^{\mathsf{AT}} \coloneqq \add(M(E)(n)[n] \mid E/k \text{ finite separable, } n \in \Z)^\natural.\]
As in the Tate case, we show that the associated weight complex functor $t$ detects tensor-nilpotence on $\cone(\tau)$. We then construct a grading functor
\[\begin{tikzcd}
    \operatorname{gr}:\K^b(\mathcal{H}^{\mathsf{AT}}) \arrow[r] & \K^b(\AM_\gm(k;\F_p)) \simeq \DAM_\gm(k;\F_p)
\end{tikzcd}\]
which removes the data coming from the positive degrees from the Milnor $K$-theory. The composite $\Psi \coloneqq \operatorname{gr}\circ t$
restricts to the identity on Artin motives, and we prove that $\Psi$ detects tensor-nilpotence on $\cone(\tau) \otimes \cone(\rho)$. In particular, this allows us to identify 
\[V:=\supp(\cone(\tau) \otimes \cone(\rho)) \cong \Spc(\DAM_\gm(k;\F_p)).\]
Its open complement is
\[U:=U(\cone(\tau) \otimes \cone(\rho))=U(\cone(\tau))\cup U(\cone(\rho)).\]
We already know that étale realisation detects the open piece $U(\cone(\tau))$, which thanks to Gallauer's computation, can be identified with $\Spec^h(K^M_*(k)/p)$. To compute $U(\cone(\rho))$,
we use Bachmann's real étale comparison theorem \cite{Bachmann:Real} to show that, after inverting $\rho$, the inclusion
\[\begin{tikzcd}
    (\DTM_\gm(k;\F_p)[\rho^{-1}])^\natural \arrow[r, hookrightarrow] & (\DM_\gm(k;\F_p)[\rho^{-1}])^\natural
\end{tikzcd}\]
is an equivalence. As a consequence, the open $U(\cone(\rho))$ in $\Spc(\DATM_\gm(k;\F_p))$ can be identified with the corresponding open in $\Spc(\DTM_\gm(k;\F_p))$ (and, moreover, the same is true in $\Spc(\DM_\gm(k;\F_p))$). Combining everything together, we have identified
\[V \cong \Spc(\DAM_\gm(k;\F_p)), \ U \cong \Spc(\DTM_\gm(k;\F_p)) \setminus \{\ideal{0}\}.\]
We then determine the specialisation relations by base changing to algebraic and real closures of the base field, and prove that this information is sufficient to describe the topology. We conclude the section by removing the root of unity assumption.

\subsection{Some remarks about the spectrum of geometric motives}
The results established along the way allow us to return to the
original problem of understanding the spectrum of all geometric motives. Our identifications of the étale and real étale localisations show that the corresponding open regions in $\Spc(\DM_\gm(k;\F_p))$ are detected at the level of Tate motives. Together, they identify the punctured Tate spectrum as an open
subspace of $\Spc(\DM_{\gm}(k;\F_p))$.

\begin{corollaryE}
    Let $k$ be a field and let $p\neq \operatorname{char}(k)$ be a prime. There is an open embedding
    \[\begin{tikzcd}
        \Spc(\DTM_{\gm}(k;\F_p))\setminus\{\ideal{0}\} \arrow[r, hookrightarrow] & \Spc(\DM_{\gm}(k;\F_p)).
    \end{tikzcd}\]
    In particular, if $k$ has infinitely many orderings, then $\Spc(\DM_{\gm}(k;\F_2))$ is non-Noetherian.
\end{corollaryE}

\subsection{Generalised periodicity}
The localisations that we have encountered throughout the paper suggest a new invariant that can be associated to prime tt-ideals. Inverting $\tau$ identifies $\unit(1)$ with $\unit$, while inverting $\rho$ identifies $\unit(1)[1]$ with $\unit$. These periodicity phenomena involve both Tate twists and suspensions, and are not recorded by
periodicity with respect to suspension alone.

In Appendix A, we extend Gallauer's theory of periodicity
\cite{Gallauer:Periods} to arbitrary $\otimes$-invertible objects $u \in \Pic(\mathcal{K})$. For an essentially small tt-category $\mathcal{K}$ and a prime $\mathcal{P} \in\Spc(\mathcal{K})$, we define its \emph{generalised periodicity} to be the subgroup
\[\begin{tikzcd}
    \per_\mathcal{K}(\mathcal{P}) \coloneqq \ker(\Pic(\mathcal{K}) \arrow[r] & \Pic(\mathcal{K}/\mathcal{P}))
\end{tikzcd}\]
of $\Pic(\mathcal{K})$, that is, the kernel of the map on Picard groups induced by the Verdier quotient of $\mathcal{K}$ by $\mathcal{P}$. In particular, a $\otimes$-invertible object $u$ in $\mathcal{K}$ is a \emph{period} of $\mathcal{P}$ if it becomes isomorphic to the unit in $\mathcal{K}/\mathcal{P}$.
Generalised periodicity is compatible with specialisation, that is, if $\mathcal{P}\rightsquigarrow \mathcal{Q}$, then $\operatorname{per}_{\mathcal{K}}(\mathcal{Q}) \subseteq \operatorname{per}_{\mathcal{K}}(\mathcal{P})$; as a consequence, the notion provides obstructions to possible specialisation relations between tt-primes. The association
\[\mathcal{P} \longmapsto \operatorname{per}_\mathcal{K}(\mathcal{P})\]
defines a period map, and we see how to equip the set of subgroups of $\Pic(\mathcal{K})$ with a natural topology which makes this map continuous.

We apply our theory to the spectrum of Tate motives and the spectrum of motives of quadrics over real closed fields $\Spc(\DQM_\gm(F; \F_2))$ computed in \cite{schemeil2026motivic}. We will see that generalised periodicity distinguishes every point in $\Spc(\DQM_\gm(F; \F_p))$ at every prime $p$. We will see that the same is true for Tate motives when $p$ is an odd prime or $k$ admits at most one ordering. When $p=2$ and $k$ admits multiple orderings, we will see that generalised periodicity distinguishes the four types of points in Corollary~B, but does not distinguish different orderings.

\subsubsection*{Acknowledgements}
We are grateful to our advisor Alexander Vishik for his support, many valuable conversations, and helpful comments on early drafts of this document. We thank Martin Gallauer for useful discussions and comments on a preliminary version of this manuscript. We also thank Colin Ni for sharing with us his results on the tt-geometry of filtered permutation modules, which inspired us to consider the region $\supp(\cone(\tau)) \cap U(\cone(\rho))$ when computing $\supp(\cone(\tau)) \subseteq \Spc(\DATM_\gm(k;\F_2))$ in the formally real case. Finally, we thank Daniel Marlowe for discussions which helped improve the proof of \Cref{prop:spectralsheavescompactgeneration} and enhance our understanding of $\infty$-sheaves. This work was supported by the EPSRC Studentship [EP/W524402/1] Project Reference~2926287 for JPS, and the EPSRC Studentship [EP/W524402/1] Project Reference~2883661 for FS.

\subsubsection*{AI Disclosure}
No AI tools were used at any stage of the research or writing of this document.
\section{Preliminaries}
\subsection{The spectrum of a tt-category}
Introduced by Balmer \cite{Balmer:tt-geometry}, tensor-triangular geometry, oft abbreviated simply to tt-geometry, can be summarised as the study of tensor-triangulated (or `tt-')categories via geometric means, in an analogous fashion to how one studies commutative rings geometrically in scheme theory. One associates to a tt-category\footnote{That is, an essentially small triangulated category with a compatible symmetric monoidal structure.} $\mathcal{K}$ a topological space $\Spc(\mathcal{K})$ equipped with a notion of \emph{support} for objects---for each $a \in \mathcal{K}$, we have a closed subset $\supp(a) \subseteq \Spc(\mathcal{K})$. These subsets are subject to natural axioms, and the pair $(\Spc(\mathcal{K}),\supp(-))$ is universal with respect to these properties for $\mathcal{K}$. We call the space $\Spc(\mathcal{K})$ the \emph{Balmer spectrum} (or simply \emph{spectrum}) of $\mathcal{K}$.

Explicitly, $\Spc(\mathcal{K})$ is the set of \emph{prime tt-ideals}\footnote{A \emph{tt-ideal} of $\mathcal{K}$ is a full, thick, triangulated subcategory of $\mathcal{K}$ closed under tensoring with arbitrary objects of $\mathcal{K}$. A proper tt-ideal is \emph{prime} if whenever $a \otimes b$ lies in the tt-ideal, necessarily $a$ lies in it or $b$ does.} $\mathcal{P} \subseteq \mathcal{K}$, with the topology determined by the basis of closed subsets
\[\supp(a) \coloneqq \{\mathcal{P} \in \Spc(\mathcal{K}) \mid a \notin \mathcal{P}\}.\] 

The following result of Balmer demonstrates the importance of the spectrum when it comes to the classification problem for $\mathcal{K}$: it tells us that classifying the radical tt-ideals of $\mathcal{K}$ is equivalent to knowing its spectrum.

\begin{definition}\label{def:radicalttideal}
    A tt-ideal $\mathcal{I}$ is \emph{radical} if $a \in \mathcal{I}$ whenever $a^{\otimes n} \in \mathcal{I}$ for some positive integer $n$.
\end{definition}
\begin{definition}\label{def:thomasonsubset}
    A subset $V \subseteq \Spc(\mathcal{K})$ is \emph{Thomason} if it can be written
    \[V=\bigcup_{i \in I}V_i\]
    where each $V_i$ is closed with quasi-compact complement.
\end{definition}
With these definitions, we can state \cite[Theorem~4.10]{Balmer:tt-geometry}:
\begin{theorem}[Balmer]\label{theorem:classificationofttideals}
    There is an order-preserving bijection
    \[\{\text{radical $tt$-ideals of $\mathcal{K}$}\} \longleftrightarrow \{\text{Thomason subsets of $\Spc(\mathcal{K})$}\}\]
    given by sending a radical tt-ideal to its support:
    \[\mathcal{I} \longmapsto \supp(\mathcal{I}) \coloneqq \bigcup_{a \in \mathcal{I}}\supp(a)\]
\end{theorem}
\begin{remark}\label{remark:rigiditymeansallidealsareradical}
    In almost all cases we are interested in, the tt-category $\mathcal{K}$ in question is \emph{rigid}, meaning every object is strongly dualisable. In this case, \emph{every} tt-ideal is automatically radical, since a standard argument involving the unit-counit adjunction shows $a$ is a retract of $a \otimes a^\vee \otimes a$ (here $a^\vee$ denotes the dual of $a$).
\end{remark}
\begin{remark}\label{remark:spectralspace}
    The spectrum of a tt-category is always a \emph{spectral space}: these are the topological spaces homeomorphic to Zariski spectra of commutative rings. Indeed, \Cref{theorem:classificationofttideals} tells us that $\Spc(\mathcal{K})$ corresponds, under Stone duality, to the (Hochster dual of the) coherent frame of radical tt-ideals of $\mathcal{K}$.
\end{remark}

\begin{remark}\label{remark:functorialityofspectrum}
    The Balmer spectrum is functorial with respect to tt-functors between tt-categories. Explicitly, letting $F:\mathcal{K} \longrightarrow \mathcal{L}$ be such, we have an induced continuous map
    \[\begin{tikzcd}
        \Spc(F):\Spc(\mathcal{L}) \arrow[r] & \Spc(\mathcal{K}); \ \mathcal{P} \longmapsto F^{-1}(\mathcal{P}).
    \end{tikzcd}\]
\end{remark}

The following result \cite[Proposition~3.11]{Balmer:tt-geometry} will be used throughout our computations of spectra in later sections.
\begin{prop}[Balmer]\label{prop:Verdierlocgiveshomeo}
    Suppose $\mathcal{I} \subseteq \mathcal{K}$ is a tt-ideal, and denote by $q:\mathcal{K} \longrightarrow \mathcal{K}/\mathcal{I}$ the associated Verdier localisation. Then $\Spc(q)$ induces a homeomorphism between $\Spc(\mathcal{K}/\mathcal{I})$ and its image $\{\mathcal{P} \in \Spc(\mathcal{K}) \mid \mathcal{P} \supseteq \mathcal{I}\}$.
\end{prop}
\begin{remark}\label{remark:principalquotientgiveshomeoontoopen}
    A particular instance of the above result which we will be interested in is the case where $\mathcal{I}$ is the principal ideal $\ideal{a}$, for some $a \in \mathcal{K}$. In this case, $\Spc(q)$ above induces a homeomorphism onto the elementary open $U(a) \coloneqq \Spc(\mathcal{K}) \setminus \supp(a)$.
\end{remark}
In algebra, nilpotence plays a central role in the study of spectra; the tt-landscape is no different, as the following result \cite[Theorem~1.3]{Balmer:Conservative} demonstrates.
\begin{theorem}[Balmer]\label{theorem:nilfaithfulfunctors}
    Suppose $F:\mathcal{K} \longrightarrow \mathcal{L}$ is a tt-functor between tt-categories, and assume that $\mathcal{K}$ is rigid. If $F$ detects $\otimes$-nilpotence of morphisms,\footnote{Sometimes such functors are referred to as \emph{nilfaithful}. Explicitly, this means for all morphisms $f$ in $\mathcal{K}$, $F(f)=0$ only if $f^{\otimes n}=0$ for some $n \geq 1$.} then $\Spc(F)$ is surjective.
\end{theorem}
\begin{remark}\label{remark:inclusionsinducesurjections}
    A particularly useful instance of \Cref{theorem:nilfaithfulfunctors} is when $F:\mathcal{K} \hookrightarrow \mathcal{L}$ is the inclusion of a full rigid tt-subcategory of the tt-category $\mathcal{L}$: in this case the induced map on spectra (given by restriction) is surjective \cite[Cor.~1.8]{Balmer:Conservative}.
\end{remark}

We will utilise the following proposition characterising the images of the induced map on spectra of tt-functors: it is \cite[Cor.~7.8]{Balmer:ThreeRealArtinTates}.
\begin{prop}[Balmer--Gallauer]\label{prop:detectingnilpotenceonanobject}
    Let $F:\mathcal{K} \longrightarrow \mathcal{L}$ be a tt-functor between tt-categories, and assume $\mathcal{K}$ is rigid. Suppose $F$ detects $\otimes$-nilpotence of morphisms on an object $x \in \mathcal{K}$, that is, for every morphism $f:a \longrightarrow b \in \mathcal{K}$, $F(f)=0$ only if $f^{\otimes n} \otimes \id_x=0$ for some $n \geq 1$.
    Then the image of $\Spc(F)$ in $\Spc(\mathcal{K})$ contains $\supp(x)$.
\end{prop}

\subsection{Comparison maps}
For a tt-category $\mathcal{K}$, there is a canonical continuous \emph{comparison map} $\varrho$ from $\Spc(\mathcal{K})$ to the Zariski spectrum $\Spec(R_\mathcal{K})$, where
\[R_\mathcal{K} \coloneqq \End_\mathcal{K}(\unit)\]
is the endomorphism ring of the unit in $\mathcal{K}$. It was introduced in \cite{Balmer:Spectra}, and is characterised by the property
\[\varrho^{-1}(Z(s))=\supp(\cone(s))\]
for each $s \in R_\mathcal{K}$.

Since its conception in \emph{loc. cit.}, it was known that one can define a \emph{graded} comparison map $\varrho_u^\bullet$ with respect to an invertible object $u \in \Pic(\mathcal{K})$ of $\mathcal{K}$, taking values in the homogeneous spectrum of the graded central ring with respect to $u$:
\begin{prop}[Balmer]\label{prop:gradedcomparisonmap}
    Let $u \in \Pic(\mathcal{K})$ be an invertible object, and denote by $R_u^*$ the graded endomorphism ring of the unit $\unit$ with respect to $u$:
    \[R_u^* \coloneqq \oplus_{n \in \Z}\Hom_\mathcal{K}(\unit,u^{\otimes n}).\]
    There is an associated continuous \emph{graded comparison map}
    \[\begin{tikzcd}
    \varrho_u^\bullet:\Spc(\mathcal{K}) \arrow[r] & \Spec^h(R_u^*)
    \end{tikzcd}\]
    given by
    \[\varrho_u^\bullet: \mathcal{P} \longmapsto \ideal{f \in R_u^* \mid f \text{ homogeneous and } \cone(f) \notin \mathcal{P}}.\]
\end{prop}
\begin{remark}\label{remark:naturalityofcomparisonmap}
    The comparison map is functorial with respect to tt-functors, that is, $\varrho$ defines a natural transformation from $\Spc(-)$ to $\Spec(R_{(-)})$. At least when $u=\Sigma \unit$, the same holds for $\varrho_{\Sigma \unit}^\bullet$ and the homogeneous spectrum.
\end{remark}
In practise one would like to be able to speak about \emph{multi}graded comparison maps with respect to a finite collection of $\otimes$-invertible objects  $u_1,\ldots, u_r \in \Pic(\mathcal{K})$. In general this is a delicate matter, but for our purposes this was resolved by Dell'Ambrogio--Stevenson \cite{dell2014even} and Dell'Ambrogio \cite[Remark~A.22]{Gallauer:Periods}:
\begin{prop}[Dell'Ambrogio--Stevenson, Dell'Ambrogio]\label{prop:multigradedcomparisonmap}
    Let $u_1,\ldots,u_r \in \Pic(\mathcal{K})$ be invertible objects; denote by $R_{u_1,\ldots,u_r}$ the multigraded ring
    \[R_{u_1,\ldots,u_r} \coloneqq \oplus_{(n_1,\ldots,n_r) \in \Z^r}\Hom_\mathcal{K}(\unit,u_1^{\otimes n_1} \otimes \dots \otimes u_r^{\otimes n_r}).\]
    Then there is a well-defined multigraded comparison map
    \[\begin{tikzcd}
        \varrho^h:\Spc(\mathcal{K}) \arrow[r] & \Spec^h(R_{u_1,\ldots,u_r})
    \end{tikzcd}\]
    given by
    \[\varrho^h:\mathcal{P} \longmapsto \ideal{f \in R_{u_1,\ldots,u_r} \mid f \text{ homogeneous and } \cone(f) \notin \mathcal{P}}\]
\end{prop}
\begin{remark}\label{remark:justifyingnotationforcomparisonmap}
    In the literature, the notation $\rho$ is usually used instead of $\varrho$ to denote the comparison maps described above. Since this clashes with our notation for the pure symbol $\rho=\{-1\} \in K^M_1(k)$ in the Milnor K-theory of a field $k$, we have chosen to adopt this alternative notation. This is also the convention adopted in \cite{Balmer:ThreeRealArtinTates}.
\end{remark}
Next, we prove a lemma which will be useful in studying spectra of categories whose (generalised) comparison map is a homeomorphism.
\begin{lemma}\label{lemma:homeoofcomparisonmapdescendstosubcategories}
    Let $\mathcal{K} \subseteq \mathcal{L}$ be a full rigid tt-subcategory of a tt-category $\mathcal{L}$. Let $R_{u_1,\ldots,u_r}$ denote the graded central ring with respect to the collection $u_1,\ldots,u_r \in \Pic(\mathcal{K}) \subseteq \Pic(\mathcal{L})$ of $\otimes$-invertible objects. Suppose the comparison map with respect to this collection is a homeomorphism for $\mathcal{L}$, that is,
    \[\begin{tikzcd}
        \varrho^h_\mathcal{L}:\Spc(\mathcal{L}) \arrow[r, "\sim"] & \Spec^h(R_{u_1,\ldots,u_r})
    \end{tikzcd}\]
    is a homeomorphism.
    Then the comparison map for $\mathcal{K}$ is also a homeomorphism:
    \[\begin{tikzcd}
        \varrho^h_\mathcal{K}:\Spc(\mathcal{K}) \arrow[r, "\sim"] & \Spec^h(R_{u_1,\ldots,u_r}).
    \end{tikzcd}\]
\end{lemma}
\begin{proof}
    By naturality of the generalised comparison map \cite[Theorem~1.1]{dell2014even} we have a commutative diagram
    \[\begin{tikzcd}
        \Spc(\mathcal{L}) \arrow[d, swap, "\pi", twoheadrightarrow] \arrow[r, "\varrho^h_\mathcal{L}"] & \Spec^h(R_{u_1,\ldots,u_r}) \arrow[d, "\id"] \\
        \Spc(\mathcal{K}) \arrow[r, "\varrho^h_\mathcal{K}"] & \Spec^h(R_{u_1,\ldots,u_r})
    \end{tikzcd}\]
    where $\pi$ is the surjective restriction map, cf.~\Cref{remark:inclusionsinducesurjections}. By assumption, the top horizontal arrow is a homeomorphism. To obtain the result, one verifies that $\varrho^h_\mathcal{K}$ is bijective with continuous inverse given by
    \[\pi \circ (\varrho^h_\mathcal{L})^{-1}.\]
\end{proof}

We will utilise the following result of \cite{dell2014even}, which generalises the Hopkins--Neeman theorem to the graded case.
\begin{theorem}[Dell'Ambrogio--Stevenson]\label{theorem:gradedperfectcomplexes}
    Let $R$ be a graded commutative ring, and denote by
    \[\D^{\mathsf{perf}}(R) \coloneqq \D(\mathsf{GrMod}_R)^\omega \simeq \K^b(\mathsf{grproj}_R)\]
    the perfect complexes over the graded ring $R$. Then the graded comparison map
    \[\begin{tikzcd}
        \varrho^h:\Spc(\D^{\mathsf{perf}}(R)) \arrow[r] & \Spec^h(R)
    \end{tikzcd}\]
    is a homeomorphism.
\end{theorem}
\subsection{Weight structures}
Weight structures, introduced by Bondarko \cite{Bondarko:WeightStructures}, permit one to filter objects in triangulated categories in a dual fashion to $t$-structures. They are one of the main tools used in this paper.
\begin{definition}\label{def:weightstructure}
    Let $\mathcal{K}$ be a triangulated category. A \emph{weight structure} on $\mathcal{K}$ is a pair of full subcategories $w=(\mathcal{K}_{\leq 0},\mathcal{K}_{\geq 0})$ of $\mathcal{K}$ satisfying the following conditions (where, for each $n \in \Z$, we denote $\mathcal{K}_{\leq n} \coloneqq \Sigma^n\mathcal{K}_{\leq 0}$; $\mathcal{K}_{\geq n} \coloneqq \Sigma^n\mathcal{K}_{\geq 0}$):
    \begin{enumerate}
        \item[$(1)$] $\mathcal{K}_{\leq 0},\mathcal{K}_{\geq 0}$ are closed under retracts;
        \item[$(2)$] $\mathcal{K}_{\leq 0} \subseteq \mathcal{K}_{\leq 1}$ and $\mathcal{K}_{\geq 1} \subseteq \mathcal{K}_{\geq 0}$;
        \item[$(3)$] $\mathcal{K}_{\leq 0} \perp \mathcal{K}_{\geq 1}$, that is, $\Hom_\mathcal{K}(X,Y)=0$ for all $X \in \mathcal{K}_{\leq 0}, Y \in \mathcal{K}_{\geq 1}$;
        \item[$(4)$] For every $X \in \mathcal{K}$, there exists a distinguished triangle
        \[\begin{tikzcd}
            A \arrow[r] & X \arrow[r] & B \arrow[r] & \Sigma A
        \end{tikzcd}\]
        where $A \in \mathcal{K}_{\leq 0}$ and $B \in \mathcal{K}_{\geq 1}$.
    \end{enumerate}
    We denote by
    \[\mathcal{K}^\heartsuit \coloneqq \mathcal{K}_{\leq 0} \cap \mathcal{K}_{\geq 0}\]
    the \emph{heart} of the weight structure: it is a full additive subcategory of $\mathcal{K}$.

    We call a weight structure \emph{bounded} if $\mathcal{K}=\bigcup_{a \leq b} (\mathcal{K}_{\leq b} \cap \mathcal{K}_{\geq a})$, that is, every object is concentrated in finitely many weight-degrees. Equivalently, a weight structure is bounded if and only if $\mathcal{K}$ is generated as a triangulated category by its heart $\mathcal{K}^\heartsuit$.
\end{definition}
\begin{remark}\label{remark:conventionforweightstructures}
    There are two conventions for weight structures, corresponding to homological versus cohomological notation. Here, we adopt the homological convention. We shall refer to the objects of $\mathcal{K}_{\geq 0}$ (resp. $\mathcal{K}_{\leq 0}$) as \emph{connective} (resp. \emph{coconnective}). Note that both subcategories $\mathcal{K}_{\geq 0}$, $\mathcal{K}_{\leq 0}$ are closed under extensions.
\end{remark}
\begin{remark}\label{remark:weightvststructures}
    As mentioned above, the definition of a weight structure is somewhat dual to that of a $t$-structure. Despite the subtle differences in the two definitions, the structures exhibit quite different behaviour. For example, it is well-known that the heart of a $t$-structure is always an \emph{abelian} category, which is generally not the case for weight structures. Another key difference is that, unlike truncation triangles for $t$-structures, the decomposition triangles of \Cref{def:weightstructure} condition $(4)$ are in general not canonical. Despite these hurdles, it turns out that for bounded weight structures on triangulated categories admitting a suitable enhancement, the weight filtrations on objects allow one to associate \emph{weight complexes} to objects in a functorial manner, well-defined up to chain homotopy equivalence. See \Cref{theorem:existenceofweightcomplexfunctor} below.
\end{remark}
\begin{definition}\label{def:tensorcompatibleweightstructure}
    A weight structure $w=(\mathcal{K}_{\leq 0},\mathcal{K}_{\geq 0})$ on a tt-category $\mathcal{K}$ is \emph{$\otimes$-compatible} if
    \begin{enumerate}
        \item[$(1)$] $\unit \in \mathcal{K}^\heartsuit$;
        \item[$(2)$] both the connective and coconnective objects are closed under $\otimes$:
        \[\mathcal{K}_{\leq 0} \otimes \mathcal{K}_{\leq 0} \subseteq \mathcal{K}_{\leq 0}\]
        and
        \[\mathcal{K}_{\geq 0} \otimes \mathcal{K}_{\geq 0} \subseteq \mathcal{K}_{\geq 0}.\]
    \end{enumerate}
    In this case, the symmetric monoidal structure of $\mathcal{K}$ restricts to the heart $\mathcal{K}^\heartsuit$, furnishing it with the structure of a $\otimes$-additive category.
\end{definition}
\begin{remark}\label{remark:boundedtensorcompatibleweightstructure}
     For a bounded weight structure, being $\otimes$-compatible is equivalent to the symmetric monoidal structure of $\mathcal{K}$ restricting to the heart $\mathcal{K}^\heartsuit$.
\end{remark}
\begin{example}\label{example:Kbweightstructure}
    Let $\mathcal{K} \coloneqq \K^b(\mathcal{A})$ denote the bounded homotopy category of chain complexes over an idempotent complete additive category $\mathcal{A}$, endowed with its obvious triangulated structure (in particular, $\Sigma=(-)[1]$). Then there is a weight structure on $\mathcal{K}$ where
    \[\mathcal{K}_{\leq 0} \coloneqq \{M_\bullet \in \mathcal{K} \mid M_\bullet \simeq N_\bullet \text{ for some $N_\bullet$ concentrated in nonpositive homological degrees}\}\]
    and
    \[\mathcal{K}_{\geq 0} \coloneqq \{M_\bullet \in \mathcal{K} \mid M_\bullet \simeq N_\bullet \text{ for some $N_\bullet$ concentrated in nonnegative homological degrees}\}.\]
    Its heart $\mathcal{K}^\heartsuit$ is equivalent to the additive category $\mathcal{A}$. When $\mathcal{A}$ is moreover a $\otimes$-additive category, $\mathcal{K}$ is a tt-category and this weight structure is $\otimes$-compatible.
\end{example}
\begin{remark}\label{remark:weightstructuresoninfinitycategories}
    A weight structure on a stable $\infty$-category $\mathcal{C}$ is a weight structure on its homotopy category $\mathrm{h}\mathcal{C}$. Caution is required, however: unlike with $t$-structures, the $\infty$-heart $\mathcal{C}_\infty^\heartsuit \coloneqq \mathcal{C}_{\leq 0} \cap \mathcal{C}_{\geq 0}$ (where $\mathcal{C}_{\geq 0}$ (resp. $\mathcal{C}_{\leq 0}$) is the full subcategory of $\mathcal{C}$ spanned by the connective (resp. coconnective) objects in $\mathrm{h}\mathcal{C}$) is in general \emph{not} discrete, i.e. the nerve of a $1$-category. We do however have $\mathrm{h}(\mathcal{C}_\infty^\heartsuit) \simeq (\mathrm{h}\mathcal{C})^\heartsuit$. In this article, whenever we speak of the ``heart'' of a weight structure, whether on a triangulated category or stable $\infty$-category, we always mean the $1$-categorical heart.
\end{remark}

One of the most useful tools associated to a ($\otimes$-compatible) bounded weight structure is the \emph{weight complex functor}. Its existence is in general only known when the tt-category possesses a model; the following result is a combination of \cite[Corollary~3.5]{Sosnilo:WeightStructures}, \cite[Corollary~4.5]{Aoki:WeightFunctorIsTensor}.

\begin{theorem}[Sosnilo, Aoki]\label{theorem:existenceofweightcomplexfunctor}
    Let $\mathcal{C}$ be a stably symmetric monoidal $\infty$-category with a $\otimes$-compatible bounded weight structure, and let
    \[\mathcal{A} \coloneqq (\mathrm{h}\mathcal{C})^\heartsuit\]
    denote its heart. Then there exists an essentially unique tt-functor
    \[\begin{tikzcd}
        t:\mathrm{h}\mathcal{C} \arrow[r] & \K^b(\mathcal{A})
    \end{tikzcd}\]
    which is the identity when restricted to the heart $\mathcal{A} \subseteq \mathrm{h}\mathcal{C}$, known as the \emph{weight complex functor}.
\end{theorem}

\begin{example}\label{example:weightcomplexfunctorforKb}
    Consider $\mathcal{K}=\K^b(\mathcal{A})$ as in \Cref{example:Kbweightstructure}. This category admits an enhancement, and the weight complex functor $t$ is the identity.
\end{example}

The main result we will use to construct weight structures is the following, which is found within \cite[Theorem~4.3.2]{Bondarko:WeightStructures}.

\begin{theorem}[Bondarko]\label{theorem:constructingweightstructures}
    Let $\mathcal{K}$ be an idempotent complete triangulated category. Let $S \subseteq \mathcal{K}$ be a set of objects in $\mathcal{K}$ such that the following conditions hold.
    \begin{itemize}
        \item[$(1)$] $S \perp \Sigma^nS$ for every $n \geq 1$;
        \item[$(2)$] the thick subcategory of $\mathcal{K}$ generated by $S$, $\thick(S)$, coincides with $\mathcal{K}$.
    \end{itemize}
    Then there exists a unique weight structure on $\mathcal{K}$ whose heart is given by the idempotent completion of the additive closure $\mathsf{add}(S)^\natural$ of $S$ in $\mathcal{K}$.
\end{theorem}
\begin{remark}\label{remark:weightsstructureistensor}
    Observe that if $S \subseteq \mathcal{K}$ contains the unit $\unit$ and $S \otimes S \subseteq S$, then the weight structure on $\mathcal{K}$ determined by \Cref{theorem:constructingweightstructures} is $\otimes$-compatible by \Cref{remark:boundedtensorcompatibleweightstructure}.
\end{remark}

\subsection{Motivic preliminaries}
Let $k$ be a field, denote by $e$ its exponential characteristic, and let $R$ be a commutative ring in which $e$ is invertible. We write 
\[\SH(k)\] 
for the \textit{Morel--Voevodsky stable motivic homotopy category over $k$} \cite{morelvoev1999}, and 
\[\DM(k;R)\] 
for the \textit{big tt-category of Voevodsky motives with coefficients in $R$} \cite{Voevodsky:TriangulatedMotives, Cisinski:TriangulatedMixedMotives}. We write 
\[\mathcal{SH}(k), \ \mathcal{DM}(k;R)\] 
for the corresponding stably symmetric monoidal $\infty$-categorical enhancements \cite{robalo2015k, rondigs2008modules, hoyois2017motivic}. Under the above assumptions on the coefficients, $\mathcal{DM}(k;R)$ may be identified with the $\infty$-category of modules over the motivic Eilenberg–MacLane spectrum $\mathrm{H} R$ in $\mathcal{SH}(k)$ \cite[Theorem 3.1]{cisinski2015integral}.

The tt-category $\DM(k;R)$ is rigidly-compactly generated by motives of smooth $k$-schemes, and we denote the full subcategory of compact objects by 
\[\DM_\gm(k;R).\] 
We write 
\[\unit \simeq M(\Spec(k))\] 
for the tensor unit of $\DM(k;R)$ and for integers $a$ and $b$ we write 
\[\unit(a)[b]\] 
for the $b$-fold suspension of the $a$-fold Tate twist of the unit. More generally, for an object $X$ in $\DM(k;R)$, we write 
\[X(a)[b] \coloneqq X \otimes \unit(a)[b].\] 
\subsection{Tate motives}
We define \textit{the big tt-category of Tate motives over $k$ with coefficients in
$R$} to be the localising subcategory
\[\DTM(k;R) \coloneqq \Loc(\unit(n)\mid n\in\mathbb Z) \subseteq \DM(k;R).\]
Its compact objects form the
\textit{tt-category of geometric Tate motives}
\[\DTM_\gm(k;R) \coloneqq \DTM(k;R)^\omega=\thick\big(\unit(n)\mid n\in\mathbb Z\big).\]
For the remainder of this subsection, we restrict to coefficients
$R=\F_p$, where $p$ is a prime different from the characteristic of $k$. We write
\[H^{a,b}_{\mathcal{M}}(k;\F_p) \coloneqq \Hom_{\DM(k;\F_p)}(\unit,\unit(b)[a])\]
for the \textit{motivic cohomology of $\Spec(k)$} with $\F_p$-coefficients. Composition of morphisms and the tensor product endow
\[H^{*,*}_{\mathcal{M}}(k;\F_p) \coloneqq \bigoplus_{a,b\in\mathbb Z} H^{a,b}_{\mathcal{M}}(k;\F_p)\]
with the structure of a bigraded-commutative ring. For the remainder of this section, we assume that $k$ contains a primitive $p^{th}$ root of unity. By the Bloch-Kato and Beilinson--Lichtenbaum conjectures (really, theorems by \cite{voevodsky2011motivic}), a choice of such a root determines a class
\[ \tau\in H^{0,1}_{\mathcal{M}}(k;\F_p),\] 
and yields an isomorphism of bigraded rings 
\[H^{*,*}_{\mathcal{M}}(k;\F_p) \simeq (K^M_*(k)/p)[\tau],\]
where $K^M_*(k)/p$ denotes the Milnor K-theory modulo $p$ of $k$. Here a pure symbol
\[ \alpha= \{a_1,\ldots,a_n\}\in K^M_n(k)/p\] 
corresponds to a morphism 
\[\begin{tikzcd}
    \alpha: \unit \arrow[r] & \unit(n)[n],
\end{tikzcd}\] 
so that $K^M_n(k)/p$ lies in bidegree $(n,n)$, while $\tau$ has bidegree $(0,1)$.

We next recall the computation of the homogeneous spectrum of the graded ring
$K^M_*(k)/p$. Let
\[\mathfrak{m} \coloneqq K^M_{>0}(k)/p\]
denote its irrelevant ideal. When $p=2$, let
\[\Sper(k)\] 
denote \textit{the space of orderings}, or \textit{the real spectrum, of $k$}, equipped with the
Harrison topology, and for
$\alpha\in\Sper(k)$ set
\[\mathfrak{p}_\alpha \coloneqq \langle \{a\} \in  K^M_1(k)/2 \mid a >_{\alpha} 0\rangle \subseteq K^M_*(k)/2.\]
\begin{theorem}[\cite{thornton2016homogeneous}]\label{thm:hspectrum-MilnorK}
The homogeneous spectrum of $K^M_*(k)/p$ is described as follows.
\begin{enumerate}
    \item If $p$ is odd, then
    \[\Spec^h(K^M_*(k)/p)=\{\mathfrak{m}\}.\]
    \item If $p=2$, its homogeneous primes are
    \[\mathfrak{m},\
        \mathfrak{p}_\alpha \, \text{ for } \, 
        \alpha\in\Sper(k).\]
    For distinct orderings $\alpha$ and $\beta$, there are no inclusions between $\mathfrak{p}_\alpha$ and $\mathfrak{p}_\beta$, while
    \[\mathfrak{p}_\alpha\subseteq\mathfrak{m}, \ \text{for all } \alpha \in \Sper(k).\]
    Moreover, the map 
    \[\alpha\longmapsto\mathfrak{p}_\alpha\]
identifies 
\[\Sper(k) \cong \Proj(K^M_*(k)/2)\]
as topological spaces.
\end{enumerate}
\end{theorem}
The homogeneous spectrum of 
\[(K^M_*(k)/p)[\tau]\]
can be readily deduced from the preceding description.
\begin{cor}\label{cor:hspectrum-MilnorKtau}
Quotienting by $\tau$ and inverting $\tau$ induce natural
homeomorphisms
\[V(\tau)\cong \Spec^h(K^M_*(k)/p), \ D(\tau)\cong \Spec^h(K^M_*(k)/p).\]
In particular, the homogeneous prime ideals in $(K^M_*(k)/p)[\tau]$ are
\[\mathfrak{q} + (\tau), \, \mathfrak{q}, \ \text{ for } \mathfrak{q} \in \Spec^h(K^M_*(k)/p).\]
Hence, if $p$ is odd or $k$ is not formally real, then
\[\Spec^h((K^M_*(k)/p)[\tau])\]
is the two-point Sierpiński space. If $p=2$ and $k$ is formally real, we obtain the following space:
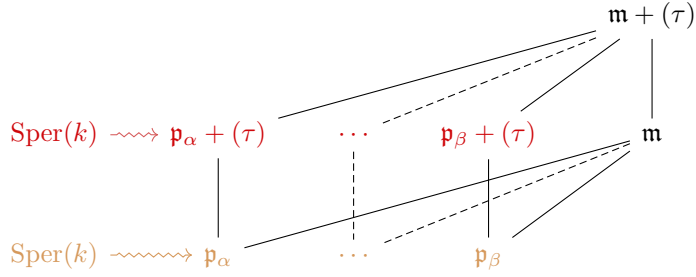
\[\begin{tikzcd}[ampersand replacement=\&,sep=scriptsize]
\&\&\&\& {\mathfrak{m}+(\tau)} \&\&\& \\
	\\
	\textcolor{rgb,255:red,207;green,7;blue,11}{{\Sper(k)}} \& \textcolor{rgb,255:red,207;green,7;blue,11}{{\mathfrak{p}_{\alpha}+(\tau)}} \& \textcolor{rgb,255:red,207;green,7;blue,11}{\cdots} \& \textcolor{rgb,255:red,207;green,7;blue,11}{{\mathfrak{p}_\beta+(\tau)}} \& {\mathfrak{m}} \& {} \& {} \& {} \\
	\\
	\textcolor{rgb,255:red,214;green,153;blue,92}{{\Sper(k)}} \& \textcolor{rgb,255:red,214;green,153;blue,92}{{\mathfrak{p}_\alpha}} \& \textcolor{rgb,255:red,214;green,153;blue,92}{\cdots} \& \textcolor{rgb,255:red,214;green,153;blue,92}{{\mathfrak{p}_{\beta}}}
	\arrow[no head, from=1-5, to=3-2]
	\arrow[dashed, no head, from=1-5, to=3-3]
	\arrow[no head, from=1-5, to=3-4]
	\arrow[no head, from=1-5, to=3-5]
	\arrow[color={rgb,255:red,214;green,92;blue,92}, squiggly, from=3-1, to=3-2]
	\arrow[no head, from=3-2, to=5-2]
	\arrow[dashed, no head, from=3-3, to=5-3]
	\arrow[no head, from=3-4, to=5-4]
	\arrow[no head, from=3-5, to=5-2]
	\arrow[dashed, no head, from=3-5, to=5-3]
	\arrow[no head, from=3-5, to=5-4]
	\arrow[color={rgb,255:red,214;green,153;blue,92}, squiggly, from=5-1, to=5-2]
\end{tikzcd}\]
with the topology determined by the Harrison topology on $\Sper(k)$ and the specialisations described above. 
\end{cor}
Since $H^{*,*}_{\mathcal{M}}(k;\F_p)$ is by definition the bigraded endomorphism ring of the tensor unit with respect to suspension and Tate twist, we have a bigraded comparison map
\[\begin{tikzcd}
    \varrho^h:\Spc(\DTM_\gm(k;\F_p)) \arrow[r] & \Spec^h(H^{*,*}_{\mathcal{M}}(k;\F_p)).
\end{tikzcd}\] 
The first main result of this paper is that this map is a homeomorphism for any prime $p$ different from the characteristic of $k$, thereby extending the known computations for algebraically closed fields \cite{Gallauer:TateSpectrum} and real closed fields \cite{Balmer:ThreeRealArtinTates} to the general case.
\subsection{Artin motives}
Fix a separable closure $k^{sep}$ of a field $k$, and let
\[\Gamma_k \coloneqq \operatorname{Gal}(k^{sep}/k)\]
denote the \textit{absolute Galois group of $k$}. By the Grothendieck--Galois correspondence, the category of finite étale $k$-schemes is equivalent to the category of finite continuous $\Gamma_k$-sets \cite{grothendieck2003revetements}. After linearising with respect to a field $\F$, this gives rise to the additive tensor category
\[\mathsf{perm}(\Gamma_k; \F)\] 
of finite permutation $\F[\Gamma_k]$-modules. This linearisation of finite étale $k$-schemes via permutation modules admits a motivic counterpart. We define the \textit{big tt-category of Artin motives over $k$ with
coefficients in $\F$} to be the localising subcategory
\[\DAM(k;\F) \coloneqq \Loc(M(E)\mid E/k \text{ finite separable}) \subseteq \DM(k;\F).\]
Its compact objects form the tt-category of geometric Artin motives \[\DAM_\gm(k;\F) \coloneqq \DAM(k;\F)^\omega= \thick(M(E)\mid E/k\text{ finite separable}).\] 
By \cite{Voevodsky:TriangulatedMotives} the Grothendieck--Galois correspondence extends to an equivalence of tt-categories
\[\DAM_\gm(k;\F) \simeq \K^b(\mathsf{perm}(\Gamma_k;\F)^\natural).\]
The Balmer spectrum
\[\Spc(\K^b(\mathsf{perm}(\Gamma_k;\F)^\natural))\]
was computed in \cite{balmer2025geometry, Balmer:Artin}; via the above
equivalence, this gives a complete description of the spectrum of Artin motives.
\begin{example}[\cite{Balmer:ThreeRealArtinTates,Balmer:Artin}]
The complexity of $\Spc(\DAM_\gm(k;\F))$ varies significantly with the arithmetic of the base field.
\begin{enumerate}
    \item If $k$ is algebraically closed, then $\Gamma_k$ is trivial and $\Spc(\DAM_\gm(k;\F))$ is a singleton.
    \item If $k$ is real closed, then $\Gamma_k\simeq C_2$. If the characteristic of $\F$ is not $2$, $\Spc(\DAM_\gm(k;\F))$ is a singleton, whereas in characteristic $2$ it is the following three-point space:
    
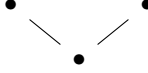
\[\begin{tikzcd}[ampersand replacement=\&, sep=small]
	\bullet \&\& \bullet \\
	\& \bullet
	\arrow[no head, from=1-1, to=2-2]
	\arrow[no head, from=2-2, to=1-3]
\end{tikzcd}\]
\item If $k$ is a finite field, then $\Gamma_k\simeq \widehat{\Z}$. If $\F$ has characteristic $0$, $\Spc(\DAM_\gm(k;\F))$ is also a singleton. In positive characteristic, it is the countably infinite
    non-Noetherian space described in \cite{Balmer:Artin} as follows,
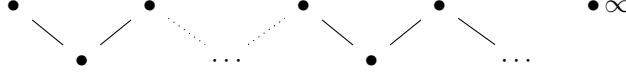
\[\begin{tikzcd}[ampersand replacement=\&, sep=small]
	\bullet \&\& \bullet \&\& \bullet \&\& \bullet \&\& {\bullet \, \infty} \\
	\& \bullet \&\& \cdots \&\& \bullet \&\& \cdots
	\arrow[no head, from=1-1, to=2-2]
	\arrow[dotted, no head, from=1-3, to=2-4]
	\arrow[no head, from=1-5, to=2-6]
	\arrow[no head, from=1-7, to=2-8]
	\arrow[no head, from=2-2, to=1-3]
	\arrow[dotted, no head, from=2-4, to=1-5]
	\arrow[no head, from=2-6, to=1-7]
\end{tikzcd}\]
where the closed subsets $Z$ are the specialisation closed subsets such that if $\infty \not \in Z$, then $Z$ is finite.  
\end{enumerate}
\end{example}
\subsection{Artin--Tate motives}
We define \textit{the big tt-category of Artin--Tate motives over $k$ with
coefficients in $R$} to be the localising subcategory
\[\DATM(k;R) \coloneqq \Loc(M(E)(n)\mid E/k \text{ finite separable},\ n\in\mathbb Z) \subseteq \DM(k;R).\]
Its compact objects form the \textit{tt-category of geometric Artin--Tate motives}
\[ \DATM_\gm(k;R) \coloneqq \DATM(k;R)^\omega = \thick(M(E)(n)\mid E/k \text{ finite separable},\ n\in\mathbb Z).\]
We shall mainly be interested in the case $R=\F_p$ where $p$ is a prime different from the characteristic of $k$. The inclusions
\[\begin{tikzcd}
    \DTM_\gm(k;\F_p) \arrow[r, hookrightarrow] & \DATM_\gm(k;\F_p) \arrow[r, hookleftarrow] & \DAM_\gm(k;\F_p)
\end{tikzcd}\]
induce surjective restriction maps on the Balmer spectra
\[\begin{tikzcd}
    \Spc(\DTM_\gm(k;\F_p)) & \arrow[l, twoheadrightarrow] \Spc(\DATM_\gm(k;\F_p)) \arrow[r, twoheadrightarrow] & \Spc(\DAM_\gm(k;\F_p)),
\end{tikzcd}\]
which record respectively the Tate and Artin components of a prime tt-ideal of $\DATM_\gm(k;\F_p)$.

The Balmer spectrum of $\DATM_\gm(k;\F_p)$ is already known in two cases.
\begin{enumerate}
    \item If $k$ is algebraically closed, then $\DATM_\gm(k;\F_p)=\DTM_\gm(k;\F_p)$, in which case the computation follows from \cite{Gallauer:TateSpectrum}.
    \item (\cite{Balmer:ThreeRealArtinTates}) If $k$ is real closed and $p$ is an odd prime, $\Spc(\DATM_\gm(k;\F_p))\cong \Spc(\DTM_\gm(k;\F_p))$ is the Sierpiński space. If $p=2$, the spectrum is described by the following six-point space:
    
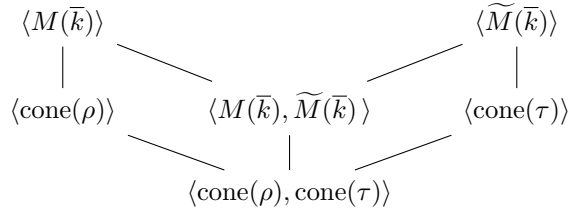
\[\begin{tikzcd}[ampersand replacement=\&, sep=scriptsize]
	{\langle M(\overline{k})\rangle} \&\& {\langle  \widetilde{M}(\overline{k})  \rangle} \\
	{\langle \cone(\rho) \rangle} \& {\langle M(\overline{k}), \widetilde{M}(\overline{k}) \, \rangle} \& {\langle \cone(\tau) \rangle} \\
	\& {\langle \cone(\rho), \cone(\tau)\rangle}
	\arrow[no head, from=1-1, to=2-1]
	\arrow[no head, from=1-1, to=2-2]
	\arrow[no head, from=1-3, to=2-3]
	\arrow[no head, from=2-1, to=3-2]
	\arrow[no head, from=2-2, to=1-3]
	\arrow[no head, from=2-2, to=3-2]
	\arrow[no head, from=2-3, to=3-2]
    \end{tikzcd}\]
    where $\widetilde{M}(\overline{k})$ denotes the \emph{(completely) reduced Rost motive} associated to the pure symbol $\rho=\{-1\} \in K^M_1(k)/2$.
\end{enumerate}
One of the main results of this paper is the computation of $\Spc(\DATM_\gm(k;\F_p))$ where $k$ is any field $k$ and $p$ is any prime different from the characteristic of $k$, which explicitly describes how its Artin and Tate
strata are glued together.
\subsection{\'Etale motives}
Following \cite{cisinski2016etale}, we write $\DM_\et(k;\F_p)$ for \textit{the big tt-category of étale motives over $k$ with coefficients in $\F_p$}. Étale realisation gives rise to a tt-functor
\[\begin{tikzcd}
    \Re_{\acute{e}t}:\DM(k;\F_p) \arrow[r] & \DM_\et(k;\F_p).
\end{tikzcd}\]
Unlike $\DM(k;\F_p)$, the category $\DM_\et(k;\F_p)$ is not
compactly generated in general \cite[Remark 5.4.10]{cisinski2016etale}. We therefore
consider its essentially small tt-subcategory of rigid objects
\[\DM_\et^\vee(k;\F_p) \subseteq \DM_\et(k;\F_p).\]
The étale Tate motive $\unit(1)[1]$ determines a graded comparison map with target the homogeneous spectrum of the graded endomorphism ring
\[\mathbb{E}\mathrm{nd}_{\unit(1)[1]}(\unit) \coloneqq \bigoplus_{n\in\mathbb Z} \Hom_{\DM_\et(k;\F_p)}(\unit,\unit(n)[n]),\]
which identifies with the graded ring
\[H^*_{\acute{e}t}(k,\mu_p^{\otimes *}).\]
By the Bloch--Kato conjecture, 
\[\mathbb{E}\mathrm{nd}_{\unit(1)[1]}(\unit) \simeq K^M_*(k)/p\]
so the graded comparison map associated with the $\otimes$-invertible object $\unit(1)[1]$ takes the form
\[\begin{tikzcd}
    \varrho_{\acute{e}t}^h:\Spc(\DM_\et^\vee(k;\F_p)) \arrow[r] & \Spec^h(K^M_*(k)/p).
\end{tikzcd}\]
\begin{theorem}[\cite{gallauer2021note}]\label{thm:gallauer-etale-spectrum}
The comparison map $\varrho_{\acute{e}t}^h$ is a homeomorphism.
\end{theorem}
Assume now that $k$ contains a primitive $p^{th}$ root of unity. The morphism $\tau: \unit \longrightarrow \unit(1)$ in $\DM(k;\F_p)$ becomes invertible under the étale realisation. We shall prove in \Cref{prop:tau-localisation-etale} that étale realisation induces an equivalence
\[\begin{tikzcd}
    (\DM_\gm(k;\F_p)/\langle\cone(\tau)\rangle)^\natural \arrow[r, "\sim"] & \DM_\et^\vee(k;\F_p).
\end{tikzcd}\]
Thus, rigid étale motives identify with geometric Voevodsky motives after inverting $\tau$. This equivalence cannot in general be promoted to the corresponding big categories: the localisation $\DM(k;\F_p)[\tau^{-1}]$ is always compactly generated, whereas $\DM_\et(k;\F_p)$ need not be.

By \Cref{remark:principalquotientgiveshomeoontoopen} and \Cref{thm:gallauer-etale-spectrum} we obtain
\[U(\cone(\tau)) \cong \Spc(\DM_\et^\vee(k;\F_p)) \cong \Spec^h(K^M_*(k)/p),\]
that is, the elementary open of $\Spc(\DM_\gm(k;\F_p))$ corresponding to $\cone(\tau)$ is completely determined by the homogeneous spectrum of Milnor $K$-theory modulo $p$.
\section{The spectrum of Tate motives}
Let $k$ be a field and let $p$ be a prime different from
$\operatorname{char}(k)$. The main result of this section is the
following.

\begin{theorem}\label{thm:tate-comparison}
The comparison map 
\[\begin{tikzcd}
    \varrho^h: \Spc(\DTM_\gm(k;\F_p)) \arrow[r] & \Spec^h(H^{*,*}_{\mathcal{M}}(k;\F_p))
\end{tikzcd}\]
is a homeomorphism.
\end{theorem}

We begin by assuming that $k$ contains a
primitive $p^{th}$ root of unity. In this case,
\[H^{*,*}_{\mathcal{M}}(k;\F_p) \simeq (K^M_*(k)/p)[\tau].\]
Moreover, by \Cref{cor:hspectrum-MilnorKtau}, we know that the homogeneous spectrum on the right decomposes into two copies of $\Spec^h(K^M_*(k)/p)$: the open subset $D(\tau)$ obtained by inverting $\tau$, and the closed subset $V(\tau)$ defined by quotienting by $\tau$. Our strategy
is to identify the corresponding open and closed pieces of
$\Spc(\DTM_\gm(k;\F_p))$ separately.

The open piece is described by the localisation
\[\begin{tikzcd}
    \DTM_\gm(k;\F_p) \arrow[r] & (\DTM_\gm(k;\F_p)/\langle\cone(\tau)\rangle)^\natural.
\end{tikzcd}\]
We will show that the target is equivalent to the category of rigid étale Tate motives. Its spectrum is therefore
homeomorphic to $\Spec^h(K^M_*(k)/p)$ by \Cref{thm:gallauer-etale-spectrum}, and the induced map on spectra identifies with the inclusion
\[U(\cone(\tau)) \subseteq \Spc(\DTM_\gm(k;\F_p)).\]
The closed piece is detected by the weight complex functor
\[\begin{tikzcd}
    t:\DTM_\gm(k;\F_p) \arrow[r] & \K^b(\mathcal{H}^{\mathsf{T}}),
\end{tikzcd}\]
where
\[\mathcal{H}^{\mathsf{T}} = \add(\unit(n)[n]\mid n\in \Z)^\natural.\]
We show that
\[\im(\Spc(t))=\supp(\cone(\tau))\]
and that this closed subspace is again homeomorphic to
$\Spec^h(K^M_*(k)/p)$.
Together, these two computations show that $\varrho^h$ is bijective; to prove that it is a homeomorphism, it remains to determine the specialisations connecting the open piece $U(\cone(\tau))$ with the closed piece $\supp(\cone(\tau))$. For odd $p$, both pieces are singletons and the spectrum is the two-point Sierpiński space, whilst for $p=2$, the specialisations are detected by base change to the real closures associated with the orderings of $k$. Finally, we remove the root of unity assumption by proving that both the Balmer spectrum of $\DTM_{\gm}(k;\F_p)$ as well as the homogeneous spectrum $\Spec^h(H^{*,*}_{\mathcal{M}}(k;\F_p))$ are unchanged upon base change from $k$ to $k(\zeta_p)$, where $\zeta_p$ denotes a primitive $p^{th}$ root of unity for $k$.
\subsection{The two detecting tt-functors}
From now on, unless stated otherwise, we assume that $k$ contains a primitive $p^{th}$ root of unity. We begin by considering the localisation obtained by inverting $\tau$. In the big
category of Tate motives, set
\[\begin{tikzcd}
    \mathfrak{X} \coloneqq \hocolim(\unit \arrow[r, "\tau"] & \unit(1) \arrow[r, "\tau(1)"] & \unit(2) \arrow[r, "\tau(2)"] & \cdots).
\end{tikzcd}\]
This is the tensor-idempotent associated with the smashing
localisation
\[\begin{tikzcd}
    \DTM(k;\F_p) \arrow[r] & \DTM(k;\F_p)/\Loc\ideal{\cone(\tau)} \simeq \mathfrak{X}\otimes\DTM(k;\F_p).
\end{tikzcd}\]
Passing to compact objects yields a tt-functor
\[\begin{tikzcd}
    L_\tau: \DTM_\gm(k;\F_p) \arrow[r] & (\DTM_\gm(k;\F_p)/\langle\cone(\tau)\rangle)^\natural \simeq (\mathfrak{X}\otimes\DTM(k;\F_p))^\omega.
\end{tikzcd}\]
By \Cref{remark:principalquotientgiveshomeoontoopen}, the induced map
\[\begin{tikzcd}
    \Spc(L_\tau):\Spc((\mathfrak{X} \otimes \DTM(k;\F_p))^\omega) \arrow[r] & \Spc(\DTM_\gm(k;\F_p))
\end{tikzcd}\]
is a homeomorphism onto the open subset
\[U(\cone(\tau)) \subseteq \Spc(\DTM_\gm(k; \F_p)).\]
We now construct a tt-functor detecting the complementary closed subset 
\[\supp(\cone(\tau)) \subseteq \Spc(\DTM_\gm(k; \F_p)).\]

\begin{lemma}[Milnor weight structure]\label{lem:milnor-weight-structure}
The category $\DTM_\gm(k;\F_p)$ admits a unique bounded
$\otimes$-compatible weight structure whose heart is
\[\mathcal{H}^{\mathsf{T}} \coloneqq \add(\unit(n)[n]\mid n\in \Z)^\natural.\]
\end{lemma}
\begin{proof}
By definition, the objects
\[\unit(n)[n], \ n\in \Z,\]
thickly generate $\DTM_\gm(k;\F_p)$. So, to apply \Cref{theorem:constructingweightstructures} it suffices to verify the orthogonality condition. For any $r \geq 1$ and $i,j \in \Z$ , we have
\[\begin{aligned}
\Hom(\unit(i)[i], \unit(j)[j+r]) &\simeq \Hom(\unit, \unit(j-i)[j-i+r]) \\
&= H^{j-i+r,j-i}_{\mathcal{M}}(k;\F_p)=0,
\end{aligned}\]
since $\F_p$-motivic cohomology of a field vanishes above the diagonal. The $\otimes$-compatibility of this weight structure is immediate thanks to \Cref{remark:weightsstructureistensor}.
\end{proof}
\Cref{theorem:existenceofweightcomplexfunctor} in combination with \Cref{lem:milnor-weight-structure} yields a tt-functor
\[\begin{tikzcd}
    t:\DTM_\gm(k;\F_p) \arrow[r] & \K^b(\mathcal{H}^{\mathsf{T}}),
\end{tikzcd}\]
which we call \textit{the Milnor weight complex functor}. The following result describes the key property of this functor which we need.
\begin{prop}\label{prop:Milnor-t_-detects-nilp-tau}
Let
\[\begin{tikzcd}
    f:X \arrow[r] & Y
\end{tikzcd}\]
be a morphism in $\DTM_\gm(k;\F_p)$. If $t(f)=0$, then
\[f\otimes\cone(\tau)=0.\]
In particular, $t$ detects the $\otimes$-nilpotence of morphisms on $\cone(\tau)$. 
\end{prop}
\begin{proof}
Since $\DTM_\gm(k;\F_p)$ is rigid, it suffices to treat the case $X=\unit$. Since $t(f)=0$, $f$ factors as 
\[\begin{tikzcd}[ampersand replacement=\&]
	\unit \&\& Y \\
	\& {Y_{\leq -1}}
	\arrow["f", from=1-1, to=1-3]
	\arrow["g"', from=1-1, to=2-2]
	\arrow[from=2-2, to=1-3]
\end{tikzcd}\]
by \cite[Proposition 2.1.1, Proposition 2.3.1 (1)]{Bondarko:KillingWeights} for some $Y_{\leq-1} \in \DTM_\gm(k;\F_p)_{\leq-1}$. We verify that $g \otimes \cone(\tau)=0$. Since $\DTM_\gm(k;\F_p)_{\leq-1}$ is the extension and retract closure of the objects 
\[\unit(x)[y], \ y-x\leq-1,\]
it is enough to prove that 
\[\Hom(\cone(\tau),\cone(\tau)(x)[y])=0\]
whenever $y-x\leq-1$, since the group $\Hom(\cone(\tau), \cone(\tau) \otimes Y_{\leq -1})$ in which $g$ resides is an iterated extension of these ones. Applying $\Hom(-,\cone(\tau)(x)[y])$ to the triangle
\[\begin{tikzcd}
    \unit \arrow[r, "\tau"] & \unit(1) \arrow[r] & \cone(\tau) \arrow[r] & \unit[1],
\end{tikzcd}\]
we see it is enough to verify
\[\Hom(\unit,\cone(\tau)(x-1)[y])=0, \ \Hom(\unit,\cone(\tau)(x)[y-1])=0\]
for such $x,y$. Multiplication by $\tau$ is injective on $H^{*,*}_{\mathcal{M}}(k;\F_p)$, hence the induced long exact sequence associated to the defining triangle of
$\cone(\tau)(i)[j]$ gives
\[\begin{aligned}
\Hom(\unit,\cone(\tau)(i)[j])
& \simeq H^{j,i+1}_{\mathcal{M}}(k;\F_p) \big / \tau H^{j,i}_{\mathcal{M}}(k;\F_p) \\
& \simeq
\begin{cases}
K^M_j(k)/p & \text{if } j=i+1 \geq 0;\\
0 & \text{otherwise.}
\end{cases}
\end{aligned}\]
For $(i,j)=(x-1,y)$, the nonvanishing condition is precisely when $y=x$, while for
$(i,j)=(x,y-1)$ it is $y=x+2$. Both violate $y-x\leq-1$, whence we conclude that
\[g\otimes\cone(\tau)=0,\]
proving the claim.
\end{proof}
\begin{remark}\label{rem:supp-cone-tau}
Since $t(\tau)=0$, we find that $t(\cone(\tau))$ does not belong to any prime tt-ideal of $\Spc(\K^b(\mathcal{H}^{\mathsf{T}}))$. In particular,
\[\im(\Spc(t)) \subseteq \supp(\cone(\tau)).\]
The reverse inclusion follows from the previous result combined with \Cref{prop:detectingnilpotenceonanobject}.
\end{remark}
\begin{cor}\label{cor:joint-detection-nilp-tate}
The pair of tt-functors $(L_\tau, t)$ jointly detects tensor-nilpotence of morphisms.
\end{cor}
\begin{proof}
This immediately follows from \Cref{prop:Milnor-t_-detects-nilp-tau} and the same argument as in the proof of \cite[Corollary 7.8]{Balmer:ThreeRealArtinTates}.
\end{proof}

\subsection{The open étale piece}
We now identify the localisation at $\tau$ with rigid étale motives. The result holds for all geometric motives, although we shall only use its restriction to Tate motives. To the authors' knowledge, this was previously known only under the assumption that $k$ has finite $p$-cohomological dimension (cf. \cite{hasemeyer2005motives} and \cite[Theorem C.1]{Gallauer:TateSpectrum}).

\begin{prop}\label{prop:tau-localisation-etale}
Étale realisation induces an equivalence of tt-categories
\[\begin{tikzcd}
    (\mathfrak{X} \otimes \DM(k;\F_p))^\omega \simeq (\DM_\gm(k;\F_p)/\langle\cone(\tau)\rangle)^\natural \arrow[r, "\sim"] & \DM_\et^\vee(k;\F_p).
\end{tikzcd}\]
\end{prop}
\begin{proof}
Since $\tau$ becomes an isomorphism under étale realisation, we get an induced tt-functor
\[\begin{tikzcd}
    \overline{\Re}_\et:(\mathfrak{X}\otimes \DM(k;\F_p))^\omega \arrow[r] & \DM_\et^\vee(k;\F_p). 
\end{tikzcd}\]
Both categories are thickly generated by the images of objects 
\[M(X)(i)[j],\] 
where $X$ is a smooth $k$-scheme and $i,j \in \Z$. By the 5-lemma, it is therefore enough to verify full faithfulness on these generators. Moreover, by rigidity, we are reduced to compare morphisms groups from $M(X)$ to $\unit(x)[y]$ in both categories. In the $\tau$-local category,
\[\begin{aligned} 
    \Hom(\mathfrak{X}\otimes M(X),\mathfrak{X}(x)[y]) \simeq & \Hom_{\DM(k;\F_p)}(M(X),\mathfrak{X}(x)[y]) \\ 
    \simeq & \operatorname*{colim}_{n} \Hom_{\DM(k;\F_p)}(M(X),\unit(x+n)[y]),
\end{aligned}\]
where the first equivalence follows by virtue of being a localisation of $\DM(k;\F_p)$; the second equivalence follows by standard results on homotopy colimits \cite[Lemma~2.8]{Neeman:BrownRep}. The transition maps in the colimit are given by multiplication by $\tau$. By the Beilinson--Lichtenbaum conjecture, the colimit stabilises to étale cohomology: 
\[\operatorname*{colim}_{n} H^{y,x+n}_{\mathcal{M}}(X;\F_p) \simeq H^{y}_{\acute{e}t}(X;\F_p). \]
On the étale side,
\[\Hom_{\DM_\et(k;\F_p)} \left(M_{\acute{e}t}(X),\unit_{\acute{e}t}(x)[y]\right) \simeq H^{y}_{\acute{e}t}(X;\F_p).\] 
The functor $\overline{\Re}_\et$ identifies these two descriptions: it is therefore fully faithful on generators, and hence an equivalence.
\end{proof}
\begin{remark}
In general, this equivalence does not extend to the corresponding big categories: the localisation $\mathfrak{X} \otimes \DM(k;\F_p)$ is compactly generated, whereas $\DMet(k;\F_p)$ need not be (cf. \cite[Remark 5.4.10]{cisinski2016etale}).
\end{remark}
Let
\[\DTM_\et^\vee(k;\F_p) \coloneqq \thick(\unit_{\acute{e}t}(n) \mid n \in \Z) \subseteq \DM_\et^\vee(k;\F_p)\]
be the category of \emph{rigid étale Tate motives}. Restricting \Cref{prop:tau-localisation-etale} to Tate motives yields an equivalence
\[\begin{tikzcd}
    (\mathfrak{X}\otimes\DTM(k;\F_p))^\omega \arrow[r, "\sim"] & \DTM_\et^\vee(k;\F_p).
\end{tikzcd}\]
We can now compute the open piece $U(\cone(\tau))$ of the Tate spectrum.
\begin{cor}\label{cor:tau-open-tates}
There are natural homeomorphisms
\[U(\cone(\tau)) \cong \Spc(\DTM_\et^\vee(k;\F_p)) \cong \Spec^h(K^M_*(k)/p).\]
\end{cor}
\begin{proof}
The first homeomorphism follows from \Cref{prop:tau-localisation-etale}. The graded comparison map
for rigid étale motives is a homeomorphism by \Cref{thm:gallauer-etale-spectrum}; its restriction to the full tt-subcategory of rigid étale Tate motives is therefore again a homeomorphism by
\Cref{lemma:homeoofcomparisonmapdescendstosubcategories}.
\end{proof}

\begin{remark}\label{rem:open-tau-DM}
Note that \Cref{prop:tau-localisation-etale} and \Cref{thm:gallauer-etale-spectrum} also identify the corresponding open subset of $\Spc(\DM_\gm(k;\F_p))$:
\[U(\cone(\tau)) \cong \Spc(\DM_\et^\vee(k;\F_p)) \cong \Spec^h(K^M_*(k)/p).\]
\end{remark}

\subsection{The Milnor closed piece}
In order to describe the closed subset 
\[\supp(\cone(\tau)) \subseteq \Spc(\DTM_\gm(k;\F_p))\]
we begin by studying the target $\K^b(\mathcal{H}^{\mathsf{T}})$ of the Milnor weight complex functor. For ease of notation, we set
\[R_k \coloneqq K^M_*(k)/p,\]
and for $n\in \Z$, we write $R_k\langle n\rangle$ for the graded shift defined by
\[R_k\langle n\rangle^d \coloneqq R_k^{n+d}.\]
This is a free graded $R_k$-module. We write 
\[\mathsf{free^{gr}}(R_k) \coloneqq \add(R_k\langle n\rangle\mid n\in \Z) \subseteq \mathsf{mod^{gr}}(R_k)\]
for the category of finitely generated graded free $R_k$-modules,  and degree-zero graded $R_k$-module homomorphisms. In particular,
\[\Hom_{\mathsf{mod^{gr}}(R_k)}(R_k\langle i\rangle,R_k\langle j\rangle) \simeq R_k^{j-i}.\]
We will identify $\K^b(\mathcal{H}^{\mathsf{T}})$ with the category of perfect
complexes of graded $R_k$-modules $\D^{\mathsf{perf}}(R_k)$. The result of
Dell'Ambrogio--Stevenson recalled in \Cref{theorem:gradedperfectcomplexes} will then yield a natural homeomorphism
\[\Spc(\K^b(\mathcal{H}^{\mathsf{T}})) \cong \Spec^h(R_k).\]
We begin by recording that there is no distinction between finitely generated graded free and projective $R_k$-modules so that, in particular, $\mathsf{free^{gr}}(R_k)$ is idempotent-complete. 

\begin{lemma}\label{lem:graded-projectives-Milnor-free}
Every finitely generated graded projective $R_k$-module is graded
free.
\end{lemma}
\begin{proof}
This is essentially an application of \cite[Lemma 3.8]{burke2015derived} (their proof goes through verbatim): we spell out the details of the argument for the reader's convenience. Let $P$ be a non-zero finitely generated graded projective $R_k$-module, and write $\mathfrak{m}$ for the irrelevant ideal of $R_k$. Since $R_k^0=\F_p$, we have
\[\overline{P} \coloneqq P/\mathfrak{m}P\]
is a finite-dimensional graded $\F_p$-vector space. Recall that the graded Nakayama's lemma for positively graded rings tells us that if $M$ is a finitely generated graded module and $\mathfrak{m}M=M$, then $M=0$. In particular, we may choose a homogeneous basis
\[\overline{x}_1,\dots,\overline{x}_r\]
of the graded $\F_p$-vector space $P/\mathfrak{m}P$, with
$\deg(\overline{x}_i)=d_i$, and choose homogeneous lifts $x_i\in P$. These lifts determine a graded $R_k$-module homomorphism
\[\begin{tikzcd}
    f:\bigoplus_{i=1}^r R_k\langle-d_i\rangle \arrow[r] & P
\end{tikzcd}\]
which becomes an isomorphism modulo $\mathfrak{m}$. By graded Nakayama, $f$ is surjective. Since $P$ is projective, there is a splitting
\[ \bigoplus_{i=1}^r R_k\langle-d_i\rangle \simeq P\oplus Q.
\]
Reducing modulo $\mathfrak{m}$ gives $Q/\mathfrak{m}Q=0$, so another
application of graded Nakayama yields $Q=0$. Thus $P$ is graded free.
\end{proof}
We can now identify the heart of the Milnor weight structure, cf.~\cite[Lemma~2.2, Lemma~2.3]{DellAmbrogio:DerivedCategoryofGradedCommutativeRing}.
\begin{prop}\label{prop:Milnor-heart-graded-mod}
The assignment
\[\unit(n)[n] \longmapsto R_k \langle n \rangle\]
extends to an equivalence of additive symmetric monoidal categories
\[\begin{tikzcd}
    F:\add(\unit(n)[n] \mid n \in \Z) \arrow[r, "\sim"] & \mathsf{free^{gr}}(R_k),
\end{tikzcd}\]
where a morphism
\[\begin{tikzcd}
    f:\unit(i)[i] \arrow[r] & \unit(j)[j]
\end{tikzcd}\]
corresponding to a homogeneous element $\alpha \in R_k^{j-i}$ is sent to the graded $R_k$-module homomorphism
\[\begin{tikzcd}
    F(f):R_k \langle i \rangle \arrow[r] & R_k \langle j \rangle
\end{tikzcd}\]
given by multiplication by $\alpha$. In particular, $\add(\unit(n)[n] \mid n \in \Z)$ is idempotent complete, and hence it coincides with the heart of the Milnor weight structure $\mathcal{H}^{\mathsf{T}}$.
\end{prop}

\begin{proof}
The latter statement follows from the former by \Cref{lem:graded-projectives-Milnor-free}. For $i,j \in \Z$, we may identify
\[\Hom_{\DTM_\gm(k;\F_p)}(\unit(i)[i],\unit(j)[j]) \simeq R_k^{j-i}.\]
On the other hand,
\[\Hom_{\mathsf{mod^{gr}}(R_k)}(R_k\langle i\rangle, R_k\langle j\rangle) \simeq R_k^{j-i}, \]
where an element $\alpha\in R_k^{j-i}$ corresponds to the graded $R_k$-module homomorphism given by multiplication by $\alpha$. Since composition corresponds to multiplication in $R_k$, these identifications extend to a fully faithful additive functor $F$, which is essentially surjective by definition. Finally, it is clear that $F$ is symmetric monoidal.
\end{proof}
The preceding results yield equivalences of tt-categories
\[\K^b(\mathcal{H}^{\mathsf{T}}) \simeq \K^b(\mathsf{free^{gr}}(R_k))=\K^b(\mathsf{proj^{gr}}(R_k))= \D^{\mathsf{perf}}(R_k).\]
Recall from \Cref{theorem:gradedperfectcomplexes} that the graded comparison map with respect to the tensor invertible object $R_k \langle 1 \rangle \in \D^{\mathsf{perf}}(R_k)$
\[\begin{tikzcd}
    \varrho^h: \Spc(\D^{\mathrm{perf}}(R_k)) \arrow[r] & \Spec^h(R_k)
\end{tikzcd}\]
is a homeomorphism. Transporting this map across the above equivalence gives the following

\begin{cor}\label{cor:spectrum-target-Milnor-weight}
The graded comparison map with respect to the $\otimes$-invertible object
\[\left[ \unit(1)[1] \right] \coloneqq t(\unit(1)[1])\in\K^b(\mathcal{H}^{\mathsf{T}})\]
is a homeomorphism:
\[\begin{tikzcd}
    \varrho^h_{\mathcal{H}^{\mathsf{T}}}:\Spc(\K^b(\mathcal{H}^{\mathsf{T}})) \arrow[r, "\sim"] & \Spec^h(K^M_*(k)/p).
\end{tikzcd}\]
\end{cor}
We conclude by identifying the closed subset
\[\supp(\cone(\tau)) \subseteq \Spc(\DTM_\gm(k;\F_p))\]
with the spectrum we have just computed.
\begin{prop}\label{prop:description-supp-cone-tau}
The Milnor weight complex functor induces a homeomorphism onto its image
\[\begin{tikzcd}
    \Spc(t):\Spc(\K^b(\mathcal{H}^{\mathsf{T}})) \arrow[r, "\sim"] & \supp(\cone(\tau)) \subseteq \Spc(\DTM_\gm(k;\F_p)).
\end{tikzcd}\]
In particular, there is a homeomorphism
\[\supp(\cone(\tau)) \cong \Spec^h(K^M_*(k)/p).\]
\end{prop}
\begin{proof}
The latter statement follows from the former by \Cref{cor:spectrum-target-Milnor-weight}. For the first statement, recall that 
\[\im(\Spc(t))=\supp(\cone(\tau))\]
by \Cref{rem:supp-cone-tau}. Moreover, by the naturality of the comparison map, the following diagram commutes:
\[\begin{tikzcd}
    \Spc(\K^b(\mathcal{H}^{\mathsf{T}})) \arrow[r, "\varrho^h_{\mathcal{H}^{\mathsf{T}}}"] \arrow[d, swap, "\Spc(t)"] & \Spec^h(R_k) \arrow[d, "\phi"] \\
    \Spc(\DTM_\gm(k;\F_p)) \arrow[r, "\varrho^h_{\unit(1)[1]}"] & \Spec^h(R_k)
\end{tikzcd}\]
Here, $\phi$ is the homeomorphism induced by the identification of graded endomorphism rings with respect to $\unit(1)[1]$ and $t(\unit(1)[1])$ via $t$.
Since the upper horizontal map is a homeomorphism, $\Spc(t)$ is
injective, and its inverse on its image is given by
\[\begin{tikzcd}
    (\varrho_{\mathcal{H}^{\mathsf{T}}}^h)^{-1} \circ \phi^{-1} \circ \varrho_{\unit(1)[1]}^h \vert_{\supp(\cone(\tau))}: \supp(\cone(\tau)) \arrow[r] & \Spc(\K^b(\mathcal{H}^{\mathsf{T}})).
\end{tikzcd}\]
\end{proof}

\subsection{The topology}
We are now in a position to describe the spectrum of Tate motives as a set. 
\begin{cor}\label{cor:bijectivity-tate-comparison-map}
The comparison map
\[\begin{tikzcd}
    \varrho^h: \Spc(\DTM_\gm(k;\F_p)) \arrow[r] & \Spec^h((K^M_*(k)/p)[\tau])
\end{tikzcd}\]
is a bijection.
\end{cor}
\begin{proof}
The naturality of the comparison maps gives a commutative diagram
\[\begin{tikzcd}[ampersand replacement=\&,sep=scriptsize]
	{\Spc((\mathfrak{X} \otimes \DTM(k;\F_p))^\omega)} \&\& {\Spc(\DTM_\gm(k;\F_p))} \&\& {\Spc(\K^b(\mathcal{H}^{\mathsf{T}}))} \\
	\\
	{\Spec^h(R_k) \cong D(\tau)} \&\& {\Spec^h(R_k[\tau])} \&\& { V(\tau) \cong \Spec^h(R_k) }
	\arrow["{\Spc(L_\tau)}", hook, from=1-1, to=1-3]
	\arrow["{\varrho^h_{\tau}}"', from=1-1, to=3-1]
	\arrow["{\varrho^h}"', from=1-3, to=3-3]
	\arrow["{\Spc(t)}"', hook', from=1-5, to=1-3]
	\arrow["{\varrho^h_{\mathcal{H}^{\mathsf{T}}}}", from=1-5, to=3-5]
	\arrow[hook, from=3-1, to=3-3]
	\arrow[hook', from=3-5, to=3-3]
\end{tikzcd}\]
whose outer vertical maps are homeomorphisms by \Cref{cor:tau-open-tates} and \Cref{prop:description-supp-cone-tau}. The upper horizontal maps are injective and have disjoint jointly surjective images. Similarly, the two lower maps identify their sources with the disjoint, jointly covering subsets $D(\tau)$ and $V(\tau)$ of $\Spec^h(R_k[\tau])$. From this it is immediate that the middle map is a bijection.
\end{proof}
To determine the topology, it remains to show that the comparison map is a homeomorphism. We shall define
\[\mathcal{P}_0 \coloneqq (\varrho^h)^{-1}(\mathfrak{m}+ \langle\tau \rangle), \ \mathcal{P}_{\acute{e}t} \coloneqq (\varrho^h)^{-1}(\mathfrak{m}),\]
where $\mathfrak{m}=K^M_{>0}(k)/p$. When $p=2$, for every ordering $\alpha \in \Sper(k)$, let $\mathfrak{p}_\alpha \in \Spec^h(R_k)$ be the corresponding homogeneous prime and set
\[\mathcal{P}^{\alpha} \coloneqq (\varrho^h)^{-1}(\mathfrak{p}_\alpha+ \langle\tau \rangle), \ \mathcal{P}^{\alpha}_{\acute{e}t} \coloneqq (\varrho^h)^{-1}(\mathfrak{p}_\alpha).\]
Since $\varrho^h$ is a bijective spectral map, it suffices to prove that it lifts specialisations \cite[Proposition 15]{hochster1969prime}. All specialisations contained entirely in $D(\tau)$ or $V(\tau)$ already lift since the restrictions of $\varrho^h$ to the open and closed pieces are homeomorphisms. In particular, we have that 
\[\mathcal{P}^{\alpha} \rightsquigarrow \mathcal{P}_0, \ \mathcal{P}^{\alpha}_{\acute{e}t} \rightsquigarrow \mathcal{P}_{\acute{e}t}, \ \text{for every } \alpha \in \Sper(k).\]
It therefore remains only to lift the specialisations between $U(\cone(\tau))$ and $\supp(\cone(\tau))$. These will be obtained by base change to algebraically closed and real closed fields. 

\begin{remark}\label{rem:real-and-algclosed-tate-description}
The Balmer spectra of Tate motives over algebraically closed fields
\cite{Gallauer:TateSpectrum} and real closed fields
\cite{Balmer:ThreeRealArtinTates} were originally computed using
the representation-theoretic models for these motivic categories. The preceding arguments can be used to independently recover these computations via our ``motivic'' methods. Indeed, if $k$ is algebraically closed, then 
\[R_k[\tau]=\F_p[\tau],\]
and one may compute directly the generators of the tt-primes
\[\mathcal{P}_0=\ideal{0}, \ \mathcal{P}_{\acute{e}t}=\langle \cone(\tau) \rangle.\]
These are the only tt-primes: indeed, when $p=2$, there are no additional primes indexed by orderings since $\Sper(k)=\emptyset$. The proper inclusion 
\[\ideal{0} \subseteq \langle \cone(\tau) \rangle \]
therefore gives the two-point Sierpiński spectrum. Whenever $k$ is real closed and $p$ is odd, base changing to the algebraic closure $\overline{k}$ of $k$ is faithful (\Cref{prop:adoining-root-unity-faithful}) so one reduces to the above computation to find the spectrum in this case is also a Sierpiński space.

If $k$ is real closed and $p=2$, then $k$ admits a unique ordering $\alpha$ and 
\[R_k[\tau]=\F_2[\rho,\tau].\]
One may describe the generators of the tt-primes
\[\mathcal{P}_0=\ideal{0}, \ \mathcal{P}_{\acute{e}t}=\ideal{\cone(\tau)}, \ \mathcal{P}^{\alpha}=\ideal{\cone(\rho)}, \ \mathcal{P}^{\alpha}_{\acute{e}t}=\ideal{\cone(\rho), \cone(\tau)}.\]
From this description, the proper inclusions
\[\ideal{0} \subseteq \ideal{\cone(\tau)},\ideal{\cone(\rho)} \subseteq \ideal{\cone(\rho), \cone(\tau)}\]
determine the remaining specialisations between $U(\cone(\tau))$ and $\supp(\cone(\tau))$.
\[\begin{tikzcd}[ampersand replacement=\&, sep=small]
	\& \ideal{0} \& \\
	{\langle \cone(\rho) \rangle} \&\& {\langle \cone(\tau) \rangle} \\
	\& {\langle \cone(\rho), \cone(\tau)\rangle}
	\arrow[no head, from=1-2, to=2-1]
	\arrow[no head, from=1-2, to=2-3]
	\arrow[no head, from=2-1, to=3-2]
	\arrow[no head, from=2-3, to=3-2]
\end{tikzcd}\]
This recovers the computation of Balmer--Gallauer entirely within the
motivic category.
\end{remark}
\begin{prop}\label{prop:tate-specialisations}
The specialisations between the tt-primes in $U(\cone(\tau))$ and $\supp(\cone(\tau))$ are generated by the following:
\begin{itemize}
    \item $\mathcal{P}_{\acute{e}t} \rightsquigarrow \ideal{0}$,
    \item When $p=2$, $\mathcal{P}^{\alpha}_{\acute{e}t} \rightsquigarrow \mathcal{P}^{\alpha}$ for every $\alpha\in\Sper(k)$.
\end{itemize}
\end{prop}
\begin{proof}
Let $\overline{k}$ be an algebraic closure of $k$. Since
\[R_{\overline{k}}[\tau]\simeq\F_p[\tau],\]
the base-change homomorphism
\[\begin{tikzcd}
    R_k[\tau] \arrow[r] & \F_p[\tau]
\end{tikzcd}\]
induces the following map on homogeneous spectra:
\[(0) \longmapsto \mathfrak{m}; \ (\tau) \longmapsto \mathfrak{m}+(\tau).\]
 By naturality of the comparison map, the map induced by
base-change
\[\begin{tikzcd}
    \phi_{\overline{k}}:\DTM_\gm(k;\F_p) \arrow[r] & \DTM_\gm(\overline{k};\F_p).
\end{tikzcd}\]
therefore sends the two points of $\Spc(\DTM_\gm(\overline{k};\F_p))$
\[\ideal{0}, \ \ideal{\cone(\tau)}\]
to $\ideal{0}$ and $\mathcal{P}_{\acute{e}t}$, respectively. Hence, the specialisation
\[\ideal{\cone(\tau)} \rightsquigarrow \ideal{0},\]
pulls back to
\[\mathcal{P}_{\acute{e}t} \rightsquigarrow \ideal{0}.\]
Suppose now that $p=2$. Let $\alpha \in \Sper(k)$ and $k_\alpha$ be the corresponding real closure.  Writing $\rho_\alpha \coloneqq \{-1\} \in K^M_1(k_\alpha)/2$, we have 
\[R_{k_\alpha} \simeq \F_2[\rho_\alpha].\]
The base-change homomorphism
\[\begin{tikzcd}
    i_\alpha: R_k \arrow[r] & \F_2[\rho_\alpha]
\end{tikzcd}\]
is such that
\[\{a\} \longmapsto 
\begin{cases}
0 & \text{if } a>_\alpha 0; \\
\rho_\alpha & \text{if } a<_\alpha 0
\end{cases}\]
for each degree one pure symbol $\{a\}$. Consequently, the induced map on homogeneous spectra sends
\[(\rho_\alpha) \longmapsto \mathfrak{m}; \ (0) \longmapsto \mathfrak{p}_\alpha.\]
After adjoining $\tau$, the four primes in $\Spec^h(\F_2[\rho_\alpha, \tau])$ get mapped as follows:
\[\begin{aligned}
(\rho_\alpha,\tau) 
&\longmapsto \mathfrak{m} + (\tau);&
(\rho_\alpha)
&\longmapsto \mathfrak{m};\\
(\tau) &\longmapsto \mathfrak{p}_\alpha+ (\tau);&
(0) &\longmapsto \mathfrak{p}_\alpha.
\end{aligned}\]
By naturality of the comparison map, the map induced by base-change 
\[\begin{tikzcd}
    \phi_\alpha:\DTM_\gm(k;\F_2) \arrow[r] & \DTM_\gm(k_\alpha;\F_2)
\end{tikzcd}\]
therefore identifies the pull backs of the four points in $\Spc(\DTM_\gm(k_\alpha;\F_2))$ as follows:
\[\begin{aligned}
\ideal{0} &\longmapsto \ideal{0}; \\
\langle\cone(\tau)\rangle
&\longmapsto \mathcal{P}_{\acute{e}t}; \\
\langle\cone(\rho_\alpha)\rangle
&\longmapsto\mathcal{P}^{\alpha}; \\
\langle\cone(\rho_\alpha),\cone(\tau)\rangle
&\longmapsto\mathcal{P}^{\alpha}_{\acute{e}t}.
\end{aligned}\]
In particular, the specialisation
\[\langle\cone(\rho_\alpha),\cone(\tau)\rangle
\rightsquigarrow \langle\cone(\rho_\alpha)\rangle\]
pulls back to 
\[\mathcal{P}^{\alpha}_{\acute{e}t} \rightsquigarrow \mathcal{P}^{\alpha}.\]
Finally, we note that there are no other specialisations, since any others produce the corresponding ones in $\Spec^h((K^M_*(k)/p)[\tau])$, and we have already accounted for all specialisations in the latter.
\end{proof}
\begin{theorem}\label{thm:spectrum-tate-root-unity}
Let $k$ be a field containing a primitive $p^{th}$ root of unity, then the comparison map
\[\begin{tikzcd}
    \varrho^h:\Spc(\DTM_\gm(k;\F_p)) \arrow[r] & \Spec^h((K^M_*(k)/p)[\tau])
\end{tikzcd}\]
is a homeomorphism.
\end{theorem}
\begin{proof}
By \Cref{cor:bijectivity-tate-comparison-map}, we know that the comparison map $\varrho^h$ is bijective, so it suffices to prove that it lifts specialisations. By \Cref{cor:tau-open-tates} and \Cref{prop:description-supp-cone-tau}, we know that the restrictions of $\varrho^h$ to $U(\cone(\tau))$ and
$\supp(\cone(\tau))$ are homeomorphisms, so all specialisations lying
entirely within $D(\tau)$ or $V(\tau)$ lift. By \Cref{cor:hspectrum-MilnorKtau}, the only remaining specialisations are
\[\mathfrak{m} \rightsquigarrow \mathfrak{m}+ \langle \tau \rangle \]
and, when $p=2$,
\[\mathfrak{p}_\alpha
\rightsquigarrow
\mathfrak{p}_\alpha+ \langle \tau\rangle,
\ \text{for all } \alpha\in\Sper(k).\]
These lift by \Cref{prop:tate-specialisations}, proving the result.
\end{proof}
We may explicitly describe the generators of the prime tt-ideals and the topology. We first recall the following general tt-geometric result.
\begin{lemma}\label{lem:prime-generated-by-cones}
Let $\mathcal{K}$ be an essentially small rigid tt-category, let $R$ be the multi-graded central ring associated with a family of tensor-invertible objects, and suppose that the comparison map
\[\begin{tikzcd}
    \varrho^h:\Spc(\mathcal{K}) \arrow[r] & \Spec^h(R)
\end{tikzcd}\]
is a homeomorphism. Then for every $\mathcal{P} \in\Spc(\mathcal{K})$ we have 
\[\mathcal{P} =\langle \cone(f) \mid f\in R \text{ homogeneous and } f\notin\varrho^h(\mathcal{P})\rangle.\]
\end{lemma}
\begin{proof}
Let
\[\mathcal{I} \coloneqq \langle \cone(f) \mid f\in R \text{ homogeneous and } f\notin\varrho^h(\mathcal{P})\rangle.\]
By definition of the comparison map, $\mathcal{I} \subseteq \mathcal{P}$. Moreover, for any $\mathcal{Q} \in\Spc(\mathcal{K})$, we have $\mathcal{I} \subseteq \mathcal{Q}$ if and only if for every homogeneous element $f \in R$, if $f \not \in \varrho^h(\mathcal{P})$, then $f \not \in  \varrho^h(\mathcal{Q})$. Equivalently, $\mathcal{I} \subseteq \mathcal{Q}$ if and only if $\varrho^h(\mathcal{Q})\subseteq\varrho^h(\mathcal{P})$. Since $\varrho^h$ is a homeomorphism, it preserves specialisations, which means that $\mathcal{Q}$ specialises to $\mathcal{P}$, therefore $\mathcal{P} \subseteq \mathcal{Q}$. This shows that $\mathcal{I}$ and $\mathcal{P}$ have the same supports, so they are the same (prime) tt-ideal by \Cref{theorem:classificationofttideals}.
\end{proof}
For each $\alpha \in \Sper(k)$, set
\[S^\alpha \coloneqq \{\cone(\delta) \mid \delta\in K^M_1(k)/2,\  \delta \notin\mathfrak{p}_\alpha\}, \ S^\alpha_\tau= S^\alpha \cup \{\cone(\tau)\}.\]

\begin{cor}\label{cor:explicit-tate-generators-topology}
If $p$ is an odd prime or $k$ is not formally real, $\Spc(\DTM_\gm(k;\F_p))$ is the Sierpiński space:
\[\begin{tikzcd}[ampersand replacement=\&,sep=scriptsize]
	{\ideal{0}} \\
	\\
	{\mathcal{P}_{\acute{e}t}}
	\arrow[no head, from=1-1, to=3-1]
\end{tikzcd}\]
where $\mathcal{P}_{\acute{e}t}=\ideal{\cone(\tau)}$.

If $p=2$ and $k$ is formally real, then $\Spc(\DTM_\gm(k;\F_2))$ is described as follows:
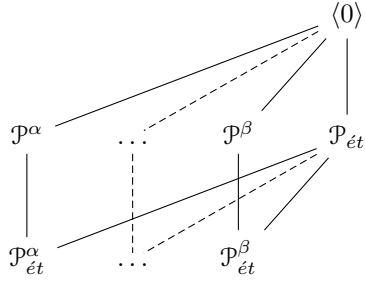
\[\begin{tikzcd}[ampersand replacement=\&,sep=scriptsize]
	\&\&\& {\ideal{0}} \\
	\\
	{\mathcal{P}^\alpha} \& \dots \& {\mathcal{P}^\beta} \& {\mathcal{P}_{\acute{e}t}} \\
	\\
	{\mathcal{P}^\alpha_{\acute{e}t}} \& \dots \& {\mathcal{P}^\beta_{\acute{e}t}}
	\arrow[no head, from=1-4, to=3-1]
	\arrow[dashed, no head, from=1-4, to=3-2]
	\arrow[no head, from=1-4, to=3-3]
	\arrow[no head, from=1-4, to=3-4]
	\arrow[no head, from=3-1, to=5-1]
	\arrow[dashed, no head, from=3-2, to=5-2]
	\arrow[no head, from=3-3, to=5-3]
	\arrow[no head, from=3-4, to=5-1]
	\arrow[dashed, no head, from=3-4, to=5-2]
	\arrow[no head, from=3-4, to=5-3]
\end{tikzcd}\]
where a subset $Z$ is closed if and only if it is specialisation closed and 
\[W_1(Z)=\{\alpha \in \Sper(k) \mid \mathcal{P}^\alpha \in Z\}, \ W_2(Z)=\{\alpha \in \Sper(k) \mid \mathcal{P}^\alpha_{\acute{e}t} \in Z\}\]
are both closed subsets of $\Sper(k)$. Moreover, the generators of the tt-primes are described as follows
\[\mathcal{P}_{\acute{e}t}= \ideal{\cone(\tau)}; \, \mathcal{P}^\alpha=\ideal{S^\alpha}; \, \mathcal{P}^\alpha_{\acute{e}t}=\ideal{S^\alpha_\tau}\]
\end{cor}
\begin{proof}
The description of the topology follows from \Cref{cor:hspectrum-MilnorKtau} and \Cref{thm:spectrum-tate-root-unity} by transporting the topology of $\Spec^h(R_k[\tau])$ along the homeomorphism $\varrho^h$. It remains to identify the generators of the corresponding tt-primes. First, it is clear that the unique closed point is $\ideal{0}$, since $\DTM_\gm(k;\F_p)$ is rigid. Moreover, by \Cref{prop:tau-localisation-etale} and \Cref{cor:tau-open-tates}, it is also clear that $\mathcal{P}_{\acute{e}t}=\langle \cone(\tau) \rangle$. Finally, by \Cref{lem:prime-generated-by-cones}, we have that 
\[\langle S^\alpha \rangle \subseteq \mathcal{P}^\alpha=\langle \cone(f) \mid f \in R_k[\tau] \text{ homogeneous and } f \notin \mathfrak{p}_\alpha +\langle \tau \rangle \rangle,  \]
\[\ideal{S^\alpha_\tau} \subseteq \mathcal{P}^\alpha_{\acute{e}t}=\langle \cone(f) \mid f \in R_k[\tau] \text{ homogeneous and } f \notin \mathfrak{p}_\alpha \rangle,\]
so we need to verify the reverse inclusions. Since $\DTM_\gm(k;\F_p)$ is rigid, $\langle S^\alpha_\tau \rangle$ is equal to the intersection of all the tt-primes containing it. It is clear that $\mathcal{P}^\alpha_{\acute{e}t}$ is the only prime containing it, and this proves the equality. Similarly, the only primes containing  $\ideal{S^\alpha}$ are $\mathcal{P}^\alpha$ and $\mathcal{P}^\alpha_{\acute{e}t}$. Thus, taking the intersection proves the remaining equality.
\end{proof}

\subsection{Removing the root of unity assumption}
We now remove the assumption that $k$ contains a primitive $p^{th}$ root
of unity. Since $\operatorname{char}(k) \neq p$, we only need to consider the case where $p$ is an odd prime. Let
\[k' \coloneqq k(\zeta_p)\]
be the finite separable extension obtained by adjoining a primitive $p^{th}$ root of unity $\zeta_p$, and let 
\[\begin{tikzcd}
    f:\Spec(k') \arrow[r] & \Spec(k)
\end{tikzcd}\]
be the associated finite étale morphism.
\begin{prop}\label{prop:adoining-root-unity-faithful}
The base-change tt-functor
\[\begin{tikzcd}
    f^*:\DM_\gm(k;\F_p) \arrow[r] & \DM_\gm(k';\F_p)
\end{tikzcd}\]
is faithful.
\end{prop}
\begin{proof}
The argument is analogous to the proof of \cite[Proposition 4.4]{gallauer2021note}. Since $[k':k]=d$ must divide $p-1$, it is invertible in $\F_p$. Moreover, the extension $k'/k$ is finite separable, thus we have the following adjunction.
\[\begin{tikzcd}[ampersand replacement=\&, sep=small]
	{\DM_\gm(k;\F_p)} \\
	\dashv \\
	{\DM_\gm(k';\F_p)}
	\arrow["{f^*}"', shift right=3, from=1-1, to=3-1]
	\arrow["{f_*}"', shift right=3, from=3-1, to=1-1]
\end{tikzcd}\]
The projection formula gives
\[f_*f^*(X)\simeq M(\Spec(k')) \otimes X.\]
Since the Euler characteristic of $M(\Spec(k'))$ is the degree of the extension $k'/k$, which is invertible in $\F_p$, we can conclude that $M(\Spec(k')) \otimes-$ is faithful, whence $f^*$ is too.
\end{proof}
With this result, we can describe the homogeneous spectrum of $H^{*,*}_{\mathcal{M}}(k;\F_p)$. 
\begin{prop}\label{prop:hSpec-motcohom-no-root-unity}
Base change 
\[\begin{tikzcd}
    f^*: H^{*,*}_{\mathcal{M}}(k;\F_p) \arrow[r] & H^{*,*}_{\mathcal{M}}(k';\F_p)\simeq(K_*^M(k')/p)[\tau]
\end{tikzcd}\]
induces a homeomorphism
\[\begin{tikzcd}
    \phi:\Spec^h((K_*^M(k')/p)[\tau]) \arrow[r, "\sim"] & \Spec^h(H^{*,*}_{\mathcal{M}}(k;\F_p)).
\end{tikzcd}\]
In particular, $\Spec^h(H^{*,*}_{\mathcal{M}}(k;\F_p))$ is the Sierpiński space.
\end{prop}
\begin{proof}
Motivic cohomology with $\F_p$-coefficients of a field is concentrated in bidegrees $(a)[b]$ satisfying $0 \leq b \leq a$. For a homogeneous element $\alpha \in H^{b,a}_{\mathcal{M}}(k;\F_p)$ with $0 < b \leq a$, we have
\[f^*(\alpha)\in (K_b^M(k')/p)\tau^{a-b}.\]
Since every element of positive degree in $K_*^M(k')/p$ is nilpotent, $f^*(\alpha)$ is so. The faithfulness of $f^*$ from \Cref{prop:adoining-root-unity-faithful} then implies that
$\alpha$ is nilpotent. It remains to study what happens when $b=0$ and $a>0$. Let $k^s$ be a separable closure of $k$ containing $k'$, and choose a primitive $p^{th}$ root of unity $\zeta_p\in k'$. For a field $E$, we shall denote by $\Gamma_E$ its absolute Galois group. By the Beilinson--Lichtenbaum conjecture we obtain isomorphisms
\[H^{0,a}_{\mathcal{M}}(k;\F_p) \simeq H^0_{\acute{e}t}(k,\mu_p^{\otimes a}) \simeq (\mu_p^{\otimes a}(k^s))^{\Gamma_k},\]
where the action of $\Gamma_k$ on $\mu_p(k^s)$ is given by the mod $p$ cyclotomic character
\[\begin{tikzcd}
    \chi_p:\Gamma_k \arrow[r] & \F_p^\times.
\end{tikzcd}\]
From this it follows that
\[H^{0,a}_{\mathcal{M}}(k;\F_p) \simeq
\begin{cases}
\F_p,& d \mid a,\\
0,& d \nmid a.
\end{cases}\]
The element $\zeta_p^{\otimes d} \in (\mu_p^{\otimes d}(k^s))^{\Gamma_k}$ therefore determines a class
\[\beta\in H^{0,d}_{\mathcal{M}}(k;\F_p)\]
which generates $H^{0,d}_{\mathcal{M}}(k;\F_p)$ as an abelian group, such that $f^*(\beta)=\tau^d$. For every $r\geq0$,
\[f^*(\beta^r)=\tau^{dr}\neq0,\]
so $\beta^r\neq0$. Since $H^{0,dr}_{\mathcal{M}}(k;\F_p)$ is
one-dimensional, it is generated by $\beta^r$. Therefore
\[H^{0,*}_{\mathcal{M}}(k;\F_p)
\simeq \F_p[\beta],
\ |\beta|=(0,d). \]
Let $\sqrt{(0)}$ denote the nilradical of the bigraded ring $H^{*,*}_{\mathcal{M}}(k;\F_p)$. The preceding work gives us
\[H^{*,*}_{\mathcal{M}}(k;\F_p)/\sqrt{(0)}
\simeq \F_p[\beta]. \]
Similarly,
\[(K_*^M(k')/p)[\tau]/\sqrt{(0)} \simeq \F_p[\tau].\]
Under these identifications, base change descends to
\[\begin{tikzcd}
    f^*/\sqrt{(0)}:\F_p[\beta] \arrow[r] & \F_p[\tau]; \ \beta \longmapsto \tau^d,
\end{tikzcd}\]
which clearly induces a homeomorphism on homogeneous spectra. Since quotienting by the nilradical does not affect homogeneous spectra, the result is now immediate.
\end{proof}
\begin{cor}\label{cor:tate-spectrum-no-root-unity}
Base change and the naturality of the comparison map give rise to the following commutative diagram,
\[\begin{tikzcd}[ampersand replacement=\&]
	{\Spc(\DTM_\gm(k';\F_p))} \&\& {\Spc(\DTM_\gm(k;\F_p))} \\
	{\Spec^h((K_*^M(k')/p)[\tau])} \&\& {\Spec^h(H^{*,*}_{\mathcal{M}}(k;\F_p))}
	\arrow["{\Spc(f^*)}", from=1-1, to=1-3]
	\arrow["{\varrho^h_{k'}}"', from=1-1, to=2-1]
	\arrow["{\varrho^h_k}", from=1-3, to=2-3]
	\arrow["\phi", from=2-1, to=2-3]
\end{tikzcd}\]
where every arrow is a homeomorphism.
\end{cor}
\begin{proof}
We already know that $\varrho^h_{k'}$ and $\phi$ are homeomorphisms by \Cref{thm:spectrum-tate-root-unity} and \Cref{prop:hSpec-motcohom-no-root-unity}, respectively. So, it suffices to show that $\Phi \coloneqq \Spc(f^*)$ is a homeomorphism. By \Cref{prop:adoining-root-unity-faithful}, we know that $f^*$ is faithful, hence $\Phi$ is surjective. By \Cref{cor:explicit-tate-generators-topology}, we know that $\Spc(\DTM_\gm(k';\F_p))$ is the two-point Sierpiński space with points
\[\ideal{0} \subseteq \langle\cone(\tau)\rangle.\]
Their images under $\Phi$ are 
\[\Phi(\ideal{0})=\ideal{0}, \, \Phi(\langle \cone(\tau) \rangle )=:\mathcal{Q}.\]
The first equality follows from the faithfulness of
$f^*$. It remains to show that $\mathcal{Q} \neq 0$. This follows from the fact that $\cone(\beta) \in \mathcal{Q}$ since $f^*(\beta)=\tau^{\otimes d}$, and $\cone(\tau^{\otimes d}) \in \langle \cone(\tau) \rangle$. 
\end{proof}
\begin{remark}
Note that the tt-prime $\mathcal{Q}$ is also identified with the kernel of étale realisation. Indeed, when $p$ is an odd prime $\DTM_\et^\vee(k;\F_p)$ is a local tt-category, so the kernel in question is prime, and it is non-trivial since it contains $\cone(\beta)$.
\end{remark}
We have now proved the main result of this section.
\begin{proof}[Proof of \Cref{thm:tate-comparison}]
Follows immediately from \Cref{thm:spectrum-tate-root-unity} and \Cref{cor:tate-spectrum-no-root-unity}.
\end{proof}

\subsection{Integral coefficients}
Let $e$ be the exponential characteristic of $k$. The connectivity of $\DM_\gm(k; \Z[1/e])$ implies that the comparison map
\[\begin{tikzcd}
    \varrho:\Spc(\DM_\gm(k; \Z[1/e])) \arrow[r, twoheadrightarrow] & \Spec(\Z[1/e])
\end{tikzcd}\]
is surjective, and hence it provides a stratification of the spectrum $\DM_\gm(k; \Z[1/e])$ into its rational and finite-coefficient fibres by \cite[Corollary A.4, Lemma 5.4]{Gallauer:TateSpectrum} and \cite[Theorem 1.7]{Balmer:Conservative}. The rational fibre is given by
\[\varrho^{-1}((0)) \cong \Spc(\DM_\gm(k;\Q)),\]
and the stratum over a prime $p \neq e$ is
\[\varrho^{-1}((p))
=\supp(\cone(p)).\]
Moreover, if we denote by
\[\begin{tikzcd}
    cc_p^*:\DM_\gm(k;\Z[1/e]) \arrow[r] & \DM_\gm(k;\F_p)
\end{tikzcd}\]
the change of coefficients tt-functor, then
\[\supp(\cone(p)) = \im(\Spc(cc_p^*)).\]
In \cite[Theorem 1]{gallauer2021note}, Gallauer showed that for $\Q$-coefficients, the existence of the abelian category of mixed motives over $k$ would imply that $\Spc(\DM_{\gm}(k;\Q))=\{\ideal{0}\}$. In this section, we shall restrict to the tt-category of integral geometric Tate motives
\[\DTM_\gm(k;\Z[1/e])\]
and work under the following weaker hypothesis.
\begin{hypothesis}\label{hyp:rational-tate-spectrum}
$\Spc(\DTM_\gm(k;\Q))=\{\ideal{0}\}$.
\end{hypothesis}
This hypothesis is known to hold whenever $k$
satisfies the Peter--Gallauer vanishing condition
\[H^{m,n}_{\mathcal{M}}(k; \Q)=0, \ \text{for } m\leq 0 <n \text{ or }
n\geq m\geq 2.
\]
Indeed, under these assumptions $\Spc(\DTM_\gm(k;\Q))$ is homeomorphic to $\Spec(\Q)$; see \cite[Proposition 6.9]{Gallauer:TateSpectrum}. The Peter--Gallauer vanishing condition is known to hold for algebraic extensions of $\Q$, algebraic extensions of finite fields, and algebraic extensions of global fields of positive characteristic; see \cite[Remark 6.8]{Gallauer:TateSpectrum}. 

Let
\[\DTM_\et^\vee(k;\Z[1/e])\]
denote the tt-category of rigid étale Tate motives. We first record the integral analogue of the finite-coefficient étale computation.
\begin{prop}
\label{prop:integral-etale-comparison}
Under \Cref{hyp:rational-tate-spectrum}, the graded
comparison map induces a homeomorphism
\[\begin{tikzcd}
    \varrho_{\et}^h:\Spc(\DTM_\et^\vee(k;\Z[1/e])) \arrow[r, "\simeq"] & \Spec^h(K^M_*(k)[1/e]).
\end{tikzcd}\]
\end{prop}
\begin{proof}
This is basically the content of \cite[Theorem 4.3]{gallauer2021note}. One just notes that the equivalence
\[\begin{tikzcd}
    \DM(k;\Q) \arrow[r, "\sim"] & \DM_\et(k;\Q)
\end{tikzcd}\]
clearly restricts to Tate motives, and that in \emph{loc. cit.} all one needs is to assume is that the $\Q$-coefficients spectrum is a singleton, i.e. the analogue of \Cref{hyp:rational-tate-spectrum}.
\end{proof}
We now turn to integral geometric Tate motives. Let
\[\begin{tikzcd}
    \varrho:\Spc(\DTM_\gm(k;\Z[1/e])) \arrow[r, twoheadrightarrow] & \Spec(\Z[1/e])
\end{tikzcd}\]
denote the comparison map, and for every prime $p \neq e$, denote by
\[\begin{tikzcd}
    cc_p^*:\DTM_\gm(k;\Z[1/e]) \arrow[r] & \DTM_\gm(k;\F_p)
\end{tikzcd}\]
the change of coefficients. Finally, denote
\[\begin{tikzcd}
    \mathcal{Q} \coloneqq \ker(\DTM_\gm(k;\Z[1/e]) \arrow[r] & \left((\Z[1/e]\setminus\{0\})^{-1}\DTM_\gm(k;\Z[1/e])\right)^\natural \simeq \DTM_\gm(k;\Q)).
\end{tikzcd}\]

\begin{theorem}\label{thm:integral-tate-spectrum}
Under \Cref{hyp:rational-tate-spectrum}, the following holds.
\begin{enumerate}
\item The fibre over the generic point is
\[\varrho^{-1}((0))=\{\mathcal{Q}\}.\]
\item For every prime $p \neq e$, change of coefficients induces a
homeomorphism onto its image
\[\begin{tikzcd}
    \Spc(cc_p^*):\Spc(\DTM_\gm(k;\F_p)) \arrow[r, "\simeq"] & \supp(\cone(p)) \subseteq \Spc(\DTM_\gm(k;\Z[1/e])).
\end{tikzcd}\]
In particular, there is a decomposition of sets
\[\Spc(\DTM_\gm(k;\Z[1/e]))=\{\mathcal{Q}\}\sqcup\bigsqcup_{p \neq e}\supp(\cone(p)).\]

\item The closure of $\mathcal{Q}$ is
\[\overline{\{\mathcal{Q}\}}=\{\mathcal{Q}\} \cup \bigsqcup_{p \neq e}\{\mathcal{P}_{p,0},\mathcal{P}_{p,1}\},\]
where
\[\mathcal{P}_{p,0} \coloneqq (cc_p^*)^{-1}(\mathcal{P}_0), \, \mathcal{P}_{p,1} \coloneqq (cc_p^*)^{-1}(\mathcal{P}_{\acute{e}t})\]

\item A subset $Z \subseteq \Spc(\DTM_\gm(k;\Z[1/e]))$ is closed if and only if
it is specialisation-closed and satisfies one of the following
conditions:
\begin{enumerate}
\item If $\mathcal{Q} \in Z$, then
\[Z=\overline{\{\mathcal{Q}\}} \cup W,\]
where $W=\emptyset$ if $e=2$, otherwise it is a closed subset of $ \supp(\cone(2))$. 
\item If $\mathcal{Q} \notin Z$, then
\[Z=Z_{p_1} \cup \cdots \cup Z_{p_n},\]
where $p_1,\dots,p_n \neq e$ are distinct primes and 
\[Z_{p_j} \subseteq \supp(\cone(p_j)), \, j=1,\dots,n\]
are closed.
\end{enumerate}
\end{enumerate}
\end{theorem}
\begin{proof}
Statement (1) follows immediately from \cite[Corollary A.4, Lemma 5.4]{Gallauer:TateSpectrum} and \Cref{hyp:rational-tate-spectrum}. For Statement (2), \cite[Theorem 1.7]{Balmer:Conservative} gives
\[\im(\Spc(cc_p^*))=\supp(\cone(p)).\]
It suffices to show that $\Spc(cc_p^*)$ is injective and lifts specialisations. We first treat the case where $p$ is odd. We know that
\[\Spc(\DTM_\gm(k;\F_p)) = \{\ideal{0},\mathcal{P}_{\acute{e}t}\}; \, \ideal{0} \subsetneq \mathcal{P}_{\acute{e}t},\]
where $\mathcal{P}_{\acute{e}t}$ is the kernel of étale realization. Set
\[\mathcal{P}_{p,0} \coloneqq (cc_p^*)^{-1}(\ideal{0}), \ \mathcal{P}_{p,1} \coloneqq (cc_p^*)^{-1}(\mathcal{P}_{\acute{e}t})\]
The inclusion
\[
\mathcal{P}_{p,0}\subseteq \mathcal{P}_{p,1}
\]
follows from $\ideal{0} \subseteq \mathcal{P}_{\acute{e}t}$. To prove that it
is strict, consider $\cone(\beta)\in\mathcal{P}_{\acute{e}t}$, where $\beta=\tau$ when $k$ contains a primitive $p^{th}$ root of unity, and otherwise $\beta: \unit \longrightarrow \unit(d)$ is the class constructed in \Cref{prop:hSpec-motcohom-no-root-unity}. Set
\[X \coloneqq cc_{p*}(\cone(\beta)),\]
where $cc_{p_*}$ is the right adjoint of $cc_p^*$. We claim that
\[\Re_{\acute{e}t,p}(cc_p^*(X))=0,\]
so that in particular $X \in \mathcal{P}_{p,1}$. Indeed, étale realisation is compatible with extension and restriction of coefficients, that is, the following diagram commutes:
\[\begin{tikzcd}[ampersand replacement=\&]
	{\DTM_\gm(k;\F_p)} \&\& {\DTM_\gm(k;\Z[1/e])} \&\& {\DTM_\gm(k;\F_p)} \\
	{\DTM_\et(k;\F_p)^\vee} \&\& {\DTM_\et(k;\Z[1/e])^\vee} \&\& {\DTM_\et(k;\F_p)^\vee}
	\arrow["{cc_{p*}}", from=1-1, to=1-3]
	\arrow["{\Re_{\acute{e}t,p}}"', from=1-1, to=2-1]
	\arrow["{cc^*_{p}}", from=1-3, to=1-5]
	\arrow["{\Re_{\acute{e}t,\Z[1/e]}}"', from=1-3, to=2-3]
	\arrow["{\Re_{\acute{e}t,p}}", from=1-5, to=2-5]
	\arrow["{cc_{\et, p*}}", from=2-1, to=2-3]
	\arrow["{cc_{\acute{et},p}^*}", from=2-3, to=2-5]
\end{tikzcd}\]
Using this, 
\[\Re_{\acute{e}t,p}(cc_p^*(X)) \simeq cc_{\acute{e}t,p}^*\left(cc_{\acute{e}t, p*} \left(\Re_{\acute{e}t,p}\left(\cone(\beta \right)\right)\right) =0,\]
since $\Re_{\acute{e}t,p}(\cone(\beta))=0$.
Finally, observe that $X \not \in \mathcal{P}_{p,0}$ because $cc_p^*(X) \neq 0$: otherwise, by adjunction,
\[\Hom(X,X) \simeq \Hom(cc_p^*X, \cone(\beta)) =0,\]
implying $X=0$, which is not the case since $cc_{p*}$ is a conservative right adjoint (cf. \cite[Section 3]{Gallauer:TateSpectrum}) and $\cone(\beta) \neq 0$.  We conclude that
\[\mathcal{P}_{p,0} \subsetneq \mathcal{P}_{p,1},\]
which proves (2) in the case where $p$ is an odd prime. When $p=2$, set
\[\begin{aligned}
\mathcal{P}_{2,0} & \coloneqq (cc_2^*)^{-1}(\ideal{0}), &
\mathcal{P}_{2,1} & \coloneqq (cc_2^*)^{-1}(\mathcal{P}_{\acute{e}t}); \\
\mathcal{P}_{2,1,\alpha} & \coloneqq (cc_2^*)^{-1}(\mathcal{P}^\alpha), &
\mathcal{P}_{2,2,\alpha} & \coloneqq (cc_2^*)^{-1}(\mathcal{P}^\alpha_{\acute{e}t});
\end{aligned}\]
where $\alpha \in \Sper(k)$. First, note that by \cite[Lemma B.6]{Gallauer:TateSpectrum},
\[cc_2^*cc_{2*}(\cone(\tau)) \simeq \cone(\tau)\oplus\cone(\tau)[1].\]
In particular, the object
$cc_{2*}(\cone(\tau))$ belongs to
$\mathcal{P}_{2,1}$ but not to $\mathcal{P}_{2,0}$. Hence
\[\mathcal{P}_{2,0}\subsetneq\mathcal{P}_{2,1}.\]
The same object belongs to $\mathcal{P}_{2,2,\alpha}$ but not to
$\mathcal{P}_{2,1,\alpha}$, because
\[\cone(\tau) \in \mathcal{P}^\alpha_{\acute{e}t} \ \text{but} \ \cone(\tau) \notin \mathcal{P}^\alpha,\]
thus
\[\mathcal{P}_{2,1, \alpha} \subsetneq \mathcal{P}_{2,2, \alpha}. \]
Next, recall from \Cref{cor:explicit-tate-generators-topology} that
$\mathcal{P}^{\alpha}$ is generated by 
\[S^\alpha \coloneqq \{\cone(\delta) \mid \delta\in K^M_1(k)/2,\  \delta \notin\mathfrak{p}_\alpha\}.\]
Since these generators all come from $\DTM_{\gm}(k;\Z[1/e])$, being cones of morphisms corresponding to homogeneous elements in $K^M_*(k)[1/e]$, we immediately get that 
\[\mathcal{P}_{2,0} \subsetneq \mathcal{P}_{2,1,\alpha}.\]
Similarly, $\mathcal{P}^{\alpha}_{\acute{e}t}$ is generated by $S^\alpha \cup \{\cone(\tau)\}$, from which the proper inclusion
\[\mathcal{P}_{2,1} \subsetneq \mathcal{P}_{2,2,\alpha}\]
is also clear. By inspection, there are no other specialisations among these points, which concludes the proof of Statement (2). For Statement (3), consider the graded comparison map with respect to the invertible object $\unit(1)[1]$,
\[\begin{tikzcd}
    \varrho^h_{diag}:\Spc(\DTM_{\gm}(k; \Z[1/e])) \arrow[r] & \Spec^h(K^M_*(k)[1/e]).
\end{tikzcd}\]
Arguing as in \cite[Lemma 8.5]{Gallauer:TateSpectrum} combined with \Cref{prop:integral-etale-comparison}, we have that étale realisation
\[\begin{tikzcd}
    \Re_{\acute{e}t}:\DTM_\gm(k; \Z[1/e]) \arrow[r] & \DTM_\et(k; \Z[1/e])^\vee
\end{tikzcd}\]
induces, on spectra, a section of $\varrho^h_{diag}$:
\[\begin{tikzcd}
    \Spec^h(K^M_*(k)[1/e]) \cong \Spc(\DTM_\et^\vee(k;\Z[1/e])) \arrow[r, "\Spc(\Re_{\acute{e}t})"] & \Spc(\DTM_{\gm}(k;\Z[1/e])) \arrow[r, "\varrho^h_{diag}"] & \Spec^h(K^M_*(k)[1/e]).
\end{tikzcd}\]
In particular, $\Spc(\Re_{\acute{e}t})$ is a homeomorphism onto its image. By \cite{thornton2016homogeneous}, we have a description of  $\Spec^h(K^M_*(k)[1/e])$, which by inspection gives us that 
\[\mathcal{Q} \rightsquigarrow \mathcal{P}_{p,1}\]
for all primes $p \neq e$, and  
\[\mathcal{Q} \not \rightsquigarrow \mathcal{P}_{2,i,\alpha}\]
for all $\alpha \in \Sper(k)$ and $i=1,2$. From this, Statement (3) is immediate.

Finally, for Statement (4), if $Z$ is a closed subset, then it must be specialisation closed. First, if it contains $\mathcal{Q}$, then it must contain its closure. In particular, if $e=2$, then 
\[\overline{\{\mathcal{Q}\}}=\Spc(\DTM_{\gm}(k; \Z[1/e]))=Z.\] 
 Otherwise, if $e \neq 2$, then
\[\Spc(\DTM_{\gm}(k; \Z[1/e]))=\overline{\{\mathcal{Q}\}} \cup \supp(\cone(2)),\]
from which it follows that if $\mathcal{Q} \in Z$, then
\[Z=(\overline{\{\mathcal{Q}\}} \cap Z) \cup (\supp(\cone(2))\cap Z)=\overline{\{\mathcal{Q}\}} \cup W,\]
where $W=\supp(\cone(2))\cap Z$ is a closed subset of $\supp(\cone(2))$. Conversely, any such subset must be closed since it is a finite union of closed subsets. 

Suppose now that $\mathcal{Q} \not \in Z$, then
\[Z \subseteq \bigcup_{p \neq e} \supp(\cone(p)).\]
We claim that $Z \cap \supp(\cone(p))= \emptyset$ for all but finitely many primes $p$. Since a basis of closed subsets is given by $\supp(M)$ where $M$ ranges over objects in $\DTM_{\gm}(k; \Z[1/e])$, it suffices to prove the claim for $Z=\supp(M)$. Observe that if
\[\supp(M) \cap \supp(\cone(p))\neq  \emptyset\]
for infinitely many prime $p$, then $\supp(M)$ must contain infinitely many $\mathcal{P}_{0,p}$. Arguing as in the proof of \cite[Proposition 9.2]{Gallauer:TateSpectrum}, this implies that the mod $p$ motivic cohomology of $M$ is non-trivial for infinitely many primes, so by Proposition 5.8 in \textit{loc.cit.}, $M$ has non-zero rational motivic cohomology. This forces $\mathcal{Q} \in \supp(M)$ which contradicts our assumption. 

From this claim, we conclude that 
\[Z=\bigcup_{p \neq e} (\supp(\cone(p)) \cap Z)= Z_{p_1}\cup\cdots\cup Z_{p_n}\]
is a finite union of closed subsets 
\[Z_{p_j}\subseteq\supp(\cone(p_j)), \ j=1,\dots, n,\]
where $p_1,\dots,p_n \neq e$ are distinct primes. Conversely, any such subset is closed since it is a finite union of closed subsets. 
\end{proof}
As an immediate corollary, we can draw the picture of these spaces.
\begin{cor}\label{cor:explicit-integral-spectrum}
    Let $k$ be a field satisfying \Cref{hyp:rational-tate-spectrum}, and let $e$ be its exponential characteristic. If $k$ is not formally real, $\Spc(\DTM_{\gm}(k; \Z[1/e]))$ is described as follows:
    
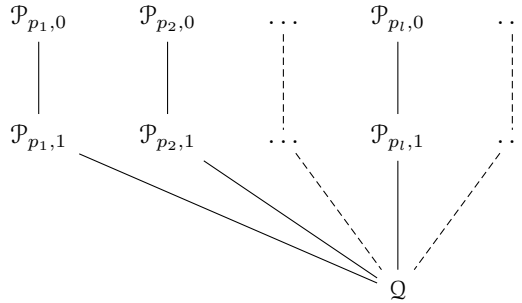
\[\begin{tikzcd}[ampersand replacement=\&,sep=scriptsize]
	   {\mathcal{P}_{p_1,0}} \& {\mathcal{P}_{p_2,0}} \& \dots \& {\mathcal{P}_{p_l,0}} \& \dots \\
	   \\
	   {\mathcal{P}_{p_1,1}} \& {\mathcal{P}_{p_2,1}} \& \dots \& {\mathcal{P}_{p_l,1}} \& \dots \\
	   \\
	   \\
	   \&\&\& {\mathcal{Q}}
	   \arrow[no head, from=1-1, to=3-1]
	   \arrow[no head, from=1-2, to=3-2]
	   \arrow[dashed, no head, from=1-3, to=3-3]
	   \arrow[no head, from=1-4, to=3-4]
	   \arrow[dashed, no head, from=1-5, to=3-5]
	   \arrow[no head, from=3-1, to=6-4]
	   \arrow[no head, from=3-2, to=6-4]
	   \arrow[dashed, no head, from=3-3, to=6-4]
	   \arrow[no head, from=3-4, to=6-4]
	   \arrow[dashed, no head, from=3-5, to=6-4]
    \end{tikzcd}\]
    where $p_i$ range over all prime numbers different from $e$. Moreover, a proper subset $Z$ is closed if and only if it is finite and specialisation closed.
    
    If $k$ is formally real, $\Spc(\DTM_\gm(k;\Z))$ is described as follows:
    
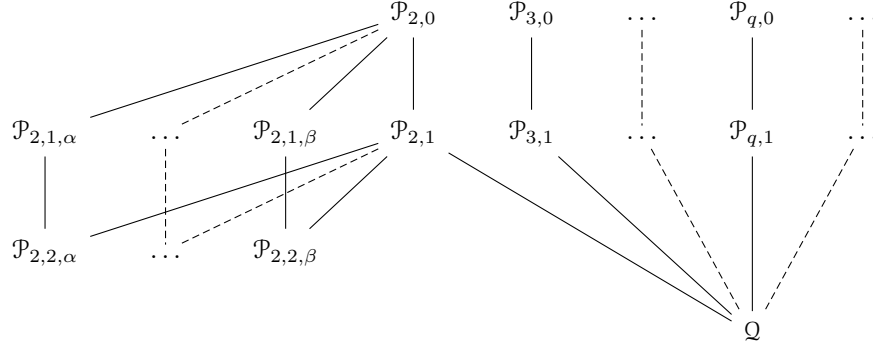
\[\begin{tikzcd}[ampersand replacement=\&,sep=scriptsize]
	   \&\&\& {\mathcal{P}_{2,0}} \& {\mathcal{P}_{3,0}} \& \dots \& {\mathcal{P}_{q,0}} \& \dots \\
	   \\
	   {\mathcal{P}_{2,1,\alpha}} \& \dots \& {\mathcal{P}_{2,1, \beta}} \& {\mathcal{P}_{2,1}} \& {\mathcal{P}_{3,1}} \& \dots \& {\mathcal{P}_{q,1}} \& \dots \\
	   \\
	   {\mathcal{P}_{2,2,\alpha}} \& \dots \& {\mathcal{P}_{2,2,\beta}} \\
	   \&\&\&\&\&\& {\mathcal{Q}}
	   \arrow[no head, from=1-4, to=3-1]
	   \arrow[dashed, no head, from=1-4, to=3-2]
	   \arrow[no head, from=1-4, to=3-3]
	   \arrow[no head, from=1-4, to=3-4]
	   \arrow[no head, from=1-5, to=3-5]
	   \arrow[dashed, no head, from=1-6, to=3-6]
	   \arrow[no head, from=1-7, to=3-7]
	   \arrow[dashed, no head, from=1-8, to=3-8]
	   \arrow[no head, from=3-1, to=5-1]
	   \arrow[dashed, no head, from=3-2, to=5-2]
	   \arrow[no head, from=3-3, to=5-3]
	   \arrow[no head, from=3-4, to=5-1]
	   \arrow[dashed, no head, from=3-4, to=5-2]
	   \arrow[no head, from=3-4, to=5-3]
	   \arrow[no head, from=3-4, to=6-7]
	   \arrow[no head, from=3-5, to=6-7]
	   \arrow[dashed, no head, from=3-6, to=6-7]
	   \arrow[no head, from=3-7, to=6-7]
	   \arrow[dashed, no head, from=3-8, to=6-7]
    \end{tikzcd}\]
    where $q$ ranges over all prime numbers, and with topology described in \Cref{thm:integral-tate-spectrum} (4).
\end{cor}

\section{The spectrum of Artin--Tate motives}
In this section we compute the spectrum of Artin--Tate motives
\[\DATM_\gm(k;\F_p)\]
over arbitrary fields $k$ whose characteristic differs from $p$.

Before we state the main theorem, let us declare some notation. For a field extension $E/k$, we denote by
\[a_E \in \Spc(\DATM_\gm(k;\F_p)\]
the restriction of the isotropic point in $\Spc(\DM_\gm(k;\F_p))$ to the spectrum of Artin--Tate motives.

For an ordered field $(k,\alpha)$, we denote by $k_\alpha$ its real closure.

\begin{theorem}\label{theorem:ArtinTatespectrum}
    The spectrum $\Spc(\DATM_{gm}(k;\F_p))$ consists of a closed subspace
    \[V \cong \Spc(\DAM_\gm(k;\F_p))\]
    whose open complement is
    \[U \coloneqq V^c \cong \Spc(\DTM_\gm(k;\F_p)) \setminus \{\ideal{0}\}.\]
    The specialisation relations from the open piece to the closed piece are as follows:
    \begin{itemize}
        \item $\mathcal{P}_{\acute{e}t} \rightsquigarrow a_{\overline{k}}$;
        \item if $k$ is formally real and $p=2$, then for every $\alpha \in \Sper(k)$ we have
        \begin{enumerate}
            \item $\mathcal{P}^\alpha \rightsquigarrow a_{k_\alpha}$;
            \item $\mathcal{P}^\alpha_{\acute{e}t} \rightsquigarrow \mathcal{Q} \in V$ iff $\mathcal{Q} \in \overline{\{b_\alpha\}}$, where $b_\alpha$ is the point in $\Spc(\DAM_\gm(k;\F_2))$ which is the image of the generic point of $\Spc(\DAM_\gm(k_\alpha;\F_2))$ under base change;
        \end{enumerate}
    \end{itemize}
    The topology is the finest one for which the above conditions hold: a subset $Z \subseteq \Spc(\DATM_\gm(k;\F_p))$ is closed if and only if
    \begin{enumerate}
        \item $Z \cap V$ is closed in $V \cong \Spc(\DAM_\gm(k;\F_p))$;
        \item $Z \cap U$ is closed in $U \cong \Spc(\DTM_\gm(k;\F_p)) \setminus \{\ideal{0}\}$;
        \item $Z$ is specialisation closed.
    \end{enumerate}
\end{theorem}

\begin{remark}
    When $p$ is odd or $k$ is not formally real, the spectrum $\Spc(\DATM_\gm(k;\F_p))$ consists of a (closed) homeomorphic copy of $\Spc(\DAM_\gm(k;\F_p))$ together with the open point $\mathcal{P}_{\acute{e}t}$ corresponding to the kernel of étale realisation restricted to Artin--Tate motives:
    
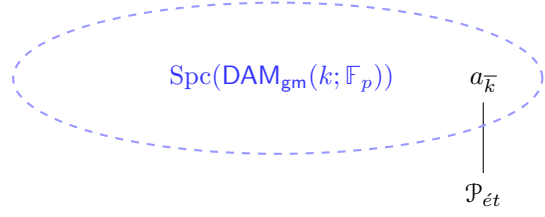
\begin{figure}[H]
    \begin{tikzcd}[ampersand replacement=\&,sep=scriptsize,execute at end picture={\draw[dashed,thick,blue!40] (-0.6,0.75) ellipse (3.5cm and 1cm);}]
	   \&\& {\textcolor{blue!80}{\Spc(\DAM_\gm(k;\F_p))}} \& {a_{\overline{k}}} \& \\
	   \\
	   \&\&\& {\mathcal{P}_{\acute{e}t}} \&
	   \arrow[no head, from=1-4, to=3-4]
    \end{tikzcd}
    \end{figure}
    The topology is such that a subset $Z \subseteq \Spc(\DATM_\gm(k;\F_p))$ is closed if and only if it is specialisation closed and its intersection with $\Spc(\DAM_\gm(k;\F_p))$ is closed.
    
    When $p=2$ and $k$ is formally real, the spectrum again consists of a (closed) homeomorphic copy of $\Spc(\DAM_\gm(k;\F_2))$ together with the open $\Spc(\DTM_\gm(k;\F_2)) \setminus \{\ideal{0}\} \cong \Proj(K^M_*(k)/2[\tau])$, with specialisations depicted like so:
    
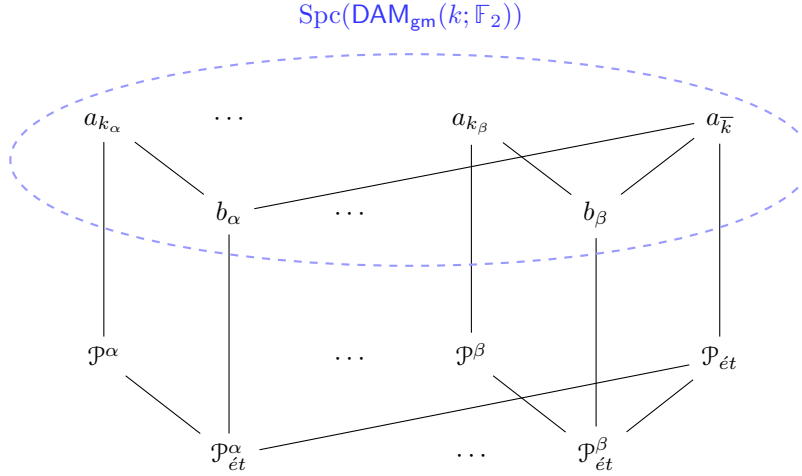
\begin{figure}[H]
    \begin{tikzcd}[ampersand replacement=\&, execute at end picture={\draw[dashed,thick,blue!40] (0,1.3) ellipse (5.3cm and 1.4cm); \node[blue!80] at (0,3.2) {$\Spc(\DAM_\gm(k;\F_2))$};}]
	   \&\&\& {} \&\& \\
	   {a_{k_\alpha}} \& \cdots \&\& {a_{k_{\beta}}} \&\& {a_{\overline{k}}} \\
	   \& {b_\alpha} \& \cdots \&\& {b_{\beta}} \\
	   \\
	   {\mathcal{P}^\alpha} \&\& \cdots \& {\mathcal{P}^\beta} \&\& {\mathcal{P}_{\acute{e}t}} \\
	   \& {\mathcal{P}^\alpha_{\acute{e}t}} \&\& \cdots \& {\mathcal{P}^\beta_{\acute{e}t}}
	   \arrow[no head, no head, from=2-1, to=3-2]
	   \arrow[no head, from=2-1, to=5-1]
	   \arrow[no head, no head, from=2-4, to=3-5]
	   \arrow[no head, from=2-4, to=5-4]
	   \arrow[no head, from=2-6, to=5-6]
	   \arrow[no head, no head, from=3-2, to=2-6]
	   \arrow[no head, from=3-2, to=6-2]
	   \arrow[no head, no head, from=3-5, to=2-6]
	   \arrow[no head, from=3-5, to=6-5]
	   \arrow[no head, from=5-1, to=6-2]
	   \arrow[no head, from=5-4, to=6-5]
	   \arrow[no head, from=6-2, to=5-6]
	   \arrow[no head, from=6-5, to=5-6]
    \end{tikzcd}
    \end{figure}
    The topology is the one spelt out in \Cref{theorem:ArtinTatespectrum}.
\end{remark}

Let us spend a few moments reviewing the strategy of the proof. As in the previous section, we begin by assuming that $k$ contains a primitive $p^{th}$ root of unity; we explain how to remove this assumption at the end of the section. We define a Milnor weight structure on $\DATM_\gm(k;\F_p)$ whose heart $\mathcal{H}^\mathsf{AT}$ is given by the idempotent completion of the additive closure of the objects
\[\{M(E)(n)[n] \mid E/k \text{ a finite separable extension, } n \in \Z\}.\]

As in \Cref{prop:Milnor-t_-detects-nilp-tau}, we shall find that the corresponding weight complex functor $t$ detects $\otimes$-nilpotence on $\cone(\tau)$. This partitions the spectrum into two pieces, the open piece $U(\cone(\tau))$, which we know thanks to \Cref{rem:open-tau-DM}, and its closed complement $\supp(\cone(\tau))=\im(\Spc(t))$. We remark that, when $k$ does not contain a primitive $p^{th}$ root of unity, $\tau$ is replaced throughout by $\beta$ from \Cref{prop:hSpec-motcohom-no-root-unity}.

It remains to compute the support of $\cone(\tau)$ (or $\cone(\beta)$) in $\Spc(\DATM_\gm(k;\F_p))$: in the case $p$ is odd or $k$ is not formally real, this closed piece is homeomorphic to the space
\[\Spc(\DAM_\gm(k;\F_p))\]
computed by Balmer--Gallauer \cite{Balmer:Artin}. The difficulty arises when $p=2$ and $k$ is formally real. In this case, we partition $\supp(\cone(\tau))$ into two further pieces:
\[\supp(\cone(\tau)) \cap \supp(\cone(\rho)) = \supp(\cone(\tau) \otimes \cone(\rho))\]
and
\[\supp(\cone(\tau)) \cap U(\cone(\rho)).\]

The first intersection consists of a homeomorphic copy of $\Spc(\DAM_\gm(k;\F_p))$. To identify this closed subspace, we construct a tt-functor
\[\begin{tikzcd}
    \K^b(\mathcal{H}^\mathsf{AT}) \arrow[r] & \K^b(\AM_\gm(k;\F_p)) \simeq \DAM_\gm(k;\F_p)
\end{tikzcd}\]
which effectively deletes all the information coming from positive degrees in the Milnor $K$-theory of the field extensions of $k$. Composing this functor with the weight complex functor defines a retraction of the inclusion
\[\begin{tikzcd}
    \DAM_\gm(k;\F_p) \arrow[r, hookrightarrow] & \DATM_\gm(k;\F_p),
\end{tikzcd}\]
and we show that it detects $\otimes$-nilpotence on $\cone(\rho) \otimes \cone(\tau)$.

To compute the remaining region
\[\supp(\cone(\tau)) \cap U(\cone(\rho)),\]
we utilise Bachmann's main result of \cite{Bachmann:Real} (see also \cite[Theorem~4.2]{Bachmann:Remarks}):

\begin{theorem}[Bachmann]
    There is an equivalence of stable $\infty$-categories
    \[\mathcal{SH}(k)[\rho^{-1}] \simeq \Shv(\Sper(k);\Sp),\]
    where $\Shv(\Sper(k);\Sp)$ is the $\infty$-category of spectral sheaves on $\Sper(k)$.
\end{theorem}

Here
\[\rho \in \pi_{-1,-1}(\unit) \simeq K^{MW}_1(k)\]
corresponds to the pure symbol $\{-1\}$ in the (integral) Milnor--Witt ring of $k$: in particular, it becomes the usual $\rho \in H_\mathcal{M}^{1,1}(k;\Z)$ upon taking motives (up to a $(1)[1]$-twist).

We prove that the stable $\infty$-category $\Shv(\Sper(k);\Sp)$ is compactly generated by the unit $\unit$, and hence so is $\mathcal{SH}(k)[\rho^{-1}]$ by the above. This suffices to conclude the equivalence
\[\DTM(k;\F_p)[\rho^{-1}] \simeq \DM(k;\F_p)[\rho^{-1}].\]
We can then identify the open corresponding to $\cone(\rho)$, since we have computed it for $\DTM_\gm(k;\F_p)$.

From this, it will follow that in $\Spc(\DATM_\gm(k;\F_p))$, as well as in $\Spc(\DM_\gm(k;\F_p))$,
\[\supp(\cone(\tau)) \cap U(\cone(\rho)) \cong \Sper(k).\]

Now that we have identified the spectrum as a set, it remains to determine the specialisation relations amongst the points and the topology of the space. The specialisations between the open and closed pieces of \Cref{theorem:ArtinTatespectrum} are detected by base changing to fields over which the spectrum of Artin--Tate motives is already known: namely, the algebraic closure $\overline{k}$ of $k$ and, in the formally real case and when $p=2$, the real closures $k_\alpha$ associated to orderings $\alpha \in \Sper(k)$ of $k$. The induced maps on spectra produce the required specialisations, whilst the retraction
\[\begin{tikzcd}
    \Spc(\DATM_\gm(k;\F_p)) \arrow[r, twoheadrightarrow] & \Spc(\DAM_\gm(k;\F_p))
\end{tikzcd}\]
can be used to show that no other specialisations can occur between the open $U$ and $V$ described in \Cref{theorem:ArtinTatespectrum}.

Finally, the topology is determined by the known topologies of $U$ and $V$: a subset is closed precisely when it is closed under specialisations and its intersection with $U$ and $V$ are closed in these respective subspaces. Alternatively, there is another succinct description of the topology in terms of the inclusions
\[\begin{tikzcd}
    \DTM_\gm(k;\F_p) \arrow[r, hookrightarrow] & \DATM_\gm(k;\F_p) \arrow[r, hookleftarrow] & \DAM_\gm(k;\F_p)
\end{tikzcd}\]
and the induced maps on spectra
\[\begin{tikzcd}
    \Spc(\DTM_\gm(k;\F_p)) & \arrow[l, twoheadrightarrow] \Spc(\DATM_\gm(k;\F_p)) \arrow[r, twoheadrightarrow] & \Spc(\DAM_\gm(k;\F_p)):
\end{tikzcd}\]
these induce a continuous map
\[\begin{tikzcd}
    \Spc(\DATM_\gm(k;\F_p)) \arrow[r] & \Spc(\DTM_\gm(k;\F_p)) \times \Spc(\DAM_\gm(k;\F_p)),
\end{tikzcd}\]
which we find is a homeomorphism onto its image. That is, this map is injective and the topology on $\Spc(\DATM_\gm(k;\F_p))$ agrees with the subspace topology inherited from the product topology on
\[\Spc(\DTM_\gm(k;\F_p)) \times \Spc(\DAM_\gm(k;\F_p)).\]

\subsection{The detecting tt-functors}
\begin{lemma}\label{lemma:MilnorweightstructureforDATM}
    The category $\DATM_\gm(k;\F_p)$ admits a unique bounded $\otimes$-compatible weight structure whose heart is
    \[\mathcal{H}^\mathsf{AT} \coloneqq \add(M(E)(n)[n] \mid E/k \text{ finite separable, } n \in \Z)^\natural.\]
\end{lemma}
\begin{proof}
    The proof is almost identical to \Cref{lem:milnor-weight-structure}, so we shall be brief.

    We use \Cref{theorem:constructingweightstructures}: the generation condition is immediate, so we must verify the orthogonality condition. To this end, let $E/k,F/k$ be two finite separable extensions, let $i,j \in \Z$, and let $r>0$. We compute
    \[\Hom(M(E)(i)[i],M(F)(j)[j+r]) \simeq \Hom(M(E) \otimes M(F),\unit(j-i)[j-i+r]) \simeq \bigoplus_\alpha H_\mathcal{M}^{j-i+r,j-i}(L_\alpha;\F_p)\]
    where we have decomposed $E \otimes_k F$ into a product of finite field extensions $L_\alpha/k$ (which is possible due to separability of $E$ and $F$ over $k$). Since $\F_p$-motivic cohomology of zero dimensional varieties vanishes above the diagonal, the groups above are zero for $r>0$. We can now apply \Cref{theorem:constructingweightstructures} to get the existence of a unique bounded weight structure with heart $\mathcal{H}^\mathsf{AT}$. Since $\mathcal{H}^\mathsf{AT}$ is clearly a tensor-subcategory of $\DATM_\gm(k;\F_p)$, the weight structure is $\otimes$-compatible by \Cref{remark:boundedtensorcompatibleweightstructure}.
\end{proof}

By \Cref{lemma:MilnorweightstructureforDATM} and \Cref{theorem:existenceofweightcomplexfunctor} we get the existence of the \emph{Milnor weight complex functor}
\[\begin{tikzcd}
    t:\DATM_\gm(k;\F_p) \arrow[r] & \K^b(\mathcal{H}^\mathsf{AT})
\end{tikzcd}\]
on the category of Artin--Tate motives.

\begin{prop}\label{prop:detectingnilpotenceoncone(tau)forDATM}
    Suppose $k$ has a primitive $p^{th}$ root of unity. Let $f:X \longrightarrow Y$ be a morphism in $\DATM_\gm(k;\F_p)$. If $t(f)=0$, then
    \[f \otimes \cone(\tau)=0.\]
    In particular, the Milnor weight complex functor detects $\otimes$-nilpotence on $\cone(\tau)$.
\end{prop}
\begin{proof}
    This again is very similar to \Cref{prop:Milnor-t_-detects-nilp-tau}: the mod-$p$ motivic cohomology of any finite separable field extension of $k$ is concentrated on and below the diagonal, and multiplication by $\tau$ is injective on the motivic cohomology algebra
    \[H_\mathcal{M}^{*,*}(E;\F_p) \simeq (K^M_*(E)/p)[\tau]\]
    where, as always, $\tau$ is in bidegree $(0,1)$.
\end{proof}

\begin{cor}\label{cor:jointdetectingfunctorsDATM}
    Suppose $k$ has a primitive $p^{th}$ root of unity. The functors
    \[\begin{cases}
        t:\DATM_\gm(k;\F_p) \longrightarrow \K^b(\mathcal{H}^\mathsf{AT}); \\
        L_\tau:\DATM_\gm(k;\F_p) \longrightarrow (\DATM_\gm(k;\F_p)/\ideal{\cone(\tau)})^\natural \simeq \DATM_\et^\vee(k;\F_p)
    \end{cases}\]
    jointly detect $\otimes$-nilpotence of morphisms. Moreover, the image of $\Spc(t)$ is precisely $\supp(\cone(\tau))$.
\end{cor}
\begin{proof}
    The first statement is a consequence of \Cref{prop:detectingnilpotenceoncone(tau)forDATM} and the proof of \cite[Corollary~7.8]{Balmer:ThreeRealArtinTates}. The statement about the image of $\Spc(t)$ is \Cref{prop:detectingnilpotenceoncone(tau)forDATM} plus \Cref{prop:detectingnilpotenceonanobject}, combined with the observation that $\im(\Spc(t))$ is contained in $\supp(\cone(\tau))$, since $t(\tau)=0$.
\end{proof}

\begin{remark}\label{remark:open(cone(tau))inDATM}
    Combining \Cref{prop:tau-localisation-etale}, \Cref{thm:gallauer-etale-spectrum}, \Cref{lemma:homeoofcomparisonmapdescendstosubcategories}, and \Cref{remark:principalquotientgiveshomeoontoopen}, we see that $L_\tau$ induces a homeomorphism
    \[U(\cone(\tau)) \cong \Spec^h(K^M_*(k)/p).\]
    So, it remains to compute $\supp(\cone(\tau)) \subseteq \Spc(\DATM_\gm(k;\F_p))$, which by \Cref{cor:jointdetectingfunctorsDATM} means we must understand both $\Spc(\K^b(\mathcal{H}^\mathsf{AT}))$ and the map $\Spc(t)$.
\end{remark}

Unfortunately, unlike with $\DTM_\gm(k;\F_p)$, we see no clear algebraic description of the heart $\mathcal{H}^\mathsf{AT}$ of the Milnor weight structure on Artin--Tate motives. We therefore introduce another tt-functor which further refines the information detected by the Milnor weight complex functor.

Let
\[\AM_\gm(k;\F_p) \subseteq \DATM_\gm(k;\F_p)\]
denote the idempotent completion of the additive subcategory of $\DATM_\gm(k;\F_p)$ generated by the motives $M(E)$ for finite separable extensions $E/k$, that is,
\[\AM_\gm(k;\F_p) \coloneqq \add(M(E) \in \DATM_\gm(k;\F_p) \mid E/k \text{ finite separable})^\natural.\]

\begin{lemma}
    There is a well-defined tt-functor
    \[\operatorname{gr}:\K^b(\mathcal{H}^\mathsf{AT}) \longrightarrow \K^b(\AM_\gm(k;\F_p))\]
    which is defined by sending the $M(E)(n)[n]$ (in degree $0$) to $M(E)$ (in degree $0$) where $E/k$ is any finite separable extension and $n \in \Z$ is any integer. On morphisms it is characterised by
    \[(f:M(E)(i)[i] \longrightarrow M(F)(j)[j]) \longmapsto \begin{cases}
        f, \ i=j; \\
        0, i \neq j
    \end{cases}\]
    where $E,F/k$ are finite separable extensions and $i,j \in \Z$ are integers. We call it the \emph{grading functor}.
\end{lemma}
\begin{proof}
    We first verify that the functor $\operatorname{gr}$ is well-defined: we will then establish its monoidality (it is additive by construction). We note that since $\operatorname{gr}$ is defined on the hearts $\mathcal{H}^\mathsf{AT} \longrightarrow \AM_\gm(k;\F_p)$, and since the functor $\K^b(-):\mathsf{Add}_\otimes \longrightarrow \mathsf{ttCat}$ takes additive symmetric monoidal functors to additive symmetric monoidal functors, it is enough to verify everything on the heart $\mathcal{H}^\mathsf{AT}$.

    To verify it is well-defined all we need to do is verify that $\operatorname{gr}(g \circ f)$ agrees with $\operatorname{gr}(g) \circ \operatorname{gr}(f)$ for any given pair of morphisms $f:M(E)(i)[i] \longrightarrow M(F)(j)[j],g:M(F)(j)[j] \longrightarrow M(L)(l)[l]$ (clearly the functor sends identity morphisms to identity morphisms).

    So, letting $f,g$ be as above, we see there are two cases to consider. The first case is where $i=j=l$: in this instance the functor $\operatorname{gr}$ does not have any effect, so we are fine. The second case is where not all the integers $i,j,l$ agree. In this scenario, we can assume $i \leq j \leq l$ for otherwise either $f$ or $g$ is zero to begin with from the description of morphisms in $\mathcal{H}^\mathsf{AT}$. So, either $i<j$ or $j<l$, that is, $\operatorname{gr}(f)=0$ or $\operatorname{gr}(g)=0$. In either case, the composition is such that $i<l$, so $\operatorname{gr}(g \circ f)=0$ also. This shows the functor is well-defined.

    To show it is monoidal, we note that
    \[\operatorname{gr}(M(E)(i)[i] \otimes M(F)(j)[j]) = M(E) \otimes M(F) = \operatorname{gr}(M(E)(i)[i]) \otimes \operatorname{gr}(M(F)(j)[j])\]
    as can be seen by decomposing $M(E) \otimes M(F)$ into a direct sum of motives of the form $M(L_\alpha)$ as in the proof of \Cref{lemma:MilnorweightstructureforDATM}. It is easy to check the functor respects tensor-products of morphisms too, so we conclude.
\end{proof}

Before stating the next result, we record a technical lemma concerning the $\otimes$-nilpotence of morphisms in $\DM_\gm(k;\F_p)$ arising from mod-$p$ Milnor K-theory.

\begin{lemma}\label{lemma:nilpotentpuresymbolsaretensornilpotent}
    Let $E/k$ be a finite separable extension of $k$ and let
    \[\begin{tikzcd}
        a:\unit \arrow[r] & M(E)(n)[n]
    \end{tikzcd}\]
    be a morphism corresponding to a pure symbol $\alpha \in K^M_n(E)/p$ with $n \geq 1$. Then $a$ is $\otimes$-nilpotent on $\cone(\rho)$. In particular, if $p$ is an odd prime or $k$ is not formally real, then $a$ is $\otimes$-nilpotent.
\end{lemma}
\begin{proof}
    For $N \geq 1$, consider the morphism
    \[\begin{tikzcd}
        a^{\otimes N}: \unit \arrow[r] & M(E)^{\otimes N}(Nn)[Nn] \cong M(E^{\otimes_k N})(Nn)[Nn].
    \end{tikzcd}\]
    We will show that for $N$ sufficiently large,
    \[a^{\otimes N} \otimes \cone(\rho)=0.\]
    First, for each $i=1,\dots,N$ we have the morphism
    \[\begin{tikzcd}
        \iota_i: E \arrow[r] & E^{\otimes_k N}; \quad x \longmapsto 1 \otimes \dots \otimes 1 \otimes \underbrace{x}_{i^{th}\text{-position}} \otimes  1 \otimes \dots \otimes 1.
    \end{tikzcd}\]
    Note that, under the identification of diagonal morphisms in $\DM_\gm(k;\F_p)$ and Milnor K-theory, $a^{\otimes N}$ corresponds to the element
    \[\alpha^{\otimes N} = \iota_1^*(\alpha) \cdot \cdots \cdot \iota_N^*(\alpha) \in  K_*^M(E^{\otimes_k N})/p,\]
    where $\iota_i^*: K^M_*(E)/p \longrightarrow K^M_*(E^{\otimes_k N})/p$ is the homomorphism induced by $\iota_i$. Since $E$ is a finite separable extension of $k$, $E^{\otimes_k N}$ decomposes into a finite product of fields $L_\lambda$ which are also finite separable over $k$:
    \[E^{\otimes_k N} \simeq \prod_{\lambda} L_\lambda.\]
    Denote by $\pi_\lambda: E^{\otimes_k N} \longrightarrow L_\lambda $ the projection onto the corresponding component. Then $\alpha^{\otimes N}$ decomposes into a sum of pure symbols in $K^M_{nN}(L_\lambda)/p$, so it suffices to show that each component is zero for $N$ sufficiently large. Fix some $\lambda$, and for $i=1,\dots,N$, consider the $k$-algebra homomorphism 
    \[\sigma_{\lambda, i}=\pi_\lambda \circ \iota_{i}.\]
    Notice the component over $L \coloneqq L_\lambda$ of $\alpha^{\otimes N}$ is given by 
    \begin{equation}\label{eqn:decompositionofpuresymbol}
        \pi_\lambda^*(\alpha^{\otimes N})=\pi_\lambda^*(\iota_1^*(\alpha) \cdot \dots \cdot \iota_N^*(\alpha))= \sigma^*_{\lambda, 1}(\alpha) \dots \cdot \sigma^*_{\lambda, N}(\alpha).
    \end{equation}
    Letting $d \coloneqq [E:k]$ denote the degree of the extension, we observe that, as a consequence of the theorem of the primitive element, there are at most $d$ distinct $k$-algebra homomorphisms $E \longrightarrow L$. Hence, given an arbitrary $R \geq 1$, we may have chosen $N>2Rd$ to begin with, so that one of the maps $\sigma \coloneqq \sigma_{\lambda,i}$ is forced to occur at least $2R$ times. In particular, writing $\beta \coloneqq \sigma^*(\alpha)$, we have $\beta^{2R}$ divides the product from the right-hand side of (\ref{eqn:decompositionofpuresymbol}).
    But now
    \[\beta^{2R}=\rho_L^{nR} \cdot \beta^R\]
    by virtue of the general identity
    \[\gamma^2= \rho^{deg(\gamma)}_L \cdot \gamma\]
    for a pure symbol $\gamma \in K^M_*(L)/p$.

    To conclude, observe that the class $\rho_L^{nR}$ is represented by the composition
    \[\begin{tikzcd}
        \rho_L^{nR}:M(L) \arrow[r, "\Delta"] & M(L)^{\otimes nR} \arrow[r, "\rho_L^{\otimes nR}"] & \unit(nR)[nR],
    \end{tikzcd}\]
    where $\Delta$ is induced by the diagonal $\Spec(L) \longrightarrow \Spec(L)^{\times nR}$. Since $\rho_L$ is the image of $\rho$ under the base change functor from $k$ to $L$, we have
    \[\rho_L=\rho \circ \epsilon_L\]
    where $\epsilon_L:M(L) \longrightarrow \unit$ is induced by the structure morphism of the $k$-variety $\Spec(L)$.
    From here, it is clear that when tensored by $\cone(\rho)$, the $L$-component of $a^{\otimes N}$ becomes zero since $\rho \otimes \cone(\rho)=0$. Repeating the argument for all $\lambda$ allows us to conclude
    \[a^{\otimes N} \otimes \cone(\rho)=0,\]
    as claimed.

    For the final statement, suppose $p$ is odd or $k$ is not formally real. Then $\rho^m=0$ for some $m \geq 1$, so that $\rho_L^m=0$ for all finite separable extensions $L$ of $k$. Choosing $R$ such that $nR \geq m$, and then $N$ sufficiently large as before, we find that each $L_\lambda$-component of $a^{\otimes N}$ is zero, so
    \[a^{\otimes N}=0\]
    and we are done.
\end{proof}

\begin{prop}\label{prop:grdetectsnilpotenceoncone(rho)}
    The grading functor
    \[\begin{tikzcd}
        \operatorname{gr}:\K^b(\mathcal{H}^\mathsf{AT}) \arrow[r] & \K^b(\AM_\gm(k;\F_p)) 
    \end{tikzcd}\]
    detects $\otimes$-nilpotence of morphisms on $\cone(\rho)$. In particular, if $p$ is odd or $k$ is not formally real, $\operatorname{gr}$ detects $\otimes$-nilpotence of morphisms.
\end{prop}
\begin{proof}
    Let $f$ be a morphism in $\K^b(\mathcal{H}^\mathsf{AT})$ such that $\operatorname{gr}(f)=0$. Rigidity permits us to assume the domain of $f$ is the unit $\unit$, say
    \[\begin{tikzcd}
        f:\unit \arrow[r] & X
    \end{tikzcd}\]
    for some $X \in \K^b(\mathcal{H}^\mathsf{AT})$. Moreover, we may assume that in each degree $X$ is a (possibly empty) finite direct sum of objects of the form $M(E)(i)[i]$ for $E/k$ a finite separable extension and $i \in \Z$, since any object of $\K^b(\mathcal{H}^\mathsf{AT})$ is a summand of such an object, and if a functor detects $\otimes$-nilpotence of a morphism (on some object) then it does so on all its retractions.

    Now choose some chain homotopy equivalence class of $X$. In degree $0$, $f:\unit \longrightarrow X$ is given by a morphism
    \[\begin{tikzcd}
        f_0:\unit \arrow[r] & X_0,
    \end{tikzcd}\]
    in $\mathcal{H}^\mathsf{AT}$, where $X_0$ is the degree $0$ component of $X$. By definition, we may decompose $X_0$ as
    \[X_0 \simeq X_0^- \oplus X_0^0 \oplus X_0^+\]
    where $X_0^-/X_0^0/X_0^+$ is a (finite) direct sum of objects of the form $M(E)(i)[i]$, where $E/k$ is a finite separable extension of $k$ and $i \in \Z$ is negative/zero/positive, respectively.
    
    Decomposing $f_0$ into the sum of its components induced by the above decomposition, we see that $\operatorname{gr}(f)=0$ precisely when the component of $\operatorname{gr}(f_0)$ corresponding to $\unit \longrightarrow X_0^0$ admits a lift along the differential of $\operatorname{gr}(X)$. Indeed, the component of $f_0$ corresponding to $X_0^-$ is already zero by assumption since
    \[\Hom(\unit,M(E)(i)[i])=0\]
    whenever $i<0$, whilst the component corresponding to $X_0^+$ vanishes under $\operatorname{gr}$ by construction. Moreover, $\operatorname{gr}$ does not change the component of $f$ corresponding to $X_0^0$. So, the homotopic triviality of $\operatorname{gr}(f)$ in $\K^b(\AM_\gm(k;\F_p))$ is to say $f_0=\operatorname{gr}(f_0)$ admits a lift along $\operatorname{gr}(d_1):\operatorname{gr}(X_1) \longrightarrow \operatorname{gr}(X_0)$ as follows:
    \[\begin{tikzcd}
        \unit \arrow[r, "\operatorname{gr}(f_0)"] \arrow[rd, swap, "g"] & \operatorname{gr}(X_0) \\
        & \operatorname{gr}(X_1) \arrow[u, swap, "\operatorname{gr}(d_1)"]
    \end{tikzcd}\]
    By construction, $\operatorname{gr}$ kills any morphisms between different diagonal shifts of objects of the heart, so $\operatorname{gr}(d_1)$ is given by a block diagonal matrix corresponding to the respective negative/zero/positive decompositions of $X_1,X_0$:
    \[\operatorname{gr}(d_1) \in \Hom(\operatorname{gr}(X_1^-),\operatorname{gr}(X_0^-)) \oplus \Hom(\operatorname{gr}(X_1^0),\operatorname{gr}(X_0^0)) \oplus \Hom(\operatorname{gr}(X_1^+),\operatorname{gr}(X_0^+))\]
    Writing $g$ in terms of its components $g=(g_-,g_0,g_+)$ corresponding to the block matrix decomposition in $\Hom(\unit,\operatorname{gr}(X_1))$, we see that without loss of generality we may assume $g_-$ is zero: indeed, $\operatorname{gr}(f_0)$ is zero on $X_0^-$.

    Write
    \[\begin{tikzcd}
        d_1^0:X_1^0 \arrow[r] & X_0^0
    \end{tikzcd}\]
    for the component of $d_1$ corresponding to the components $X_1^0,X_0^0$. Observe this is the same as the component $\operatorname{gr}(d_1)\vert_0:\operatorname{X_1^0} \longrightarrow \operatorname{X_0^0}$. We now see that the morphism $f_0-d_1^0 \circ g_0$ is homotopic to $f$, and its component corresponding to $\unit \longrightarrow X_0^0$ is zero by construction. So, replacing $f$ with this new representative, we may assume that
    \[\begin{tikzcd}
        f_0:\unit \arrow[r] & X_0^+.
    \end{tikzcd}\]

    By \Cref{lemma:nilpotentpuresymbolsaretensornilpotent}, $f_0$ is such that $f_0^{\otimes N} \otimes \cone(\rho)=0$ for $N \gg 0$. This completes the proof of the first statement.

    The second statement follows from the first, noting that under the hypotheses $p$ is odd or $k$ is not formally real, we have $\rho$ is $\otimes$-nilpotent by \Cref{lemma:nilpotentpuresymbolsaretensornilpotent}, so
    \[\unit \in \ideal{\cone(\rho)}.\]
    Since the collection of objects for which $f$ is $\otimes$-nilpotent on form a tt-ideal \cite[Proposition~2.10]{Balmer:Spectra}, we conclude $f$ is already $\otimes$-nilpotent in this case.
\end{proof}

\begin{cor}\label{cor:detectingnilpotenceontauandrho}
    Suppose $k$ has a primitive $p^{th}$ root of unity. The composite functor
    \[\begin{tikzcd}
        \DATM_\gm(k;\F_p) \arrow[r, "t"] & \K^b(\mathcal{H}^\mathsf{AT}) \arrow[r, "\operatorname{gr}"] & \K^b(\AM_\gm(k;\F_p))
    \end{tikzcd}\]
    detects $\otimes$-nilpotence of morphisms on $\cone(\tau) \otimes \cone(\rho)$.
\end{cor}
\begin{proof}
    Combine \Cref{prop:detectingnilpotenceoncone(tau)forDATM} with \Cref{prop:grdetectsnilpotenceoncone(rho)}.
\end{proof}

\begin{cor}\label{cor:supp(cone(rho)cone(tau))}
    Suppose $k$ has a primitive $p^{th}$ root of unity. The induced map on spectra of the tt-functor
    \[\begin{tikzcd}
        \Psi \coloneqq \operatorname{gr} \circ t:\DATM_\gm(k;\F_p) \arrow[r] & \K^b(\AM_\gm(k;\F_p))
    \end{tikzcd}\]
    induces a homeomorphism onto its image $\supp(\cone(\rho) \otimes \cone(\tau))$:
    \[\begin{tikzcd}
        \Spc(\DAM_\gm(k;\F_p)) \arrow[r, "\simeq"] & \im(\Spc(\Psi))=\supp(\cone(\tau) \otimes \cone(\rho)) \subseteq \Spc(\DATM_\gm(k;\F_p)).
    \end{tikzcd}\]
\end{cor}
\begin{proof}
    Since $\Psi(\rho)=0=\Psi(\tau)$, \Cref{cor:detectingnilpotenceontauandrho} and \Cref{prop:detectingnilpotenceonanobject} imply that
    \[\im(\Spc(\Psi))=\supp(\cone(\rho) \otimes \cone(\tau)).\]
    To prove it is a homeomorphism onto its image, we observe that $\Psi$ defines a retraction of the inclusion of Artin motives into Artin--Tate motives:
    \[\begin{tikzcd}
        \DATM_\gm(k;\F_p) \arrow[r, "t"] & \K^b(\mathcal{H}^\mathsf{AT}) \arrow[r, "\operatorname{gr}"] & \K^b(\AM_\gm(k;\F_p)) \simeq \DAM_\gm(k;\F_p) \\
        \DAM_\gm(k;\F_p) \arrow[u, hookrightarrow, "\iota"] \arrow[urr, swap, "\simeq"]
    \end{tikzcd}\]
    Indeed, the functor $\Psi \circ \iota$ is weight exact and its restriction to the heart is the identity functor. Passing to spectra, it follows $\Spc(\iota)$ is the continuous inverse of $\Spc(\Psi)$.
\end{proof}
\begin{remark}
    A similar argument to the proof of \Cref{cor:supp(cone(rho)cone(tau))} was used in \cite[§10]{Balmer:ThreeRealArtinTates} to locate the homeomorphic copy of $\Spc(\DAM_\gm(k;\F_2))$ inside $\Spc(\DATM_\gm(k;\F_2))$ where $k$ is a real closed field (although their proof does not use this fact). There, they define a functor
    \[\operatorname{gr}^{w_{Chow}}:\DATM_\gm(k;\F_2) \rightarrow \K^b(\AM_\gm(k;\F_2))\]
    which is a composition of the Chow weight complex functor following by summing over all pure twists, and this functor also defines a retraction of the inclusion $\DAM_\gm(k;\F_2)$ into $\DATM_\gm(k;\F_2)$.
\end{remark}
\begin{remark}
    Note that, if $p$ is odd and $k$ has a primitive $p^{th}$ root of unity, this already computes the spectrum as a set. Indeed, since $\rho$ is $\otimes$-nilpotent in these cases we have $\ideal{\cone(\rho)}=\DATM_\gm(k;\F_p)$, hence \Cref{cor:supp(cone(rho)cone(tau))} simplifies to
    \[\supp(\cone(\tau)) \cong \Spc(\DAM_\gm(k;\F_p)).\]
    Moreover, the complementary open
    \[U(\cone(\tau))\]
    of $\cone(\tau)$ is a singleton under these hypotheses (\Cref{remark:open(cone(tau))inDATM}), so as a set the spectrum
    \[\Spc(\DATM_\gm(k;\F_p))\]
    consists of a (closed) copy of $\Spc(\DAM_\gm(k;\F_p))$ together with the (open) étale point.
\end{remark}

\subsection{The open of $\cone(\rho)$}
In this subsection we describe the open
\[U(\cone(\rho)) \subseteq \Spc(\DATM_\gm(k;\F_p)).\]
Such a computation will be a consequence of the following main result, paired with our knowledge of the spectrum of Tate motives from \Cref{thm:tate-comparison}.

\begin{theorem}\label{theorem:compactgenerationofDM/cone(rho)}
    The triangulated category $\DM(k;\F_p)[\rho^{-1}]$ is compactly generated by the unit $\unit$. In particular,
    \[(\DM_\gm(k;\F_p)[\rho^{-1}])^\natural \simeq (\DTM_\gm(k;\F_p)[\rho^{-1}])^\natural.\]
\end{theorem}

Observe that, when $p$ is odd or $k$ is not formally real, the above statement is obvious since both categories are trivial in these cases (since $\rho$ is $\otimes$-nilpotent in these cases). However, the result is non-trivial when $p=2$ and $k$ is formally real. To prove it, we establish the following result regarding compact generation for spectral sheaves on Boolean spaces.

\begin{prop}\label{prop:spectralsheavescompactgeneration}
    Let $X$ be a Boolean space. Then the stably symmetric monoidal $\infty$-category $\Shv(X;\Sp)$ of spectral sheaves on $X$ is compactly generated by the unit $\unit$.
\end{prop}
\begin{proof}
    We show, for any $E \in \Shv(X;\Sp)$,
    \begin{equation}\label{eqn:trivialglobalhomotopygroups}
        \Map_{\Shv(X;\Sp)}(\unit,\Sigma^nE) \simeq *, \text{ for every } n \in \Z
    \end{equation}
    if and only if $E$ is the trivial object.
    
    Let $V \subseteq X$ be a clopen subset, and write $U \coloneqq X \setminus V$. The clopen decomposition
    \[X = V \sqcup U\]
    induces a decomposition of $\infty$-categories
    \[\Shv(X;\Sp) \simeq \Shv(V;\Sp) \times \Shv(U;\Sp),\]
    which gives a decomposition of the $\otimes$-unit $\unit \in \Shv(X;\Sp)$
    \[\unit \simeq \unit_V \oplus \unit_U\]
    where $\unit_V$ (resp. $\unit_U$) is the image under the left adjoint to the canonical projection
    \[\begin{tikzcd}
        p_V:\Shv(X;\Sp) \simeq \Shv(V;\Sp) \times \Shv(U;\Sp) \arrow[r, twoheadrightarrow] & \Shv(V;\Sp)
    \end{tikzcd}\]
    onto the $\infty$-category of spectral sheaves on $V$ (resp. the canonical projection
    \[\begin{tikzcd}
        p_U:\Shv(X;\Sp) \simeq \Shv(V;\Sp) \times \Shv(U;\Sp) \arrow[r, twoheadrightarrow] & \Shv(U;\Sp)
    \end{tikzcd}\]
    onto the $\infty$-category of spectral sheaves on $U$).

    Now suppose $E \in \Shv(X;\Sp)$ satisfies (\ref{eqn:trivialglobalhomotopygroups}). Since $\unit_V$ is a retract of $\unit$, we have
    \[\Map_{\Shv(X;\Sp)}(\unit_V,\Sigma^nE) \simeq *\]
    for every $n \in \Z$, which in turn gives
    \[\Map_{\Shv(V;\Sp)}(\unit,(\Sigma^nE)\vert_V) \simeq *\]
    for all $n$ by the above characterisation of $\unit_V$.

    We have an equivalence of anima
    \[\Map_{\Shv(V;\Sp)}(\unit,(\Sigma^nE)\vert_V) \simeq \Omega^\infty\Sigma^nE(V),\]
    so the above tells us $\Omega^\infty\Sigma^nE(V)$ is contractible for all $n \in \Z$. Consequently, all the (stable) homotopy groups
    \[\pi_n(E(V)) \simeq 0\]
    are trivial for each $n \in \Z$. Since $\Sp$ is compactly generated by the $\otimes$-unit, we deduce triviality of $E(V)$. Since the clopen subsets of $X$ form a basis for the topology of $X$, we deduce $E \simeq 0$ by the sheaf condition.

    It remains to observe that $\unit \in \Shv(X;\Sp)$ is indeed compact. One can check this directly, but it can also be seen to follow by Bachmann's equivalence
    \[\mathcal{SH}(k)[\rho^{-1}] \simeq \Shv(\Sper(k);\Sp)\]
    (\cite{Bachmann:Real}, \cite{Bachmann:Remarks}) together with \cite{craven1975boolean}, which permits us to write any Boolean $X$ as $\Sper(k)$ for some $k$: since the $\otimes$-unit is clearly compact in $\SH(k)[\rho^{-1}]$, we are done.
\end{proof}

\begin{cor}\label{cor:compactgenerationofSH/cone(rho)}
    The triangulated category $\SH(k)[\rho^{-1}]$ is compactly generated by the unit $\unit$.
\end{cor}
\begin{proof}
    By \cite{Bachmann:Real} and \cite{Bachmann:Remarks}, there is an equivalence of $\infty$-categories
    \[\mathcal{SH}(k)[\rho^{-1}] \simeq \Shv(\Sper(k);\Sp)\]
    between $\rho$-local motivic spectra and spectral sheaves on the real spectrum $\Sper(k)$ of $k$. Since $\Sper(k)$ is a Boolean space, \Cref{prop:spectralsheavescompactgeneration} tells us the homotopy category $\SH(k)[\rho^{-1}]$ is compactly generated by the unit as a triangulated category.
\end{proof}
Before we can get to the proof of \Cref{theorem:compactgenerationofDM/cone(rho)}, we record the following lemma.
\begin{lemma}\label{lemma:descendingcompactgeneration}
    Let $F:\mathcal{C} \longrightarrow \mathcal{D}$ be a triangulated functor between triangulated categories which preserves compact objects and admits a conservative right adjoint $G:\mathcal{D} \longrightarrow \mathcal{C}$. Suppose $\mathcal{C}$ is compactly generated by a set of objects $S \subseteq \mathcal{C}$. Then $\mathcal{D}$ is compactly generated by the set $F(S) \subseteq \mathcal{D}$.
\end{lemma}
\begin{proof}
    Let $X \in \mathcal{D}$ be an object. Suppose for all $F(A) \in F(S)$ and $n \in \Z$,
    \[\Hom_\mathcal{D}(\Sigma^nF(A),X)=0.\]
    Then by adjunction we have
    \[\Hom_\mathcal{C}(\Sigma^n A,G(X))=0\]
    for all $A \in S$, so compact generation of $\mathcal{C}$ by $S$ gives us $G(X) \simeq 0$. $G$ is conservative, so this implies $X \simeq 0$. This proves the lemma.
\end{proof}
\begin{proof}[Proof of \Cref{theorem:compactgenerationofDM/cone(rho)}]
    By \Cref{cor:compactgenerationofSH/cone(rho)}, $\SH(k)[\rho^{-1}]$ is compactly generated by the unit. Consider the following commutative diagram
    \[\begin{tikzcd}
        \SH(k) \arrow[r, twoheadrightarrow, shift left] \arrow[d, shift right, swap, "M"] \arrow[d, hookleftarrow, shift left, "H"] & \arrow[l, hookrightarrow, shift left] \SH(k)[\rho^{-1}] \arrow[d, dashrightarrow] \\
        \DM(k;\F_p) \arrow[r, twoheadrightarrow, shift left] & \arrow[l, hookrightarrow, shift left] \DM(k;\F_p)[\rho^{-1}]
    \end{tikzcd}\]
    where the upper horizontal arrows are left adjoints to the lower horizontal arrows, and the solid left vertical arrow $M$ is left adjoint to the solid right vertical arrow $H$; here, hooked arrows denote faithful functors, and the double-headed arrows denote localisations. The dashed vertical arrow on the right is induced by the universal property of localising at $\rho$: it clearly preserves coproducts, so it has a right adjoint
    \[\begin{tikzcd}
        \DM(k;\F_p)[\rho^{-1}] \arrow[r] & \SH(k)[\rho^{-1}].
    \end{tikzcd}\]
    We claim this right adjoint is conservative. This follows since its composition with the fully faithful right adjoint
    \[\begin{tikzcd}
        \SH(k)[\rho^{-1}] \arrow[r] & \SH(k)
    \end{tikzcd}\]
    agrees with the composition of the right adjoints
    \[\begin{tikzcd}
        \DM(k;\F_p)[\rho^{-1}] \arrow[r] & \DM(k;\F_p) \arrow[r, "H"] & \SH(k),
    \end{tikzcd}\]
    which is conservative.

    We now conclude by \Cref{lemma:descendingcompactgeneration}; the final assertion that we have an equivalence
    \[\DM(k;\F_p)[\rho^{-1}] \simeq \DTM(k;\F_p)[\rho^{-1}]\]
    is now immediate.
\end{proof}

\begin{cor}\label{cor:openofcone(rho)inDATM}
    The open
    \[U(\cone(\rho)) \subseteq \Spc(\DATM_\gm(k;\F_p))\]
    is homeomorphic to the corresponding open in $\Spc(\DTM_\gm(k;\F_p))$.
\end{cor}
\begin{proof}
    \Cref{theorem:compactgenerationofDM/cone(rho)} furnishes an equivalence of tt-categories
    \[(\DATM_\gm(k;\F_p)/\ideal{\cone(\rho)})^\natural \simeq (\DTM_\gm(k;\F_p)/\ideal{\cone(\rho)})^\natural\]
    and so their spectra are equivalent; applying \Cref{remark:principalquotientgiveshomeoontoopen} gives the desired homeomorphism.
\end{proof}
\begin{cor}\label{cor:Spc(DATM)astwopieces}
    Suppose $k$ has a primitive $p^{th}$ root of unity. Then
    \[\Spc(\DATM_\gm(k;\F_p))\]
    is the union of the closed subspace
    \[\supp(\cone(\rho) \otimes \cone(\tau)) \cong \Spc(\DAM_\gm(k;\F_p))\]
    and its complementary open
    \[U(\cone(\rho) \otimes \cone(\tau)),\]
    which is homeomorphic to the corresponding open subspace
    \[U(\cone(\rho) \otimes \cone(\tau)) \cong \Spc(\DTM_\gm(k;\F_p)) \setminus \{0\}\]
    of $\Spc(\DTM_\gm(k;\F_p))$.
\end{cor}
\begin{proof}
    Combine \Cref{remark:open(cone(tau))inDATM}, \Cref{cor:supp(cone(rho)cone(tau))}, and \Cref{cor:openofcone(rho)inDATM}.
\end{proof}
\begin{remark}
    If $k$ does not contain a primitive $p^{th}$ root of unity, the analogous statement still holds if one replaces $\tau$ with $\beta$ from \Cref{prop:hSpec-motcohom-no-root-unity}.
\end{remark}

\subsection{The topology}
Throughout this subsection, we assume that $k$ contains a primitive $p^{th}$ root of unity. By \Cref{cor:Spc(DATM)astwopieces} the spectrum
\[\Spc(\DATM_\gm(k;\F_p))\]
decomposes into a closed piece and an open piece, both of which are already understood. It remains to describe how these two pieces `\emph{glue together}.' To fix notation, for the remainder of this section we shall denote
\[V \coloneqq \supp(\cone(\rho) \otimes \cone(\tau)) \subseteq \Spc(\DATM_\gm(k;\F_p));\]
\[U \coloneqq U(\cone(\rho) \otimes \cone(\tau)) \subseteq \Spc(\DATM_\gm(k;\F_p)).\]
\Cref{cor:Spc(DATM)astwopieces} furnishes homeomorphisms
\[V \cong \Spc(\DAM_\gm(k;\F_p));\]
\[U \cong \Spc(\DTM_\gm(k;\F_p)) \setminus \{\ideal{0}\}.\]
From now on, we shall refer to the subspace $U$ as the \emph{Tate open}.

To understand the topology, we begin by describing the specialisations from $U$ to $V$. We shall denote by
\[a_E \in \Spc(\DATM_\gm(k;\F_p))\]
the restriction of the isotropic point of $\Spc(\DM_\gm(k;\F_p))$ corresponding to a field extension $E/k$.

\begin{prop}\label{prop:specialisationsforcone(tau)}
    The étale point
    \[\mathcal{P}_{\acute{e}t} \coloneqq \ideal{\cone(\tau)} \in \Spc(\DATM_\gm(k;\F_p))\]
    specialises to $a_{\overline{k}}$. Moreover, $\mathcal{P}_{\acute{e}t}$ specialises to no other point in $V \subseteq \Spc(\DATM_\gm(k;\F_p))$.
\end{prop}
\begin{proof}
    Consider the base-change functor from Artin--Tate motives over $k$ to Artin--Tate motives over the algebraic closure $\overline{k}$ of $k$:
    \[\begin{tikzcd}
        F_{\overline{k}}:\DATM_\gm(k;\F_p) \arrow[r] & \DATM_\gm(\overline{k};\F_p).
    \end{tikzcd}\]
    On spectra, this induces a continuous map
    \[\begin{tikzcd}
        \Spc(F_{\overline{k}}):\Spc(\DTM_\gm(\overline{k};\F_p)) \cong \Spc(\DATM_\gm(\overline{k};\F_p)) \arrow[r] & \Spc(\DATM_\gm(k;\F_p)).
    \end{tikzcd}\]
    Recall that by \cite{Gallauer:TateSpectrum} (cf.~\Cref{cor:explicit-tate-generators-topology}), the domain of the above is the Sierpiński space
    \[\begin{tikzcd}[ampersand replacement=\&,sep=scriptsize]
	   {\ideal{0}} \\
	   \\
	   {\ideal{\cone(\tau)}}
	   \arrow[no head, from=1-1, to=3-1]
    \end{tikzcd}\]
    
    The closed point $\ideal{0}$ is the restriction to $\Spc(\DTM_\gm(\overline{k};\F_p))$ of the isotropic point in the spectrum $\Spc(\DM_\gm(\overline{k};\F_p))$ corresponding to the base field $\overline{k}$, whilst the generic point $\ideal{\cone(\tau)}$ is the restriction of the étale point of $\Spc(\DM_\gm(k;\F_p))$. Now, functoriality of the spectrum produces the commutative diagram
    \[\begin{tikzcd}
        \Spc(\DM_\gm(\overline{k};\F_p)) \arrow[r] \arrow[d] & \Spc(\DM_\gm(k;\F_p)) \arrow[d] \\
        \Spc(\DTM_\gm(\overline{k};\F_p)) \arrow[r, "\Spc(F_{\overline{k}})"] & \Spc(\DATM_\gm(k;\F_p))
    \end{tikzcd}\]
    from which we deduce
    \[\Spc(F_{\overline{k}})(\ideal{0})=a_{\overline{k}};\]
    \[\Spc(F_{\overline{k}})(\ideal{\cone(\tau)})=\mathcal{P}_{\acute{e}t}.\]
    Since continuous maps preserve specialisations, we obtain
    \[U \ni \mathcal{P}_{\acute{e}t} \rightsquigarrow a_{\overline{k}} \in V.\]
    We must verify there are no other specialisations from $\mathcal{P}_{\acute{e}t}$ to points in $V$. Consider the commutative diagram
    \[\begin{tikzcd}
        \DATM_\gm(k;\F_p) \arrow[r] & \DATM_\gm(\overline{k};\F_p) \\
        \DAM_\gm(k;\F_p) \arrow[r] \arrow[u, hookrightarrow, "\iota"] & \DAM_\gm(\overline{k};\F_p) \arrow[u, hookrightarrow]
    \end{tikzcd}\]
    By the proof of \Cref{cor:supp(cone(rho)cone(tau))}, these inclusions induce retractions of $\Spc(\DATM_\gm(k;\F_p))$ onto $V$, and likewise for $\overline{k}$ in place of $k$. Now suppose
    \[\mathcal{P}_{\acute{e}t} \rightsquigarrow b \in V\]
    is a specialisation. Applying $\Spc(-)$ to the above yields the commutative square
    \[\begin{tikzcd}
        \Spc(\DATM_\gm(k;\F_p)) \arrow[d, twoheadrightarrow, swap, "\Spc(\iota)"] & \arrow[l] \Spc(\DATM_\gm(\overline{k};\F_p)) \arrow[d, twoheadrightarrow] \\
        \Spc(\DAM_\gm(k;\F_p)) & \arrow[l] \Spc(\DAM_\gm(\overline{k};\F_p))
    \end{tikzcd}\]
    which gives us
    \[\Spc(\iota)(\mathcal{P}_{\acute{e}t})=a_{\overline{k}}.\]
    Since $\mathcal{P}_{\acute{e}t} \rightsquigarrow b$, we have
    \[a_{\overline{k}}=\Spc(\iota)(\mathcal{P}_{\acute{e}t}) \rightsquigarrow \Spc(\iota)(b)=b\]
    by continuity, where we have used that $\Spc(\iota)$ is a retraction onto the subspace $V$ to get the right-hand side equality. But the isotropic point $a_{\overline{k}}$ is closed in $\Spc(\DAM_\gm(k;\F_p))$ by \cite[Proposition~10.8]{Gallauer:Periods}, which forces $b=a_{\overline{k}}$. This concludes the proof.
\end{proof}

\begin{prop}\label{prop:specialisationsforrealpoints}
    Suppose $k$ is formally real and $p=2$. Let $\mathcal{P}^{\alpha} \in U$ be the point in the Tate open corresponding to an ordering $\alpha \in \Sper(k)$ of $k$. If we denote by $k_\alpha$ the real closure of the ordered field $(k,\alpha)$, then
    \[\mathcal{P}^{\alpha} \rightsquigarrow a_{k_\alpha}.\]
    Moreover, $\mathcal{P}^{\alpha}$ specialises to no other point in $V \subseteq \Spc(\DATM_\gm(k;\F_2))$.
\end{prop}
\begin{proof}
    Arguing as in the proof of \Cref{prop:specialisationsforcone(tau)}, consider the commutative diagram
    \[\begin{tikzcd}
        \Spc(\DATM_\gm(k;\F_2)) \arrow[d, twoheadrightarrow, swap, "\Spc(\iota)"] & \arrow[l] \Spc(\DATM_\gm(k_\alpha;\F_2)) \arrow[d, twoheadrightarrow] \\
        \Spc(\DAM_\gm(k;\F_2)) & \arrow[l] \Spc(\DAM_\gm(k_\alpha;\F_2))
    \end{tikzcd}\]
    induced by base change to $k_\alpha$. Since $k_\alpha$ is real closed, we may examine the description of $\Spc(\DATM_\gm(k_\alpha;\F_2))$ from \cite{Balmer:ThreeRealArtinTates} to find the kernel of real realisation $\ideal{\cone(\rho_\alpha)}$ specialises to the (restriction of the) isotropic point associated to the base field $k_\alpha$ in $\Spc(\DATM_\gm(k_\alpha;\F_2))$. From this and the above diagram we deduce
    \[\mathcal{P}^{\alpha} \rightsquigarrow a_{k_\alpha}.\]

    This can be the only specialisation from $\mathcal{P}^{\alpha}$ to a point in $V$. Indeed, if
    \[\mathcal{P}^{\alpha} \rightsquigarrow b \in V\]
    then
    \[a_{k_\alpha}=\Spc(\iota)(\mathcal{P}^{\alpha}) \rightsquigarrow \Spc(\iota)(b)=b\]
    in $\Spc(\DAM_\gm(k;\F_p))$. Another application of \cite[Proposition~10.8]{Gallauer:Periods} finishes the proof.
\end{proof}

For a formally real field $k$ and an ordering $\alpha \in \Sper(k)$, we shall denote by
\[b_\alpha \in V \subseteq \Spc(\DATM_\gm(k;\F_2))\]
the point corresponding to the image of the generic point
\[\ideal{M(\overline{k_\alpha}),\widetilde{M}(\overline{k_\alpha})} \in \Spc(\DAM_\gm(k_\alpha;\F_2)) \subseteq \Spc(\DATM_\gm(k_\alpha;\F_2))\]
(see \cite{Balmer:ThreeRealArtinTates}) under the map on spectra induced by base change to $k_\alpha$. Here, $\widetilde{M}(\overline{k_\alpha})$ is the (completely) reduced Rost motive associated to the pure symbol $\rho \in K^M_1(k_\alpha)/2$: see Appendix~A for a recollection of the definition of such a motive.

\begin{prop}\label{prop:specialisationsforrealtaupoints}
    Suppose $k$ is formally real and $p=2$, and let $\alpha \in \Sper(k)$. The point $\mathcal{P}^{\alpha}_{\acute{e}t}$ in the Tate open $U \subseteq \Spc(\DATM_\gm(k;\F_2))$ specialises to $b_\alpha \in V$. Moreover,
    \[\mathcal{P}^{\alpha}_{\acute{e}t} \rightsquigarrow \mathcal{Q} \in V\]
    if and only if $\mathcal{Q} \in \overline{\{b_\alpha\}}$.
\end{prop}
\begin{proof}
    This time, consider the commutative diagram of spectra
    \[\begin{tikzcd}
        \Spc(\DATM_\gm(k;\F_2)) & \arrow[l, swap, "\Spc(F_{k_\alpha})"] \Spc(\DATM_\gm(k_\alpha;\F_2)) \\
        \Spc(\DAM_\gm(k;\F_2)) \arrow[u, "\Spc(\Psi_k)"] & \arrow[l] \Spc(\DAM_\gm(k_\alpha;\F_2)) \arrow[u, swap, "\Spc(\Psi_{k_\alpha})"]
    \end{tikzcd}\]
    where $\Psi_E = \operatorname{gr} \circ t$ denotes the section of the inclusion from Artin--Tate motives over $E$ to Artin motives over $E$, and
    \[\begin{tikzcd}
        F_{k_\alpha}:\DATM_\gm(k;\F_2) \arrow[r] & \DATM_\gm(k_\alpha;\F_2)
    \end{tikzcd}\]
    is the base change functor from $k$ to $k_\alpha$.
    Recall from \cite{Balmer:ThreeRealArtinTates} we have a specialisation
    \[\ideal{\cone(\rho_\alpha),\cone(\tau)} \rightsquigarrow \ideal{M(\overline{k_\alpha}),\widetilde{M}(\overline{k_\alpha})}\]
    in $\Spc(\DATM_\gm(k_\alpha;\F_2))$. Since
    \[\Spc(F_{k_\alpha})(\ideal{\cone(\rho_\alpha),\cone(\tau)})=\mathcal{P}^{\alpha}_{\acute{e}t}\]
    and
    \[\Spc(F_{k_\alpha})(\ideal{M(\overline{k_\alpha}),\widetilde{M}(\overline{k_\alpha})})=b_\alpha\]
    by definition, we have our sought after specialisation $\mathcal{P}^{\alpha}_{\acute{e}t} \rightsquigarrow b_\alpha$.

    For the second statement, suppose $\mathcal{P}^{\alpha}_{\acute{e}t}$ admits a specialisation to some $\mathcal{Q} \in V$. Commutativity of
    \[\begin{tikzcd}
        \Spc(\DATM_\gm(k;\F_2)) \arrow[d, twoheadrightarrow, swap, "\Spc(\iota)"] & \arrow[l] \Spc(\DATM_\gm(k_\alpha;\F_2)) \arrow[d, twoheadrightarrow] \\
        \Spc(\DAM_\gm(k;\F_2)) & \arrow[l] \Spc(\DAM_\gm(k_\alpha;\F_2))
    \end{tikzcd}\]
    coupled with the fact that $\ideal{\cone(\rho_\alpha),\cone(\tau)}$ gets sent to the generic point $\ideal{M(\overline{k_\alpha}),\widetilde{M}(\overline{k_\alpha})}$ of $\Spc(\DAM_\gm(k_\alpha;\F_2))$ under the map induced by the inclusion tells us
    \[\Spc(\iota)(\mathcal{P}^{\alpha}_{\acute{e}t})=b_\alpha.\]
    It follows
    \[b_\alpha=\Spc(\iota)(\mathcal{P}^{\alpha}_{\acute{e}t}) \rightsquigarrow \Spc(\iota)(\mathcal{Q})=\mathcal{Q},\]
    which is to say $\mathcal{Q} \in \overline{\{b_\alpha\}}$.
\end{proof}

\Cref{prop:specialisationsforcone(tau)}, \Cref{prop:specialisationsforrealpoints} and \Cref{prop:specialisationsforrealtaupoints} determine all the specialisation relations in $\Spc(\DATM_\gm(k;\F_p))$: any specialisation relation is either internal to $U$ or $V$, or it is of one of the forms described above.

Analysing the proofs of the preceding three propositions, we obtain the following
\begin{cor}\label{cor:creationofspecialisations}
    The specialisations in $\Spc(\DATM_\gm(k;\F_p))$ from $U$ to $V$ are those created by the inclusion
    \[\begin{tikzcd}
        \iota:\DAM_\gm(k;\F_p) \arrow[r, hookrightarrow] & \DATM_\gm(k;\F_p),
    \end{tikzcd}\]
    that is, for any $\mathcal{P} \in U$ and $\mathcal{Q} \in V$, we have
    \[\mathcal{P} \rightsquigarrow \mathcal{Q}\]
    if and only if
    \[\Spc(\iota)(\mathcal{P}) \rightsquigarrow \Spc(\iota)(\mathcal{Q})=\mathcal{Q}.\]
\end{cor}

We are now ready to deduce the topology of $\Spc(\DATM_\gm(k;\F_p))$.

\begin{prop}\label{prop:topologyofSpc(DATM)}
    A subset $Z \subseteq \Spc(\DATM_\gm(k;\F_p))$ is closed if and only if
    \begin{itemize}
        \item[$(i)$] $Z \cap V$ is closed in $V$;
        \item[$(ii)$] $Z \cap U$ is closed in $U$;
        \item[$(iii)$] $Z$ is closed under specialisations.
    \end{itemize}
\end{prop}
\begin{proof}
    Clearly conditions $(i)$-$(iii)$ are necessary to be closed; we must show they are sufficient.
    To this end, consider the maps on spectra induced by the inclusions of Tate and Artin motives into Artin--Tate motives:
    \[\begin{tikzcd}
        \Spc(\DTM_\gm(k;\F_p)) & \arrow[l, swap, "g"] \Spc(\DATM_\gm(k;\F_p)) \arrow[r, "f"] & \Spc(\DAM_\gm(k;\F_p))
    \end{tikzcd}\]
    Let $Z$ be a subset of $\Spc(\DATM_\gm(k;\F_p))$ satisfying $(i)$-$(iii)$, and write
    \[Z_U \coloneqq Z \cap U;\]
    \[Z_V \coloneqq Z \cap V.\]
    If $Z_U = \emptyset$, then $Z$ is clearly closed by condition $(i)$, so let us assume $Z_U$ is non-empty. We claim that
    \begin{equation}\label{eqn:closedinDATM}
        Z=f^{-1}(Z_V) \cap g^{-1}(\overline{Z_U}).
    \end{equation}

    By condition $(ii)$, $Z_U$ is closed in the Tate open, and so because it is non-empty we have
    \[\overline{Z_U}=Z_U \cup \{\ideal{0}\}.\]
    Moreover, since $g^{-1}(\{\ideal{0}\})=V$, it immediately follows that
    \[g^{-1}(\overline{Z_U}) \cap U = Z_U\]
    and
    \[g^{-1}(\overline{Z_U}) \supseteq V.\]
    On the other hand,
    \[f^{-1}(Z_V) \cap V = Z_V.\]
    Furthermore, if $\mathcal{P} \in Z_U$, then
    \[\mathcal{P} \rightsquigarrow f(\mathcal{P}) \in V\]
    by \Cref{cor:creationofspecialisations}. By condition $(iii)$, we deduce
    \[f(\mathcal{P}) \in Z_V.\]
    Hence
    \[Z_U \subseteq f^{-1}(Z_V).\]
    
    From this we deduce the equality (\ref{eqn:closedinDATM}), since the above tells us it holds after intersecting with $U$ and $V$ separately, which cover $\Spc(\DATM_\gm(k;\F_p))$. The assumptions on $Z$ and continuity of $f,g$ give that $f^{-1}(Z_V)$ and $g^{-1}(\overline{Z_U})$ are closed, so we conclude $Z$ is too.
\end{proof}

\begin{theorem}\label{theorem:ArtinTatespectrumwithrootofunity}
    \Cref{theorem:ArtinTatespectrum} holds when $k$ contains a primitive $p^{th}$ root of unity.
\end{theorem}
\begin{proof}
    This is \Cref{prop:specialisationsforcone(tau)}, \Cref{prop:specialisationsforrealpoints}, \Cref{prop:specialisationsforrealtaupoints}, and \Cref{prop:topologyofSpc(DATM)}.
\end{proof}

\subsection{Removing the root of unity assumption}
Suppose now that $k$ does not contain a primitive $p^{th}$ root of unity, and let
\[k' \coloneqq k(\zeta_p)\]
be the field obtained by adjoining one to $k$. Note $p$ is necessarily odd, since $\operatorname{char}(k) \neq p$.

Recall we have the adjunction
\[\begin{tikzcd}
    f^*:\DM_\gm(k;\F_p) \arrow[r, shift left] & \arrow[l, shift left] \DM_\gm(k';\F_p):f_*
\end{tikzcd}\]
where the left adjoint $f^*$ is given by base change: it is faithful by \Cref{prop:adoining-root-unity-faithful}. Letting $d \coloneqq [k':k]$, we know by \Cref{prop:hSpec-motcohom-no-root-unity} that there is a unique class
\[\beta \in H_\mathcal{M}^{0,d}(k;\F_p)\]
such that $f^*(\beta)=\tau^d \in H_\mathcal{M}^{0,d}(k';\F_p)$.

Note we still have the Milnor weight structure on $\DATM_\gm(k;\F_p)$. The main modification we need to make is to replace \Cref{prop:detectingnilpotenceoncone(tau)forDATM} with the following analogue.

\begin{prop}\label{prop:nilpotenceonbeta}
    Let $g:X \longrightarrow Y$ be a morphism in $\DATM_\gm(k;\F_p)$. If $t(g)=0$, then $g$ is $\otimes$-nilpotent on $\cone(\beta)$. In particular, the Milnor weight complex functor detects $\otimes$-nilpotence on $\cone(\beta)$.
\end{prop}
\begin{proof}
    Observe the functor $f^*$ is weight exact with respect to the Milnor weight structure: this gives rise to the commutative diagram
    \[\begin{tikzcd}
        \DATM_\gm(k;\F_p) \arrow[d, swap, "f^*"] \arrow[r, "t_k"] & \K^b(\mathcal{H}^\mathsf{AT}_k) \arrow[d] \\
        \DATM_\gm(k';\F_p) \arrow[r, "t_{k'}"] & \K^b(\mathcal{H}^\mathsf{AT}_{k'})
    \end{tikzcd}\]
    where the horizontal arrows are the weight complex functors, and all functors are weight exact. Here we have chosen to denote the heart of the Milnor weight structure on $\DATM_\gm(E;\F_p)$ by $\mathcal{H}^\mathsf{AT}_E$, and by $t_E$ the associated weight complex functor.

    Now let $g:X \longrightarrow Y$ be as in the statement. We wish to show that $g$ is $\otimes$-nilpotent on $\cone(\beta)$. But, as $f^*$ is a faithful tt-functor, it suffices to prove that $f^*(g)$ is $\otimes$-nilpotent on  $f^*(\cone(\beta))$. We have $f^*(\cone(\beta)) \simeq \cone(f^*(\beta)) = \cone(\tau^d)$ by exactness of $f^*$ and construction of $\beta$. The condition $t_k(g)=0$ gives $t_{k'}(f^*(g))=0$ by commutativity of the above diagram. We now apply \Cref{prop:detectingnilpotenceoncone(tau)forDATM} to $f^*(g)$ to deduce the desired claim: note $f^*(g)$ being $\otimes$-nilpotent on $\cone(\tau)$ implies $f^*(g)$ is $\otimes$-nilpotent on $\cone(\tau^d)$. 
\end{proof}

\begin{theorem}\label{theorem:ArtinTatespectrumwithnorootofunity}
    \Cref{theorem:ArtinTatespectrum} holds when $k$ does not contain a primitive $p^{th}$ root of unity.
\end{theorem}
\begin{proof}
    $\Spc(\DATM_\gm(k;\F_p))$ is partitioned into the closed subset
    \[V \coloneqq \supp(\cone(\beta))\]
    and its open complement
    \[U \coloneqq U(\cone(\beta)).\]
    Observe $t(\beta)=0$, so \Cref{prop:detectingnilpotenceonanobject} together with \Cref{prop:nilpotenceonbeta} and \Cref{prop:grdetectsnilpotenceoncone(rho)} gives
    \[V=\im(\Spc(\Psi))\]
    where $\Psi \coloneqq \operatorname{gr} \circ t$. To see $\Spc(\Psi)$ is a homeomorphism onto $V$, the same argument as for the case $k=k'$ works: $\Psi$ is a section of the inclusion of $\DAM_\gm(k;\F_p)$ into $\DATM_\gm(k;\F_p)$. Observe this also gives
    \[V \cong \Spc(\DAM_\gm(k;\F_p)).\]

    It remains to identify $U$, which we note is non-empty. Faithfulness of $f^*$ induces a surjection
    \[\begin{tikzcd}
        \Spc(f^*):\Spc(\DATM_\gm(k';\F_p)) \arrow[r, twoheadrightarrow] & \Spc(\DATM_\gm(k;\F_p))
    \end{tikzcd}\]
    on spectra. Since $f^*(\beta)=\tau^d$, this induces a surjection
    \[\begin{tikzcd}
        U(\cone(\tau)) = U(\cone(\tau^d)) \arrow[r, twoheadrightarrow] & U(\cone(\beta)) = U.
    \end{tikzcd}\]
    As $p$ is odd, $U(\cone(\tau))$ is a singleton, and hence so is $U$.
    
    It remains to determine the specialisation relations from the étale point
    \[\mathcal{P}^{\beta} \in U \subseteq \Spc(\DATM_\gm(k;\F_p))\]
    to points in $V$. The same argument as in the proof of \Cref{prop:specialisationsforcone(tau)} gives that the only specialisation from this point is to the (restriction of the) isotropic point corresponding to the algebraic closure $\overline{k}$ of $k$.

    Finally, the proof of \Cref{prop:topologyofSpc(DATM)} goes through verbatim, and we conclude.
\end{proof}

\begin{proof}[Proof of \Cref{theorem:ArtinTatespectrum}]
    This is \Cref{theorem:ArtinTatespectrumwithrootofunity} and \Cref{theorem:ArtinTatespectrumwithnorootofunity}.
\end{proof}
\section{Remarks on the spectrum of geometric motives}
We conclude with some consequences of the preceding computations. Let $k$ be a field with characteristic different from $p$, not necessarily containing a primitive $p^{th}$ root of unity. Choose a root $\zeta_p$ in a separable closure of $k$ and let $d=[k(\zeta_p): k]$. Let $\beta: \unit \longrightarrow \unit(d)$ be the class constructed in \Cref{prop:hSpec-motcohom-no-root-unity}, with $\beta=\tau$ when $k$ contains a primitive $p^{th}$ root of unity.

\begin{theorem}\label{theorem:TateopeninSpc(DM)}
    The spectrum
    \[\Spc(\DM_\gm(k;\F_p))\]
    contains a homeomorphic copy of
    \[\Spc(\DTM_\gm(k;\F_p)) \setminus \{\ideal{0}\} \cong \Proj(H_\mathcal{M}^{*,*}(k;\F_p)),\]
    given by
    \[U(\cone(\rho)) \cup U(\cone(\beta))=U(\cone(\rho) \otimes \cone(\beta)).\]
    In particular, the punctured Tate spectrum is an open subset of $\Spc(\DM_\gm(k;\F_p))$.
\end{theorem}
\begin{proof}
If $k$ contains a primitive $p^{th}$ root of unity, combine \Cref{cor:tau-open-tates}, \Cref{rem:open-tau-DM}, and \Cref{theorem:compactgenerationofDM/cone(rho)}. Otherwise, $p$ is odd, so $U(\cone(\rho))=\emptyset$. By \Cref{prop:adoining-root-unity-faithful}, base change to $k(\zeta_p)$ is faithful, thus $U(\cone(\tau)) \subseteq \Spc(\DM_\gm(k(\zeta_p);\F_p))$ surjects onto $U(\cone(\beta))\subseteq \Spc(\DM_\gm(k;\F_p))$. This implies that the latter must be a singleton, which proves the claim.
\end{proof}

We first observe that, when $k$ is a formally real field with infinitely many orderings, the spectrum of (Tate) motives with $\F_2$-coefficients is non-Noetherian.

\begin{lemma}\label{lem:infinite-sper-non-noetherian}
If $k$ has infinitely many orderings, then $\Sper(k)$ is not
Noetherian.
\end{lemma}
\begin{proof}
Note that $\Sper(k)$ is Hausdorff, so in particular every irreducible subset is a singleton. However, a Noetherian space has only finitely many
irreducible components. Since $\Sper(k)$ is infinite, it therefore cannot be Noetherian.
\end{proof}

\begin{cor}
\label{cor:non-noetherian-motivic-spectra}
If $k$ has infinitely many orderings, then $\Spc(\DTM_\gm(k;\F_2))$ and $\Spc(\DM_\gm(k;\F_2))$ are not Noetherian.
\end{cor}
\begin{proof}
By \Cref{cor:tau-open-tates} and \Cref{rem:open-tau-DM} (cf.~\Cref{theorem:TateopeninSpc(DM)}), $\Sper(k)$ identifies with an open subset of
\[U(\cone(\tau))\]
in each of the two spectra (it corresponds to $U(\cone(\rho) \oplus \cone(\tau))$). Since every open subspace of a Noetherian space is Noetherian, the result follows from \Cref{lem:infinite-sper-non-noetherian}.
\end{proof}

Our second application is that the cardinality of the spectrum of the motivic category may be arbitrarily large. This is a consequence of the following result.

\begin{theorem}[\cite{craven1975boolean}]\label{thm:craven-boolean-realisation}
For every Boolean space $X$, there exists a formally real field $k$
such that
\[\Sper(k) \cong X.\]
\end{theorem}
\begin{cor}\label{cor:arbitrarily-large-balmer-motivic-spectrum}
For every infinite cardinal $\lambda$, there exists a field $k$ such that
\[\vert \Spc(\DM_\gm(k;\Z)) \vert > \lambda.\]
\end{cor}
\begin{proof}
For an infinite discrete set $Y$ of cardinality $\lambda$, its Stone--Čech compactification $\beta Y$ is a Boolean space of cardinality no less than $\lambda$. Thus, by \Cref{thm:craven-boolean-realisation}, we may choose a field $k$ such that the cardinality of $\Sper(k)$ exceeds the cardinality of $\lambda$. Finally, the following inequality
\[\lambda < \vert \Sper(k) \vert \leq \vert \Spc(\DTM_\gm(k;\Z)) \vert \leq \vert \Spc(\DM_\gm(k;\Z)) \vert\]
proves the claim.
\end{proof}
\appendix
\section{Generalised Periodicity}
In \cite{Gallauer:Periods}, Gallauer introduced a notion of \emph{periodicity} and investigated it in the setting of tensor-triangular geometry. He successfully applied these tools to deepen our understanding of the structure of the spectrum of Artin motives $\Spc(\DAM_\gm(k;\F))$: for instance, the fact that the restriction of the isotropic points in $\Spc(\DM_\gm(k;\F_p))$ to $\Spc(\DAM_\gm(k;\F_p))$ are closed relies on the proof that the periodic locus of $\Spc(\DAM_\gm(k;\F_p))$ is the complement of the closed points, as proved by Gallauer--Miller in Appendix~B of \emph{loc. cit.}. Let us recall Gallauer's definition of periodicity:

\begin{definition}\label{def:Gallauerperiodicity}
    A tt-category $\mathcal{K}$ has \emph{period} $d \in \Z_{>0}$ if there exists an isomorphism
    \[\unit \simeq \Sigma^d\unit\]
    in $\mathcal{K}$. If $\mathcal{K}$ has period $d$ for some $d>0$, then we say $\mathcal{K}$ is \emph{periodic}, and we refer to the minimal such $d$ as its \emph{periodicity}. Otherwise, we say $\mathcal{K}$ is \emph{non-periodic}, and of period $0$.
\end{definition}

One then defines the \emph{periodicity} of a point $\mathcal{P} \in \Spc(\mathcal{K})$ as the periodicity of the Verdier localisation $\mathcal{K}/\mathcal{P}$. In \emph{loc. cit.}, these periodic loci proved very fruitful in categorising different types of points in tt-categories. Moreover, \emph{propagation of periodicity} allows one to efficiently deduce information regarding specialisation relations amongst points of the spectrum.

In this appendix, we extend Gallauer's theory to the setting of periodicity with respect to arbitrary invertible objects $u \in \Pic(\mathcal{K})$ in place of $\Sigma\unit \in \Pic(\mathcal{K})$. Our motivation stems from the fact that in motivic settings, there are two clear notions of twisting/shifting: by the usual suspension, and by the Tate twist. It is thus natural to develop the theory of periodicity in the generality that permits one to include such non-trivial twists (and of course, this notion of generalised periodicity is a natural extension in other contexts, not just in the motivic setting). The reason this extension is useful is that, as we will see, there are examples of motivic categories for which the periodicity in the sense of \Cref{def:Gallauerperiodicity} does not distinguish points in their entirety, but for which our generalised notion of periodicity \emph{does} distinguish points, that is, every distinct point has a different periodicity. For instance, the spectrum of motives of quadrics $\Spc(\DQM_\gm(F;\F_p))$ over a real closed field $F$ enjoys this property, see \Cref{prop:periodisinjectiveforrealquadrics}. On the other hand, we will see that even this more general notion of periodicity can fail to distinguish points: the spectrum of Tate motives $\Spc(\DTM_\gm(k;\F_p))$ has its points distinguished by their periodicity if and only if $p$ is odd or $k$ admits at most one ordering.

We begin by developing the key foundational aspects of generalised periodicity, and conclude by exploring the handful of motivic examples we have at our disposal. To make our definition of generalised periodicity, we adopt a parallel perspective of \Cref{def:Gallauerperiodicity}, also discussed in \cite{Gallauer:Periods}.

\begin{example}
    Let $\mathcal{K}$ be the homotopy category of the compact part of a stable presentably symmetric monoidal $\infty$-category $\mathcal{C} \in \CAlg(\Pr^L_{st})$. By the universal property of spectra $\Sp$, there exists an essentially unique cocontinuous tt-functor
    \[\begin{tikzcd}
        F:\Sp \arrow[r] & \mathcal{C},
    \end{tikzcd}\]
    and the periodicity $d$ of $\mathcal{K}$ of \Cref{def:Gallauerperiodicity} coincides with the unique non-negative generator of the kernel of the induced map on Picard groups:
    \[\begin{tikzcd}
        (d) = \ker(\Pic(F):\Z \simeq \Pic(\Sp) \arrow[r] & \Pic(\mathcal{C}))
    \end{tikzcd}\]
    (recall the Picard group $\Pic(\Sp)$ is generated by the suspension of the unit $\Sigma\unit$).
\end{example}

\begin{definition}\label{def:generalisedperiodicity}
    Let $\mathcal{K}$ be a tt-category, and let $\mathcal{P} \in \Spc(\mathcal{K})$ be a point of its spectrum. Denote by
    \[\begin{tikzcd}
        q:\mathcal{K} \arrow[r] & \mathcal{K}/\mathcal{P}.
    \end{tikzcd}\]
    the Verdier localisation of $\mathcal{K}$ by $\mathcal{P}$. We define the \emph{periodicity} of $\mathcal{P} \in \Spc(\mathcal{K})$ to be the subgroup
    \[\begin{tikzcd}
        \ker(\Pic(q):\Pic(\mathcal{K}) \arrow[r] & \Pic(\mathcal{K}/\mathcal{P})) \subseteq \Pic(\mathcal{K}),
    \end{tikzcd}\]
    which we denote $\per_\mathcal{K}(\mathcal{P}) \subseteq \Pic(\mathcal{K})$. We call $u \in \Pic(\mathcal{K})$ a \emph{period} of $\mathcal{P} \in \Spc(\mathcal{K})$ if $u \in \per_\mathcal{K}(\mathcal{P})$.
\end{definition}

\begin{remark}
    We remark that, unlike in \Cref{def:Gallauerperiodicity}, we are unable to speak of the periodicity of a \emph{category} $\mathcal{K}$: we can only speak of the periodicity of the \emph{points} of its spectrum. The reason for this is that our definition of generalised periodicity is a relative notion: it depends on the Picard group of a source category, namely $\mathcal{K}$. Moreover, since the universal example---the $\infty$-category of spectra---has $\Pic(\Sp) \simeq \Z$ generated by $\Sigma\unit$, any notion of periodicity different from \cite{Gallauer:Periods} needs to be relative in some sense.
\end{remark}

\begin{example}\label{example:biperiodicity}
    Consider $\DM_\gm(k;R)$. The Tate motives $\unit(i)[j] \in \Pic(\DM_\gm(k;R))$ provide a natural notion of biperiodicity for $\DM_\gm(k;R)$: the \emph{biperiodicity} of a point $\mathcal{P} \in \Spc(\DM_\gm(k;R))$ is the intersection of $\per_{\DM_\gm(k;R)}(\mathcal{P})$ with the subgroup $\Z^2 \simeq \ideal{\unit(i)[j] \mid i,j \in \Z}$ of $\Pic(\DM_\gm(k;R))$ generated by the Tate motives:
    \[\per_{\DM_\gm(k;R)}(\mathcal{P}) \cap \Z^2 \subseteq \Z^2.\]
    Equivalently, the biperiodicity of a point $\mathcal{P}$ is the kernel of the map
    \[\begin{tikzcd}
        \Z^2 \arrow[r] & \Pic(\DM_\gm(k;R)) \arrow[r, "\Pic(q)"] & \Pic(\DM_\gm(k;R)/\mathcal{P})
    \end{tikzcd}\]
    where the first map sends $(a,b)$ to the class of $\unit(a)[b]$ in $\Pic(\DM_\gm(k;R))$.
    
    We say $(a,b) \in \Z^2$ is a \emph{biperiod} of $\mathcal{P} \in \Spc(\DM_\gm(k;R))$ if the corresponding element resides in $\per_{\DM_\gm(k;R)}(\mathcal{P}) \cap \Z^2 \subseteq \Z^2$.

    Of course, the notion of biperiodicity holds for any subcategory of $\DM_\gm(k;R)$ which contains $\DTM_\gm(k;R)$.
\end{example}

Before going further into examples, we state some results concerning the theory of generalised periodicity: these are mostly direct generalisations of the results from \cite{Gallauer:Periods}.

\begin{lemma}\label{lemma:periodicityalongttfunctors}
    Let $F: \mathcal{K} \longrightarrow \mathcal{L}$ be a tt-functor between tt-categories $\mathcal{K},\mathcal{L}$. Then for $\mathcal{P} \in \Spc(\mathcal{L})$, we have
    \[\per_{\mathcal{K}}(\Spc(F)(\mathcal{P})) \subseteq \Pic(F)^{-1}(\per_{\mathcal{L}}(\mathcal{P})).\]
\end{lemma}
\begin{proof}
    We have a commutative diagram
    \[\begin{tikzcd}
        \mathcal{K} \arrow[d] \arrow[r, "F"] & \mathcal{L} \arrow[d] \\
        \mathcal{K}/F^{-1}(\mathcal{P}) \arrow[r] & \mathcal{L}/\mathcal{P}
    \end{tikzcd}\]
    which upon taking Picard groups produces the commutative diagram
    \[\begin{tikzcd}
        \Pic(\mathcal{K}) \arrow[d] \arrow[r, "\Pic(F)"] & \Pic(\mathcal{L}) \arrow[d] \\
        \Pic(\mathcal{K}/F^{-1}(\mathcal{P})) \arrow[r] & \Pic(\mathcal{L}/\mathcal{P})
    \end{tikzcd}\]
    Since, by definition, the kernels of the vertical arrows are the periodicities of the points $\Spc(F)(\mathcal{P})$ and $\mathcal{P}$ respectively, it is clear that $x \in \per_{\mathcal{K}}(\Spc(F)(\mathcal{P}))$ implies $F(x) \in \per_{\mathcal{L}}(\mathcal{P})$.
\end{proof}

\begin{lemma}\label{lemma:propagationforgeneralisedperiodicity}
    Let $\mathcal{K}$ be a tt-category, and let $\mathcal{P},\mathcal{Q} \in \Spc(\mathcal{K})$ be two points of its spectrum. Suppose we have a specialisation relation $\mathcal{P} \rightsquigarrow \mathcal{Q}$. Then
    \[\per_\mathcal{K}(\mathcal{Q}) \subseteq \per_\mathcal{K}(\mathcal{P}).\]
\end{lemma}
\begin{proof}
    The existence of a specialisation $\mathcal{P} \rightsquigarrow \mathcal{Q}$ is equivalent to the existence of a factorisation of the Verdier quotient for $\mathcal{P}$ by that of $\mathcal{Q}$:
    \[\begin{tikzcd}
        \mathcal{K} \arrow[r, "q_\mathcal{Q}"] \arrow[rd, swap, "q_\mathcal{P}"] & \mathcal{K}/\mathcal{Q} \arrow[d, "\exists F"] \\
        & \mathcal{K}/\mathcal{P}
    \end{tikzcd}\]
    It follows
    \[\per_\mathcal{K}(\mathcal{P}) = \ker(\Pic(q_\mathcal{P})) = \ker(\Pic(F) \circ \Pic(q_\mathcal{Q})) \supseteq \ker(\Pic(q_\mathcal{Q})) = \per_\mathcal{K}(\mathcal{Q}).\]
\end{proof}

Let us denote by $\mathrm{S}({\Pic(\mathcal{K})})$ the poset of subgroups of the Picard group of $\mathcal{K}$, ordered by $\supseteq$. By construction, generalised periodicity gives a map of sets
\[\begin{tikzcd}
    \per_{\mathcal{K}}: \Spc(\mathcal{K}) \arrow[r] & \mathrm{S}({\Pic(\mathcal{K})}).
\end{tikzcd}\]
which we call the \emph{(generalised) period map}. Note that, when we consider $\Spc(\mathcal{K})$ as a poset whose order is given by $\rightsquigarrow$, this function is order-preserving by virtue of \Cref{lemma:propagationforgeneralisedperiodicity}.

For a given subgroup $H \subseteq \Pic(\mathcal{K})$, we consider the following subset of $\mathrm{S}(\Pic(\mathcal{K}))$:
\[U(H) = \{L \in \mathrm{S}(\Pic(\mathcal{K})) \mid L \supseteq H\}.\]

\begin{definition}
    Let $H \subseteq \Pic(\mathcal{K})$ be a subgroup. We define the \emph{$H$-periodic locus} $\per_H(\mathcal{K})$ of $\Spc(\mathcal{K})$ via
    \[\per_H(\mathcal{K}) \coloneqq \per_\mathcal{K}^{-1}(U(H)) \subseteq \Spc(\mathcal{K}).\]
\end{definition}
\begin{remark}
    Unwinding the definition, we see that the $H$-periodic locus of $\Spc(\mathcal{K})$ consists of those points $\mathcal{P} \in \Spc(\mathcal{K})$ for which $H \subseteq \per_\mathcal{K}(\mathcal{P})$.
\end{remark}
In \cite{Gallauer:Periods}, it was demonstrated that the Alexandrov topology on $\mathrm{S}(\Z) \cong \Z_{\geq 0}$, generated by the basic opens $U(H)$ for $H \in \mathrm{S}(\Z)$, 
makes the period map $\per_\mathcal{K}:\Spc(\mathcal{K}) \longrightarrow \Z_{\geq 0}$ of \emph{loc. cit.} continuous. It turns out that in our generality this can fail.
\begin{example}\label{ex: failure-continuity-quadrics}
Let $F$ be a real closed field and $\mathcal{K}=\DQM_\gm(F;\F_2)$. Recall from \cite{schemeil2026motivic} that $\Spc(\mathcal{K})$ is the following space
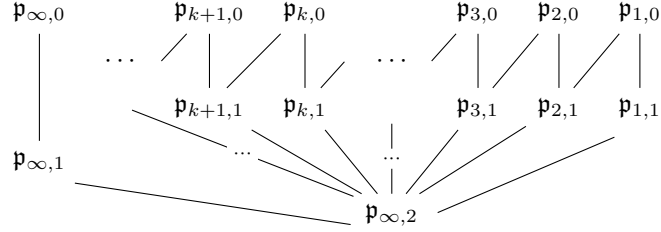
\[
\begin{tikzcd}[ampersand replacement=\&,sep=tiny]
    {\p_{\infty,0}} \& \&
    {\p_{k+1,0}} \& {\p_{k,0}} \& {} \&
    {\p_{3,0}} \& {\p_{2,0}} \& {\p_{1,0}} \\
    \& \cdots \& \& \& \cdots \\
    \& {} \&
    {\p_{k+1,1}} \& {\p_{k,1}} \& {} \&
    {\p_{3,1}} \& {\p_{2,1}} \& {\p_{1,1}} \\
    {\p_{\infty,1}} \\
    \& \& \& \& {\p_{\infty,2}}
    \arrow[no head, from=1-1, to=4-1]
    \arrow[between={0}{0.4}, no head, from=1-3, to=3-2]
    \arrow[no head, from=1-3, to=3-3]
    \arrow[no head, from=1-4, to=3-3]
    \arrow[no head, from=1-4, to=3-4]
    \arrow[between={0}{0.4}, no head, from=1-6, to=3-5]
    \arrow[no head, from=1-6, to=3-6]
    \arrow[no head, from=1-7, to=3-6]
    \arrow[no head, from=1-7, to=3-7]
    \arrow[no head, from=1-8, to=3-7]
    \arrow[no head, from=1-8, to=3-8]
    \arrow["\cdots"{description}, no head, from=3-2, to=5-5]
    \arrow[no head, from=3-3, to=5-5]
    \arrow[between={0}{0.4}, no head, from=3-4, to=1-5]
    \arrow[no head, from=3-4, to=5-5]
    \arrow[no head, from=3-6, to=5-5]
    \arrow[no head, from=3-7, to=5-5]
    \arrow[shift right=3, no head, from=4-1, to=5-5]
    \arrow["\cdots"{description, pos=0.3},
        between={0}{0.6}, no head, from=5-5, to=2-5]
    \arrow[shift right=3, no head, from=5-5, to=3-8]
\end{tikzcd}
\]
where a subset $Z$ is closed if and only if it is specialisation closed and is finite whenever $\p_{\infty,0} \not \in Z$. We show that $\Pic(\mathcal{K})$ contains a subgroup $L$ whose periodic locus is not open, from which it immediately follows that the period map is not continuous with respect to the Alexandrov topology. For $n \geq 1$, let $M_n$ be the Rost motive associated with the pure symbol $\rho^n \in K^M_n(F)/2$, and let $R_n=\fib(M_n \to \unit)$; in particular, $R_n$ is the reduced motive of an affine quadric, hence it is $\otimes$-invertible by \cite{Bachmann:Picard}. We shall use the following diamond
\[\begin{tikzcd}[ampersand replacement=\&,sep=small]
	\&\& \unit \\
	\\
	{M_{n}} \&\& {R_n} \&\& {\widetilde{M}_n} \\
	\\
	\&\& {\unit\{d_{n-1}\}}
	\arrow["{[1]}", from=1-3, to=3-3]
	\arrow["{[1]}", from=1-3, to=3-5]
	\arrow[""{name=0, anchor=center, inner sep=0}, from=3-1, to=1-3]
	\arrow[from=3-3, to=3-1]
	\arrow[from=3-3, to=3-5]
	\arrow[""{name=1, anchor=center, inner sep=0}, "{[1]}", from=3-5, to=5-3]
	\arrow[from=5-3, to=3-1]
	\arrow[from=5-3, to=3-3]
	\arrow["\star"'{pos=0.6}, shift left=3, draw=none, from=0, to=3-3]
	\arrow["\star", shift right=3, draw=none, from=1, to=3-3]
\end{tikzcd}\]
of Vishik, where $d_{n}=2^{n}-1$, and $\widetilde{M}_n$ is the corresponding (completely) reduced Rost motive. In this diagram, a star $\star$ indicates a distinguished triangle; the remaining triangles are commutative.

Consider the subgroup
\[L \coloneqq \left\langle R_n[1] \mid n \geq 1\right\rangle \subseteq \Pic(\DQM_\gm(F; \F_2)).\]
We claim that its periodic locus is given by
\[\per_L(\mathcal{K})=\{\p_{\infty,0}, \p_{\infty,1}, \p_{\infty,2}\}.\]
Since
\[\p_{\infty,0}=\left\langle M_n \mid n \geq 1 \right\rangle,\]
the above diamond shows that every $R_n[1]$ becomes isomorphic to the unit $\unit$ in $\mathcal{K}/\p_{\infty,0}$. In particular, $L \subseteq \per_{\mathcal{K}}(\p_{\infty,0})$; moreover, the specialisations
\[\p_{\infty,2} \rightsquigarrow \p_{\infty,1} \rightsquigarrow \p_{\infty,0}\]
and \Cref{lemma:propagationforgeneralisedperiodicity} give
\[\{\p_{\infty,0}, \p_{\infty,1}, \p_{\infty,2}\} \subseteq \per_L(\mathcal{K}).\]
For the reverse inclusion, note that by \Cref{lemma:propagationforgeneralisedperiodicity} it suffices to show that for every finite $n$, $\p_{n,1} \not \in \per_L(\mathcal{K})$. We begin with the case $n=1$: recall from \cite{schemeil2026motivic} that $\p_{1,1}$ is the kernel of base change to the algebraic closure followed by étale realisation:
\[\begin{tikzcd}
    F_1:\mathcal{K} \arrow[r] & \DTM_\gm(\overline{F};\F_2) \arrow[r] & \DTM_\et^\vee(\overline{F};\F_2).
\end{tikzcd}\]
In particular $F_1(R_1[1]) \simeq \unit_{\acute{e}t}[1]$, which is not isomorphic to $\unit_{\acute{e}t}$. Therefore, $R_1[1] \not \in \per_{\mathcal{K}}(\p_{1,1})$, so $\p_{1,1} \not \in \per_L(\mathcal{K})$. For $n \geq 2$, recall that $\p_{n,1}$ is the kernel of the following composition:
\[\begin{tikzcd}
    F_n:\mathcal{K} \arrow[r, "\pi_n"] & \DTM_{\tilde{\mathcal{X}}_{\rho^{n-1}}}(F(Q_n);\F_2)^{\omega} \arrow[r] & \DTM_{\tilde{\mathcal{X}}_{\rho^{n-1}}}(F(Q_n);\F_2)[r_{n-2}^{-1}]^{\omega},
\end{tikzcd}\]
where $\pi_n$ is the tt-functor defined in \cite[Section 3]{schemeil2026motivic}, and the second arrow is the localisation that inverts the unique non-zero morphism $r_{n-2}: \tilde{\mathcal{X}}_{\rho^{n-1}} \longrightarrow \tilde{\mathcal{X}}_{\rho^{n-1}}(-d_{n-2})[-2d_{n-2}-1]$. Since $\pi_n(\tilde{M}_n)=0$, Vishik's diamond gives the following isomorphism in $\DTM_{\tilde{\mathcal{X}}_{\rho^{n-1}}}(F(Q_n);\F_2)[r_{n-2}^{-1}]^{\omega}$:
\[F_n(R_n[1]) \simeq \unit(d_{n-1})[2d_{n-1}+1].\]
However, $\unit(d_{n-1})[2d_{n-1}+1] \not \simeq \unit $ in $\DTM_{\tilde{\mathcal{X}}_{\rho^{n-1}}}(F(Q_n);\F_2)[r_{n-2}^{-1}]^{\omega}$. Thus, $R_n[1] \not \in \per_{\mathcal{K}}(\p_{n,1})$, so $\p_{n,1} \not \in \per_L(\mathcal{K})$. In particular, 
\[\per_L(\mathcal{K})=\{\p_{\infty,0}, \p_{\infty,1}, \p_{\infty,2}\},\]
and it is clear that this subset cannot be open since, by the description of the topology of $\Spc(\mathcal{K})$, any open neighbourhood of $\p_{\infty,0}$, must be cofinite.
\end{example}

To rectify this issue, we introduce the notion of a \emph{relative} period map: given a subgroup $H \subseteq \Pic(\mathcal{K})$, we define the following order-preserving function:
\[\begin{tikzcd}
    \per_\mathcal{K}^H:\Spc(\mathcal{K}) \arrow[r, "\per_\mathcal{K}"] & \mathrm{S}(\Pic(\mathcal{K})) \arrow[r, "- \cap H"] & \mathrm{S}(H).
\end{tikzcd}\]

\begin{prop}\label{prop:fgmeanscontinuous}
    Let $H \subseteq \Pic(\mathcal{K})$ be a finitely-generated subgroup. Then the relative period map
    \[\begin{tikzcd}
        \per_\mathcal{K}^H:\Spc(\mathcal{K}) \arrow[r] & \mathrm{S}(H)
    \end{tikzcd}\]
    is continuous when $\mathrm{S}(H)$ is equipped with the Alexandrov topology.
\end{prop}
\begin{proof}
    For a subgroup $L \subseteq H$, let us denote $U_H(L) \coloneqq U(L) \cap \mathrm{S}(H)$. We must show 
    \[(\per_\mathcal{K}^H)^{-1}(U_H(L)) \subseteq \Spc(\mathcal{K})\] 
    is open for all $L \subseteq H$. Since $H$ is finitely-generated, so is $L$, so we may write
    \[L= \langle a_1, \dots, a_n \rangle\]
    for some $a_1,\ldots,a_n \in \Pic(\mathcal{K})$. Suppose $\mathcal{P} \in \Spc(\mathcal{K})$ is such that $\mathcal{P} \in \per_{\mathcal{K}}^{-1}(U(L))$. Then there exists objects $x_1,\dots,x_n \in \mathcal{K}$ and morphisms $\alpha_i:x_i \longrightarrow \mathbb{1}$, $\beta_i: x_i \longrightarrow a_i$ for $i=1,\dots,n$ such that the fractions
    \[\begin{tikzcd}
        \unit & \arrow[l, swap, "\alpha_i"] x_i \arrow[r, "\beta_i"] & a_i
    \end{tikzcd}\]
    witness the isomorphisms $\mathbb{1} \simeq a_i \in \mathcal{K}/\mathcal{P}$. In particular, for $i=1,\ldots,n$ we have
    \[\cone(\alpha_i),\cone(\beta_i) \in \mathcal{P}.\]
    This gives us
    \[\mathcal{P} \in \bigcap_{i=1,\dots,n} (U(\cone(\alpha_i)) \cap U(\cone(\beta_i))),\]
    and the right-hand side is an open neighbourhood of $\mathcal{P}$ in $\Spc(\mathcal{K})$ contained in $(\per_\mathcal{K}^H)^{-1}(U_H(L))$. This completes the proof.
\end{proof}
\begin{remark}
    Observe that the subgroup $L$ of \Cref{ex: failure-continuity-quadrics} is not finitely-generated. Take caution, however, as there are examples where the Picard group of the category is not finitely generated, yet the period map is still continuous.
\end{remark}

\begin{definition}\label{def:limittopology}
    Define the \emph{limit topology} on $\mathrm{S}(\Pic(\mathcal{K}))$ to be the one determined by the basis of opens
    \[\{U(H) \subseteq \mathrm{S}(\Pic(\mathcal{K})) \mid H \text{ finitely-generated subgroup}\}.\]
\end{definition}

\begin{prop}
    Equipped with the limit topology, the period map
    \[\begin{tikzcd}
        \per_\mathcal{K}:\Spc(\mathcal{K}) \arrow[r] & \mathrm{S}(\Pic(\mathcal{K}))
    \end{tikzcd}\]
    is continuous.
\end{prop}
\begin{proof}
    This is straightforward given \Cref{prop:fgmeanscontinuous}; we omit the details.
\end{proof}
\begin{remark}
    There is a more conceptual way to view the limit topology. Indeed, denote by $\mathrm{S}^{\operatorname{fg}}(\Pic(\mathcal{K})) \subseteq \mathrm{S}(\Pic(\mathcal{K}))$ the subset of $\mathrm{S}(\Pic(\mathcal{K}))$ consisting of those finitely-generated subgroups. Restriction gives rise to a cofiltered system
    \[\{S(H)\}_{H \in \mathrm{S}^{fg}(\Pic(\mathcal{K}))}\]
    which yields a canonical map of sets
    \[\begin{tikzcd}
        \varphi:\mathrm{S}(\Pic(\mathcal{K})) \arrow[r] & \underset{H \in \mathrm{S}^{\operatorname{fg}}(\Pic(\mathcal{K}))}{\lim}\mathrm{S}(H); \ L \longmapsto (L \cap H)_H.
    \end{tikzcd}\]
    One can verify that $\varphi$ is an isomorphism of sets, and so we may define a topology on $\mathrm{S}(\Pic(\mathcal{K}))$ via declaring $\varphi$ to be a homeomorphism to the space
    \[\underset{H \in \mathrm{S}^{\operatorname{fg}}(\Pic(\mathcal{K}))}{\lim}\mathrm{S}(H) \in \Top\]
    where each $S(H)$ has its Alexandrov topology (recall the forgetful functor from $\Top$ to $\Set$ preserves limits). This is the weakest topology on $\mathrm{S}(\Pic(\mathcal{K}))$ such that the period map is compatible with the relative period maps for finitely generated subgroups, and one can show that this topology coincides with that of \Cref{def:limittopology}.\footnote{This justifies the terminology for the limit topology.}
\end{remark}

We now turn our attention to motivic examples and applications. 
\begin{prop}\label{prop:Pic(DTM)}
    Let $k,\F$ be fields where the exponential characteristic of $k$ is nonzero in $\F$. The Tate motives $\unit(i)[j]$ determine an isomorphism between the Picard group $\Pic(\DTM_\gm(k;\F))$ and $\Z^2$:
    \[\begin{tikzcd}
        \Z^2 \arrow[r, "\sim"] & \Pic(\DTM_\gm(k;\F)); \, (a,b) \longmapsto [\unit(a)[b]].
    \end{tikzcd}\]
\end{prop}
\begin{proof}
    The Chow weight structure on $\DTM_\gm(k;\F)$ determines a tt-functor
    \[\begin{tikzcd}
        t:\DTM_\gm(k;\F) \arrow[r] & \K^b(\mathsf{grvect}_\F)
    \end{tikzcd}\]
    from $\DTM_\gm(k;\F)$ to the bounded homotopy category of finite-dimensional graded $\F$-vector spaces: it sends the Tate motive $\unit(i)[j]$ to the graded vector space $\F$ in graded degree $i$, concentrated in homological degree $j-2i$. The only invertibles of $\K^b(\mathsf{grvect}_\F)$ are the $1$-dimensional vector spaces concentrated in a single $(\text{graded},\text{homological})$-bidegree, which are precisely the images of $\unit(i)[j]$ under $t$. By \cite[Lemma~20]{Bachmann:Picard}, $t$ is $\Pic$-injective, which proves the claim.
\end{proof}

The above proposition tells us, for Tate motives, generalised periodicity is equivalent to biperiodicity as in \Cref{example:biperiodicity}. In the following examples we completely determine the biperiodicities of the points of $\Spc(\DTM_\gm(k;\F_p))$.

\begin{example}\label{example:periodicityofTatesforpodd}
    Let $p$ be an odd prime and consider $\mathcal{K} \coloneqq \DTM_\gm(k;\F_p)$: its spectrum is the Sierpiński space, see \Cref{cor:explicit-tate-generators-topology} and \Cref{cor:tate-spectrum-no-root-unity}. The closed point has no non-trivial periods, since the category is local and so the associated Verdier localisation for this point is the identity functor. Therefore, we focus on the generic (étale) point $\mathcal{P}_{\acute{e}t}$. From the description of $\Pic(\DTM_\gm(k;\F_p))$ from \Cref{prop:Pic(DTM)}, we see that any period of $\mathcal{P}_{\acute{e}t}$ is first and foremost of the form $\unit(a)[b]$ for some $a,b \in \Z$. Recalling the definition of $\beta \in H_\mathcal{M}^{0,d}(k;\F_p)$ from \Cref{prop:hSpec-motcohom-no-root-unity} (where $d=[k(\zeta_p):k]$ is the degree of the extension of $k$ formed by adjoining to it a primitive $p^{th}$ root of unity $\zeta_p$; when $d=1$ we have $\beta=\tau$), we know
    \[\cone(\beta) \in \mathcal{P}_{\acute{e}t}\]
    so that $\unit(d) \in \per_\mathcal{K}(\mathcal{P}_{\acute{e}t})$. Moreover, by construction, étale realisation factors through the Verdier quotient of $\mathcal{K}$ by $\mathcal{P}_{\acute{e}t}$:
    \[\begin{tikzcd}
        \Re_{\acute{e}t}:\DTM_\gm(k;\F_p) \arrow[r] & \DTM_\gm(k;\F_p)/\mathcal{P}_{\acute{e}t} \arrow[r] & \DTM_\et^\vee(k;\F_p).
    \end{tikzcd}\]
    So, if $\unit(a)[b] \in \per_\mathcal{K}(\mathcal{P}_{\acute{e}t})$ is a period of $\mathcal{P}_{\acute{e}t}$, \Cref{lemma:propagationforgeneralisedperiodicity} informs us of an isomorphism
    \[\mu_p^{\otimes a}(k^s)[b]=\Re_{\acute{e}t}(\unit(a)[b]) \simeq \Re_{\acute{e}t}(\unit)=\unit\]
    in $\DM_\et^\vee(k;\F_p)$. But $\mu_p^{\otimes a}(k^s)[b]$ is trivial if and only if $b=0$ and $a$ is a multiple of $d$. It follows
    \[\per_\mathcal{K}(\mathcal{P}_{\acute{e}t})=\ideal{\unit(d)}.\]
\end{example}

\begin{example}\label{example:periodicityofTatesforpeven}
    Consider the case where $p=2$, and let $\mathcal{K} \coloneqq \DTM_\gm(k;\F_2)$. In the case $k$ is not formally real, the argument of \Cref{example:periodicityofTatesforpodd} carries over verbatim to tell us the periodicity of the points of $\Spc(\DTM_\gm(k;\F_2))$. Also, the same argument applies to the étale point for any $k$ and any prime $p$.
    
    Suppose now $k$ is formally real. For an ordering $\alpha$ on $k$, the point $\mathcal{P}^{\alpha} \in \Spc(\DTM_\gm(k;\F_2))$ has periodicity $\ideal{\unit(1)[1]} \subseteq \Z^2 \simeq \Pic(\DTM_\gm(k;\F_2))$. Indeed, we know the generators of this prime, and since under the generalised comparison map
    \[\begin{tikzcd}
        \varrho^h:\Spc(\DTM_\gm(k;\F_2)) \arrow[r] & \Spec^h(K^M_*(k)/2[\tau])
    \end{tikzcd}\]
    $\varrho^h(\mathcal{P}^{\alpha})$ is not the closed point (i.e. the irrelevant ideal) and contains $\tau$, we know there exists some $f:\unit \rightarrow \unit(1)[1]$, corresponding to some pure symbol $\{a\} \in K^M_1(k)/2$, for which $\cone(f) \in \mathcal{P}^{\alpha}$. It immediately follows $\unit(1)[1]$ is a period of $\mathcal{P}^{\alpha}$.

    The existence of another period which is not a $\otimes$-multiple of $\unit(1)[1]$ is forbidden, due to the existence of the tt-functor
    \[\begin{tikzcd}
        \DTM_\gm(k;\F_2) \arrow[r] & \DTM_\gm(k;\F_2)/\mathcal{P}^{\alpha} \arrow[r, "\operatorname{Re}_\alpha"] & \D^b(\F_2)
    \end{tikzcd}\]
    landing in a non-periodic category (employing a similar argument to \Cref{example:periodicityofTatesforpodd}).

    The periodicity of any one of the points $\mathcal{P}^\alpha_{\acute{e}t}$ is $\Z^2$. Indeed, propagation of periodicity (\Cref{lemma:propagationforgeneralisedperiodicity}) tells us both $\unit(1)[1]$ and $\unit(1)$ are periods: these generate the Picard group $\Z^2$, so exhaust all periods.
\end{example}
\begin{remark}\label{remark:noGallauerperiodicity}
    Note that none of the points of $\mathcal{P}^{\alpha},\mathcal{P}_{\acute{e}t} \in \Spc(\DTM_\gm(k;\F_p))$ are periodic in the sense of \Cref{def:Gallauerperiodicity}: for each such $\mathcal{P}$ we have $\per_{\DTM_\gm(k;\F_p)}^{\ideal{\unit[1]}}(\mathcal{P})=0$.
\end{remark}

\begin{cor}
    Generalised periodicity distinguishes points of $\Spc(\DTM_\gm(k;\F_p))$ except when $p=2$ and $k$ admits multiple orderings.
\end{cor}
\begin{proof}
If $p$ is odd or $k$ is not formally real, the closed point has trivial periodicity and the nonclosed point has periodicity $\ideal{\unit(d)}$, where $d=[k(\zeta_p):k]$, by \Cref{example:periodicityofTatesforpeven}; in particular, generalised periodicity distinguishes the points of the spectrum.

    In the case $k$ is formally real and $p=2$, if $k$ admits a unique ordering then the periodicities of the points $\ideal{0},\mathcal{P}^\alpha,\mathcal{P}_{\acute{e}t},\mathcal{P}^\alpha_{\acute{e}t}$ are
    \[\{\unit\}, \ideal{\unit(1)[1]}, \ideal{\unit(1)},\Z^2\]
    respectively, by \Cref{example:periodicityofTatesforpeven}. All these periodicities are distinct, so we are able to distinguish points by their periodicity in this case. If $k$ admits at least two orderings, say $\alpha,\beta$, then the periodicities of $\mathcal{P}^{\alpha}$ and $\mathcal{P}^{\beta}$ (as well as the periodicities of $\mathcal{P}^\alpha_{\acute{e}t}$ and $\mathcal{P}^\beta_{\acute{e}t}$) agree, again by \Cref{example:periodicityofTatesforpeven}. This proves the statement.
\end{proof}

\begin{remark}
    It is interesting to note that the generalised periodicities of points
    \[\mathcal{P} \in \Spc(\DTM_\gm(k;\F_p))\]
    are related to their depths in the spectrum. Indeed, the depth of $\mathcal{P} \in \Spc(\DTM_\gm(k;\F_p))$ coincides with the rank $\operatorname{rk}_\Z(\per_\DTM(\mathcal{P}))$ of the periodicity of $\mathcal{P}$, as can be seen from \Cref{example:periodicityofTatesforpodd} and \Cref{example:periodicityofTatesforpeven}.
\end{remark}

We return to the example of motives of quadrics over real closed fields. In this case, we shall see that generalised periodicity distinguishes every point of the spectrum. 

\begin{prop}\label{prop:periodisinjectiveforrealquadrics}
Let $F$ be a real closed field, $p$ a prime, and $\mathcal{K}=\DQM_\gm(F;\F_p)$. The generalised period map
\[\begin{tikzcd}
    \per_{\mathcal{K}}:\Spc(\mathcal{K}) \arrow[r] & \mathrm{S}(\Pic(\mathcal{K}))
\end{tikzcd}\]
is injective.
\end{prop}
\begin{proof}
When $p$ is an odd prime, $\Spc(\mathcal{K})$ is the two-point Sierpiński space: the periodicity of the closed point is $0$, whereas the generic point, being the kernel of étale realisation, contains $\unit(2)$ as a period. For $p=2$, recall from \Cref{ex: failure-continuity-quadrics} the description of $\Spc(\mathcal{K})$. Keeping the same notation of the example, we know that $M_n \in \p_{m,j}$ for every $1 \leq n < m \leq \infty$ and $j=0,1,2$. In particular, Vishik's diamond gives us that $R_n[1]$ is a period of $\p_{m,j}$. However, as shown in \Cref{ex: failure-continuity-quadrics}, $R_n[1]$ is not a period of either $\p_{n,0}$ or $\p_{n,1}$. This shows that if $n \neq m$, then $\p_{n,i}$ and $\p_{m,j}$ have distinct periodicities for $i,j=0,1,2$. For $n=1$, $\unit(1)$ is a period of $\p_{1,1}$, since this point is the kernel of étale realisation. It is clearly not a period of $\p_{1,0}$, since this point is the kernel of base change to $\overline{F}$. For $2 \leq n < \infty$, we know from \cite{schemeil2026motivic} that $M_{n-1}, \widetilde{M}_{n-1} \in \p_{n,1}$; in particular, Vishik's diamond gives us that $\unit(d_{n-2})[2d_{n-2}+1]$ is a period of $\p_{n,1}$. However, it is clearly not a period of $\p_{n,0}$: indeed, this tt-prime is the kernel of the functor
\[\begin{tikzcd}
    \pi_n:\mathcal{K} \arrow[r] & \DTM_{\tilde{\mathcal{X}}_{\rho^{n-1}}}(F(Q_n);\F_2)^{\omega}
\end{tikzcd}\]
from \cite[Section 3]{schemeil2026motivic}, and the latter category admits a Chow weight structure \cite[Lemma 4.2]{schemeil2026motivic}, in particular there are no morphisms between the unit and positive shifts of objects in its heart. Finally, for $n=\infty$, we know that $\unit[1]$ is a period of $\p_{\infty, 2}$ since $\cone(\rho), \cone(\tau) \in \p_{\infty,2}$. However, $\unit[1]$ is not a period of $\p_{\infty,1}$ since this tt-prime is the kernel of the localisation
\[\begin{tikzcd}
    \mathcal{K} \arrow[r] & (\mathcal{K}[\rho^{-1}])^\natural \simeq \DTM(F;\F_2)[\rho^{-1}]^\omega,
\end{tikzcd}\]
and we already know that the latter category is not periodic. Finally, $\unit(1)[1]$ is a period of $\p_{\infty,1}$ since this tt-prime contains $\cone(\rho)$, however $\unit(1)[1]$ is not a period of $\p_{\infty,0}$: indeed, this tt-prime is the kernel of the tt-functor
\[\begin{tikzcd}
    \pi_{\infty}:\mathcal{K} \arrow[r] & \DTM_{\Upsilon^{iso}}(F;\F_2)^\omega,
\end{tikzcd}\]
and by the computation of the bigraded endomorphism ring of the unit in the target category \cite[Lemma 4.16]{schemeil2026motivic}, we can immediately see that $\unit \not \simeq \unit(1)[1]$. Thus, we have verified that generalised periodicity distinguishes all points, proving the proposition.
\end{proof}
\printbibliography[title=References]
\end{document}